\documentclass[a4paper,11pt,reqno]{amsart}
\usepackage{graphicx}

\usepackage{color}

\usepackage{amsmath,amstext,amssymb,amsopn,amsthm}
\usepackage{url}

\usepackage[margin=30mm]{geometry}
\usepackage{eucal,mathrsfs,dsfont}
\usepackage{soul}

\newtheorem{theorem}{Theorem}[section]
\newtheorem{corollary}[theorem]{Corollary}
\newtheorem{lemma}[theorem]{Lemma}
\newtheorem{proposition}[theorem]{Proposition}

\theoremstyle{definition}
\newtheorem{conjecture}[theorem]{Conjecture}
\newtheorem{definition}[theorem]{Definition}

\theoremstyle{remark}
\newtheorem{remark}[theorem]{Remark}

\definecolor{mr}{rgb}{0.1,0.2,0.7}

\newcommand{\eps}{\varepsilon}

\newcommand{\ox}{\overline{x}}

\newcommand{\R}{\mathds{R}}

\newcommand{\N}{{\mathds{N}}}

\newcommand{\RR}{\mathrm{I\kern-0.20emR}}
\newcommand{\D}{\mathrm{d}\kern0.2pt}

\newcommand{\vp}{{\varphi}}

\DeclareMathOperator{\dist}{dist}

\title[On concavity of solution of the equation $(-\Delta)^{1/2} \vp = 1$]{On concavity of solution of Dirichlet problem for the equation $(-\Delta)^{1/2} \vp = 1$ in a convex planar region}
\author[T. Kulczycki]{Tadeusz Kulczycki}
\thanks{The research was supported in part by NCN grant no. 2011/03/B/ST1/00423.}
\address{Institute of Mathematics and Computer Science, Wroc{\l}aw University of Technology, Wyb. Wyspia{\'n}skiego 27, 50-370 Wroc{\l}aw, Poland.}
\email{Tadeusz.Kulczycki@pwr.edu.pl}

\begin{document}
\begin{abstract}
For a sufficiently regular open bounded set $D \subset \R^2$ let us consider the equation $(-\Delta)^{1/2} \vp(x) = 1$, $x \in D$ with the Dirichlet exterior condition $\vp(x) = 0$, $x \in D^c$. $\vp$ is the expected value of the first exit time from $D$ of the Cauchy process in $\R^2$. We prove that if $D \subset \R^2$ is a convex bounded domain then $\vp$ is concave on $D$. To show it we study the Hessian matrix of the harmonic extension of $\vp$. The key idea of the proof is based on a deep result of Hans Lewy concerning determinants of Hessian matrices of harmonic functions.
\end{abstract}

\maketitle

\section{Introduction}

Let $D \subset \R^2$ be an open bounded set which satisfies a uniform exterior cone condition on $\partial D$ and let us consider the following Dirichlet problem for the square root of the Laplacian 
\begin{eqnarray}
\label{maineq1}
(-\Delta)^{1/2} \vp(x) &=& 1, \quad \quad x \in D,
\\
\label{maineq2}
\vp(x) &=& 0, \quad \quad x \in D^c,
\end{eqnarray}
where we understand that $\vp$ is a continuous function on $\R^2$.
$(-\Delta)^{1/2}$ in $\R^2$ is given by $(-\Delta)^{1/2} f(x) = \frac{1}{2 \pi} \lim_{\eps \to 0^+} \int_{|y - x| > \eps} \frac{f(x) - f(y)}{|y - x|^3} \, dy$, whenever the limit exists.

It is well known that (\ref{maineq1}-\ref{maineq2}) has a unique solution. It has a natural probabilistic interpretation. Let $X_t$ be the Cauchy process in $\R^2$ (that is a symmetric $\alpha$-stable process in $\R^2$ with $\alpha = 1$) with a transition density $p_t(x) = \frac{1}{2 \pi} t (t^2 + |x|^2)^{-3/2}$ and let $\tau_D = \inf\{t \ge 0: \, X_t \notin D\}$ be the first exit time of $X_t$ from $D$. Then $\vp(x) = E^x(\tau_D)$ \cite{G1961}, $x \in \R^2$, where $E^x$ is the expected value of the process $X_t$ starting from $x$. The function $E^x(\tau_D)$ plays an important role in the potential theory of symmetric stable processes (see e.g. \cite{BKK2008}, \cite{book2009}, \cite{CS2010}).

About 10 years ago R. Ba{\~n}uelos posed a problem of $p$-concavity of $E^x(\tau_D)$ for symmetric $\alpha$-stable processes. The problem was inspired by a beautiful result of Ch. Borell about $1/2$-concavity of $E^x(\tau_D)$ for the Brownian motion.

The main result of this paper is the following theorem. It solves the problem posed by R. Ba{\~n}uelos for the Cauchy process in $\R^2$.

\begin{theorem}
\label{mainthm}
If $D \subset \R^2$ is a bounded convex domain then the solution of (\ref{maineq1}-\ref{maineq2}) is concave on $D$.
\end{theorem}

To the best of author's knowledge this is the first result concerning concavity of solutions of equations for fractional Laplacians on general convex domains. There is a recent interesting paper of R. Ba{\~n}uelos and R. D. DeBlassie \cite{BB2013} in which the first eigenfunction of the Dirichlet eigenvalue problem for fractional Laplacians on Lipschitz domains is studied but in that paper superharmonicity and not concavity of the first eigenfunction is proved (similar results were also obtained by M. Ka{\ss}mann and L. Silvestre \cite{KS}). In \cite{BKM2006} concavity of the first eigenfunction for fractional Laplacians was studied but \cite{BKM2006} concerns boxes and not general convex domains.

Now let $D \subset \R^d$, $d \ge 1$ be an open bounded set which satisfies a uniform exterior cone condition on $\partial D$, $\alpha \in (0,2]$ and let us consider a more general Dirichlet problem for the fractional Laplacian 
\begin{eqnarray}
\label{aeq1}
(-\Delta)^{\alpha/2} \vp(x) &=& 1, \quad \quad x \in D,
\\
\label{aeq2}
\vp(x) &=& 0, \quad \quad x \in D^c,
\end{eqnarray}
where we understand that $\vp$ is a continuous function on $\R^d$.
$(-\Delta)^{\alpha/2}$ in $\R^d$ for $\alpha \in (0,2)$ is given by $(-\Delta)^{\alpha/2} f(x) = \mathcal{A}_{d,-\alpha}
 \lim_{\eps \to 0^+} \int_{|y - x| > \eps} \frac{f(x) - f(y)}{|y - x|^{d + \alpha}} \, dy$, whenever the limit exists, $\mathcal{A}_{d,-\alpha} = 2^{\alpha} \Gamma((d + \alpha)/2)/(\pi^{d/2} |\Gamma(-\alpha/2)|)$. For $\alpha = 2$ the operator $(-\Delta)^{\alpha/2}$ is simply $-\Delta$.

It is well known that (\ref{aeq1}-\ref{aeq2}) has a unique solution. It is the expected value of the first exit time from $D$ of the symmetric $\alpha$-stable process in $\R^d$.

\begin{remark}
\label{Remark1}
For $\alpha  = 2$ i.e. for the Laplacian, it is well known that if $D \subset \R^d$ is a bounded convex domain then the solution of (\ref{aeq1}-\ref{aeq2}) is $1/2$-concave, that is $\sqrt{\vp}$ is concave. This was proved for $d = 2$ in 1969 by L. Makar-Limanov \cite{ML1971}. For $d \ge 3$ it was proved in 1983 by Ch. Borell \cite{B1985} and independently by A. Kennington \cite{Ke1985}, \cite{Ke1984} using ideas of N. Korevaar \cite{K1983}.
\end{remark}

\begin{remark}
\label{Remark2}
Let $\alpha \in (0,2]$ and $\vp$ be a solution of (\ref{aeq1}-\ref{aeq2}) for $D = B(0,r) \subset \R^d$, $d \ge 1$ a ball with centre $0$ and radius $r > 0$. Then $\vp$ is given by an explicit formula \cite{G1961} (see also \cite{KP1950}, \cite{E1959}) $\vp(x) = C_B (r^2 - |x|^2)^{\alpha/2}$, $x \in B(0,r)$, where $C_B = \Gamma(d/2)(2^{\alpha} \Gamma(1 + \alpha/2) \Gamma(d/2 + \alpha/2))^{-1}$. In particular $\vp$ is concave on $B(0,r)$.
\end{remark}

\begin{remark}
\label{Remark3}
For any $\alpha \in (1,2)$ and $d \ge 2$ there exists a bounded convex domain $D \subset \R^d$ (a sufficiently narrow bounded cone) such that $\vp$ is not concave on $D$. The justification of this statement is in Section 7. In particular, this implies that the assertion of Theorem \ref{mainthm}  is not true for the problem (\ref{aeq1}-\ref{aeq2}) for $\alpha \in (1,2)$.
\end{remark}

For general $\alpha \in (0,2)$ and $d \ge 2$ we have the following regularity result. 
\begin{theorem}
\label{generalthm}
Let $\alpha \in (0,2)$, $d \ge 2$ and let $\vp$ be a solution of (\ref{aeq1}-\ref{aeq2}). If $D \subset \R^d$ is a bounded convex domain then we have

a) for any $x_0 \in \partial D$, $x \in D$, $\lambda \in (0,1)$
$$
\vp(\lambda x + (1 - \lambda) x_0) \ge \lambda^{\alpha} \vp(x),
$$

b) for any $x, y \in D$, $\lambda \in (0,1)$
$$
\vp(\lambda x + (1 - \lambda) y) 
\ge \frac{1}{2} \left(\lambda^{\alpha} \vp(x) + (1 - \lambda)^{\alpha} \vp(y)\right).
$$
\end{theorem}
The proof of this theorem is in Section 7. It is based on one tricky observation and is much easier than the proof of Theorem \ref{mainthm}. Clearly, Theorem \ref{generalthm} does not imply $p$-concavity of $\vp$ for any $p \in [-\infty,1]$. Some conjectures concerning $p$-concavity of solutions of (\ref{aeq1}-\ref{aeq2}) are presented in Section 7.

Below we present the idea of the proof of Theorem \ref{mainthm}. The proof is in the spirit of papers by L. Caffarelli, A. Friedman \cite{CF1985} and N. Korevaar, J. Lewis \cite{KL1987} in which they study geometric properties of solutions of some PDEs using the constant rank theorem and the method of continuity. In the proof of Theorem \ref{mainthm} the role of the constant rank theorem plays the following result of Hans Lewy from 1968.
\begin{theorem}[Hans Lewy, \cite{L1968}]
\label{HL}
Let $u(x_1,x_2,x_3)$ be real and harmonic in a domain $\Omega$ of $\R^3$. Suppose the Hessian $H(u)$ i.e. the determinant of the matrix of second derivatives $((\partial^2 u/\partial x_i \partial x_j))$ vanishes at a point $x_0 \in \Omega$ without vanishing identically in $\Omega$. 
Then $H(u)$ assumes both positive and negative values near $x_0$.
\end{theorem}
The use of this result is the key element of the proof of Theorem \ref{mainthm}.
Note that it is known the generalization of Theorem \ref{HL} to higher dimensions. This generalization is a remarkable achievement obtained by S. Gleason and T. Wolff in 1991 (see Theorem 1 in \cite{GW1991}). It gives 
some hope that it is possible to extend Theorem \ref{mainthm} to higher dimensions, see Conjecture \ref{alpha1} in Section 7.

Let us come back to presenting the idea of the proof of Theorem \ref{mainthm}. We first show this result for a sufficiently smooth bounded convex domain $D \subset B(0,1) \subset \R^2$, which boundary has a strictly positive curvature. Let us consider the harmonic extension $u$ of $\vp$. Namely, let
\begin{equation}
\label{Kkernel}
K(x) = C_K \frac{x_3}{(x_1^2 + x_2^2 + x_3^2)^{3/2}}, \quad \quad x \in \R_+^3,
\end{equation}
where $C_K = 1/(2 \pi)$, $\R_+^3 = \{x = (x_1,x_2,x_3) \in \R^3: \, x_3 > 0\}$. Put $u(x_1,x_2,0) = \vp(x_1,x_2)$, $(x_1,x_2) \in \R^2$ and
\begin{equation}
\label{ext1}
u(x_1,x_2,x_3) = \int_D K(x_1-y_1,x_2-y_2,x_3) \vp(y_1,y_2) \, dy_1 \, dy_2,
\quad \quad (x_1,x_2,x_3) \in \R_+^3.
\end{equation}
Note that $K(x_1-y_1,x_2-y_2,x_3)$ is the Poisson kernel of $\R_+^3$ for points $x = (x_1,x_2,x_3) \in \R_+^3$ and $(y_1,y_2,0) \in \partial \R_+^3$. By $f_i$ we denote $\frac{\partial f}{\partial x_i}$, by $f_{ij}$ we denote $\frac{\partial^2 f}{\partial x_i \partial x_j}$. It is well known that $u_3(x_1,x_2,0) = - (-\Delta)^{1/2} \vp(x_1,x_2)$, $(x_1,x_2) \in D$ so $u$ satisfies
\begin{eqnarray}
\label{harmonic}
\Delta u(x) &=& 0, \quad \quad x \in \R_+^3,\\
\label{Steklov}
u_3(x) &=& -1, \quad \quad x \in D \times \{0\},\\
\label{Dirichlet}
u(x) &=& 0, \quad \quad x \in D^c \times \{0\},
\end{eqnarray}
where $\Delta u = u_{11} + u_{22} + u_{33}$.

The idea of studying equations for fractional Laplacians via harmonic extensions is well known. It was used for the first time by F. Spitzer in \cite{S1958}. Harmonic extensions were used by many authors e.g. by  S. A. Molchanov, E. Ostrovskii \cite{MO1969}, R. D. DeBlassie \cite{B1990}, P. Mendez-Hernandez \cite{M2002}, R. Ba{\~n}uelos, T. Kulczycki \cite{BK2004}, A. El Hajj, H. Ibrahim, R. Monneau \cite{EIM2009}, L. Caffarelli, L. Silvestre \cite{CS2007}.

In the next step of the proof we extend $u$ to $\R_-^3 = \{x = (x_1,x_2,x_3) \in \R^3: \, x_3 < 0\}$ by putting
\begin{equation}
\label{ext2}
u(x_1,x_2,x_3) = u(x_1,x_2,-x_3) - 2 x_3, \quad \quad (x_1,x_2,x_3) \in \R_-^3.
\end{equation} 
Note that $u$ is continuous on $\R^3$ and for $(x_1,x_2) \in D$ it satisfies
\begin{eqnarray*}
u_{3^-}(x_1,x_2,0) &=& 
\lim_{h \to 0^-} \frac{u(x_1,x_2,h) - u(x_1,x_2,0)}{h}\\
&=& 
\lim_{h \to 0^-} \frac{u(x_1,x_2,-h) -2h - u(x_1,x_2,0)}{h} 
= -1.
\end{eqnarray*}
By standard arguments it follows that $u$ is harmonic in $\R_+^3 \cup \R_-^3 \cup (D \times \{0\}) = \R^3 \setminus (D^c \times \{0\})$.

Since we need to consider different domains $D$ we change our notation $\vp$, $u$ to $\vp^{(D)}$, $u^{(D)}$. Let $H(u^{(D)})$ be the determinant of the Hessian matrix of $u^{(D)}$. Our next aim is to show that $H(u^{(D)})(x) > 0$ for any $x \in \R^3 \setminus (D^c \times \{0\})$. Note that (see Lemma \ref{lowerhalfspace}) $H(u^{(D)})(x_1,x_2,-x_3) = H(u^{(D)})(x_1,x_2,x_3)$ so it is sufficient to control $H(u^{(D)})(x)$ for $x \in \R_+^3 \cup (D \times \{0\})$. Now for technical reasons we need to add an auxiliary function to $u^{(D)}$. Namely, for any $\eps \ge 0$ we consider 
$v^{(\eps,D)}(x) = u^{(D)}(x) + \eps (-x_1^2/2 - x_2^2/2 + x_3^2)$. This is done to control $H(v^{(\eps,D)})(x)$ near $(\text{int} D^c) \times \{0\}$ because $H(u^{(D)})(x) \to 0$ when $x$ approaches $(\text{int} D^c) \times \{0\}$. Note that $v^{(\eps,D)}$ is harmonic in $\R^3 \setminus (D^c \times \{0\})$. Note also that (see Lemma \ref{lowerhalfspace1}) $H(v^{(\eps,D)})(x_1,x_2,-x_3) = H(v^{(\eps,D)})(x_1,x_2,x_3)$ so it is sufficient to control $H(v^{(\eps,D)})(x)$ for $x \in \R_+^3 \cup (D \times \{0\})$.

Now, on the contrary, assume that there exists $x_0 \in \R^3 \setminus (D^c \times \{0\})$ such that we have $H(u^{(D)})(x_0) \le 0$. One can show that $H(u^{(D)})(x)$ is not identically zero in $\R^3 \setminus (D^c \times \{0\})$. If $H(u^{(D)})(x_0) = 0$ and $H(u^{(D)})(x) \ge 0$ for all $x \in \R^3 \setminus (D^c \times \{0\})$ then we get contradiction with Theorem \ref{HL}. So, we may assume that $H(u^{(D)})(x_0) < 0$. Then for sufficiently small $\eps > 0$ we have $H(v^{(\eps,D)})(x_0) < 0$. Recall that $D \subset B(0,1) \subset \R^2$. For $M \ge 10$ we consider the set $W(M,D) = \{x \in \R^3: \, x_1^2 + x_2^2 \le M, x_3 \in [-M,M]\} \setminus (D^c \times \{0\})$ (it is a large cylinder without $D^c \times \{0\}$). One can take large enough $M$ so that $x_0 \in W(M,D)$.

In the next step of the proof using direct formula of $\vp^{(B(0,1))}$ and also using some "tricks" we show that $H(u^{(B(0,1))})(x) > 0$ for any $x \in \R^3 \setminus (B^c(0,1) \times \{0\})$. This is done in Section 5. Later we show that for sufficiently large $M$ and small $\eps$ we have $H(v^{(\eps,B(0,1))})(x) > 0$ for $x \in W(M,B(0,1))$.

Then we use method of continuity (cf. \cite[page 20]{KL1987}, \cite{CF1985}). Namely, we deform $D$ to the ball $B(0,1)$. More precisely we consider the family of domains $\{D(t)\}_{t \in [0,1]}$ such that $D(0) = D$, $D(1) = B(0,1)$, all $D(t)$ are smooth bounded convex domains which boundaries have strictly positive curvature and $\partial D(t) \to \partial D(s)$ when $t \to s$ in the appropriate sense. One can fix (in the appropriate way) sufficiently large $M \ge 10$ and sufficiently small $\eps > 0$ so that for all domains $\{D(t)\}_{t \in [0,1]}$ we can control $H(v^{(\eps,D(t))})(x)$ for $x$ near the boundary of "cylinders" $W(M,D(t))$. For $x \in \partial W(M,D(t))$ such that $|x_3| = M$ or $x_1^2 + x_2^2 = M$ and $x_3 > 0$ ($x_3$ not too small) we have $H(v^{(\eps,D(t))})(x) \approx H(u^{(D(t))})(x) \approx H(K)(x) > 0$. The last inequality follows by a direct computation. Showing that $H(v^{(\eps,D(t))})(x) \approx H(u^{(D(t))})(x) > 0$ near $\partial D \times \{0\}$ is the most technical part of the proof and this is done in Sections 3, 4 and in the proof of Proposition \ref{vepsilon}. The fact that $H(v^{(\eps,D(t))})(x) > 0$ for $x$ near $D^c \times \{0\}$ (when $x$ is not too close to $\partial D \times \{0\}$) is rather easy and here is the place where the auxiliary function $\eps (-x_1^2/2 - x_2^2/2 + x_3^2)$ helps.

In fact, one can show that there exists $c > 0$ such that $H(v^{(\eps,D(t))})(x) \ge c > 0$ for all $t \in [0,1]$ and all $x$ near $\partial W(M,D(t))$.

Recall that $H(v^{(\eps,D(0))})(x_0) < 0$ for some $x_0 \in W(M,D(0))$ and $H(v^{(\eps,D(1))})(x) > 0$ for all $x \in W(M,D(1))$. Using the method of continuity one can show that there exists $\tilde{t} \in (0,1)$, $\tilde{D} = D(\tilde{t})$ and $\tilde{x} \in W(M,\tilde{D})$ such that $H(v^{(\eps,\tilde{D})})(\tilde{x}) = 0$ and $H(v^{(\eps,\tilde{D})})(x) \ge 0$ for all $x \in W(M,\tilde{D})$. Moreover one can show that $H(v^{(\eps,\tilde{D})})(x) > 0$ for $x$ near $\partial W(M,\tilde{D})$ so $x \in \text{int} (W(M,\tilde{D}))$. This gives contradiction with Theorem \ref{HL}. So we finally obtain 
\begin{equation}
\label{Hessian}
H(u^{(D)})(x) > 0, \quad \quad x \in \R^3 \setminus (D^c \times \{0\}).
\end{equation}
A closer look gives that in fact the Hessian matrix of $u$ has a constant signature $(1,2)$. It seems that this observation could help in studying the analogous problem in higher dimensions (cf. Conjecture \ref{alpha1} in Section 7 and Theorem 1 in \cite{GW1991}).

Let $(x_1,x_2) \in D$. By (\ref{Steklov}) we get $u_{13}^{(D)}(x_1,x_2,0) = 0$, $u_{23}^{(D)}(x_1,x_2,0) = 0$, $u_{33}^{(D)}(x_1,x_2,0) > 0$ (see Lemma \ref{onD}). Using this and (\ref{Hessian}) we obtain 
\begin{equation}
\label{u112212intro}
u_{11}^{(D)}(x_1,x_2,0) u_{22}^{(D)}(x_1,x_2,0) - (u_{12}^{(D)}(x_1,x_2,0))^2 > 0.
\end{equation}
We also have $u_{11}^{(D)}(x_1,x_2,0) + u_{22}^{(D)}(x_1,x_2,0) = - u_{33}^{(D)}(x_1,x_2,0) < 0$. This and (\ref{u112212intro}) implies $u_{11}^{(D)}(x_1,x_2,0) < 0$, $u_{22}^{(D)}(x_1,x_2,0) < 0$.

This gives that $\vp^{(D)}(x_1,x_2) = u^{(D)}(x_1,x_2,0)$ is strictly concave on $D$. Recall that we have assumed that $D$ is a sufficiently smooth bounded convex domain, $D \subset B(0,1)$ and $\partial D$ has a strictly positive curvature. The concavity of $\vp^{(D)}$ for arbitrary convex domains $D$ follows by approximation arguments and scaling.

Of course, this is only the sketch of the proof. In fact one has to be very carefull with the method of continuity. In particular one has to control $H(v^{(\eps,D(t))})(x)$ for $x$ near $\partial W(M,D(t))$ in a "uniform way" according to $t \in [0,1]$.

The paper is organized as follows. In Section 2 we present notation and collect some known facts needed in the rest of the paper. In Section 3 we estimate $\vp_{ij}^{(D)}$ near $\partial D$. 
Section 4 contains estimates of $u_{ij}^{(D)}$ near $\partial D \times \{0\}$. In Section 5 the harmonic extension for a ball is studied. Section 6 contains the proof of the main theorem. In Section 7 some extensions and conjectures are presented.

\section{Preliminaries}

For $x \in \R^d$ and $r > 0$ we let $B(x,r) = \{y \in \R^d: \, |y - x| < r\}$. By $a \wedge b$ we denote $\min(a,b)$ and by $a \vee b$ we denote $\max(a,b)$ for $a, b \in \R$. For $x \in \R^d$, $D \subset \R^d$ we put $\delta_D(x) = \dist(x,\partial D)$. For any $\psi: \R^d \to \R$ we denote $\psi_i(x) = \frac{\partial \psi}{\partial x_i}(x)$, $\psi_{ij}(x) = \frac{\partial^2 \psi}{\partial x_i \partial x_j}(x)$, $i, j \in \{1,\ldots,d\}$.
We put $\R_+^3 = \{(x_1,x_2,x_3) \in \R^3: \, x_3 > 0\}$, $\R_-^3 = \{(x_1,x_2,x_3) \in \R^3: \, x_3 < 0\}$. The definition of a uniform exterior cone condition may be found e.g. in \cite[page 195]{GT1977}.

Let us define a subclass of bounded, convex $C^{2,1}$ domains in $\R^2$ with strictly positive curvature, which will be suitable for our purposes.

\begin{definition}
\label{classF}
Let $C_1 > 0$, $R_1 > 0$, $\kappa_2 \ge \kappa_1 > 0$ and let us fix a Cartesian coordinate system $CS$ in $\R^2$. We say that a domain $D \subset \R^2$ belongs to the class $F(C_1,R_1,\kappa_1,\kappa_2)$ when

1. $D$ is convex. In $CS$ coordinates we have
$$
\{(y_1,y_2): \, y_1^2 + y_2^2 < R_1^2\} \subset D
\subset \{(y_1,y_2): \, y_1^2 + y_2^2 < 1\}.
$$ 

2. For any $x \in \partial D$ there exist a Cartesian coordinate system $CS_x$ with origin at $x$ obtained by translation and rotation of $CS$, there exist $R > 0$, $f: [-R,R] \to [0,\infty)$ ($R$, $f$ depend on $x$), such that $f \in C^{2,1}[-R,R]$, $f(0) = 0$, $f'(0) = 0$ and in $CS_x$ coordinates
$$
\{(y_1,y_2): \,  y_2 \in [-R,R], y_1 \in (f(y_2),R]\} 
= D \cap \{(y_1,y_2): \,  y_1 \in [-R,R], y_2 \in [-R,R]\}. 
$$

3. For any $y \in \partial D$ we have 
$$
\kappa_1 \le \kappa(y) \le \kappa_2,
$$
where $\kappa(y)$ denotes the curvature of $\partial D$ at $y$.

4. For any $y, z \in \partial D$ we have 
$$
|\kappa(y) - \kappa(z)| \le C_1 |y - z|.
$$

For brevity, we will often use notation $\Lambda = \{C_1,R_1,\kappa_1,\kappa_2\}$ and write $D \in F(\Lambda)$.
\end{definition}

Let $C_1 > 0$, $R_1 > 0$, $\kappa_2 \ge \kappa_1 > 0$ and put $\Lambda = \{C_1,R_1,\kappa_1,\kappa_2\}$. Let $D \in F(\Lambda)$. For any $y \in \partial D$ by $\vec{n}(y)$ we denote the normal inner unit vector at $y$ and by $\vec{T}(y)$ we denote the tangent unit vector at $y$ which agrees with negative (clockwise) orientation of $\partial D$. We put $e_1 = (1,0)$, $e_2 = (0,1)$.

It may be easily shown that there exists $\tilde{R} = \tilde{R}(\Lambda)$ such that for any $y \in D$, $\delta_D(y) \le \tilde{R}$ there exists a unique $y^* \in \partial D$ such that $|y - y^*| = \delta_D(y)$. For any $y \in D$ such that $\delta_D(y) \le \tilde{R}$ we define $\vec{n}(y) = \vec{n}(y^*)$, $\vec{T}(y) = \vec{T}(y^*)$. For any $\psi \in C^2(D)$, $y \in D$, $v_1(y), v_2(y) \in \R$ and $\vec{v}(y) = v_1(y) {e_1} + v_2(y) {e_2}$ we put 
$\frac{\partial \psi}{\partial \vec{v}}(y) = v_1(y) \psi_1(y) + v_2(y) \psi_2(y)$, (recall that $\psi_i(y) = \frac{\partial \psi}{\partial x_i}(y)$). Similarly, for any $w_1(y), w_2(y) \in \R$ and $\vec{w}(y) = w_1(y) {e_1} + w_2(y) {e_2}$ we put $\frac{\partial^2 \psi}{\partial \vec{v} \partial \vec{w}}(y) = v_1(y) w_1(y) \psi_{11}(y) + v_2(y) w_2(y) \psi_{22}(y) + 
(v_1(y) w_2(y) + v_2(y) w_1(y)) \psi_{12}(y)$.

\begin{figure}
\centering
\includegraphics[scale=0.7]{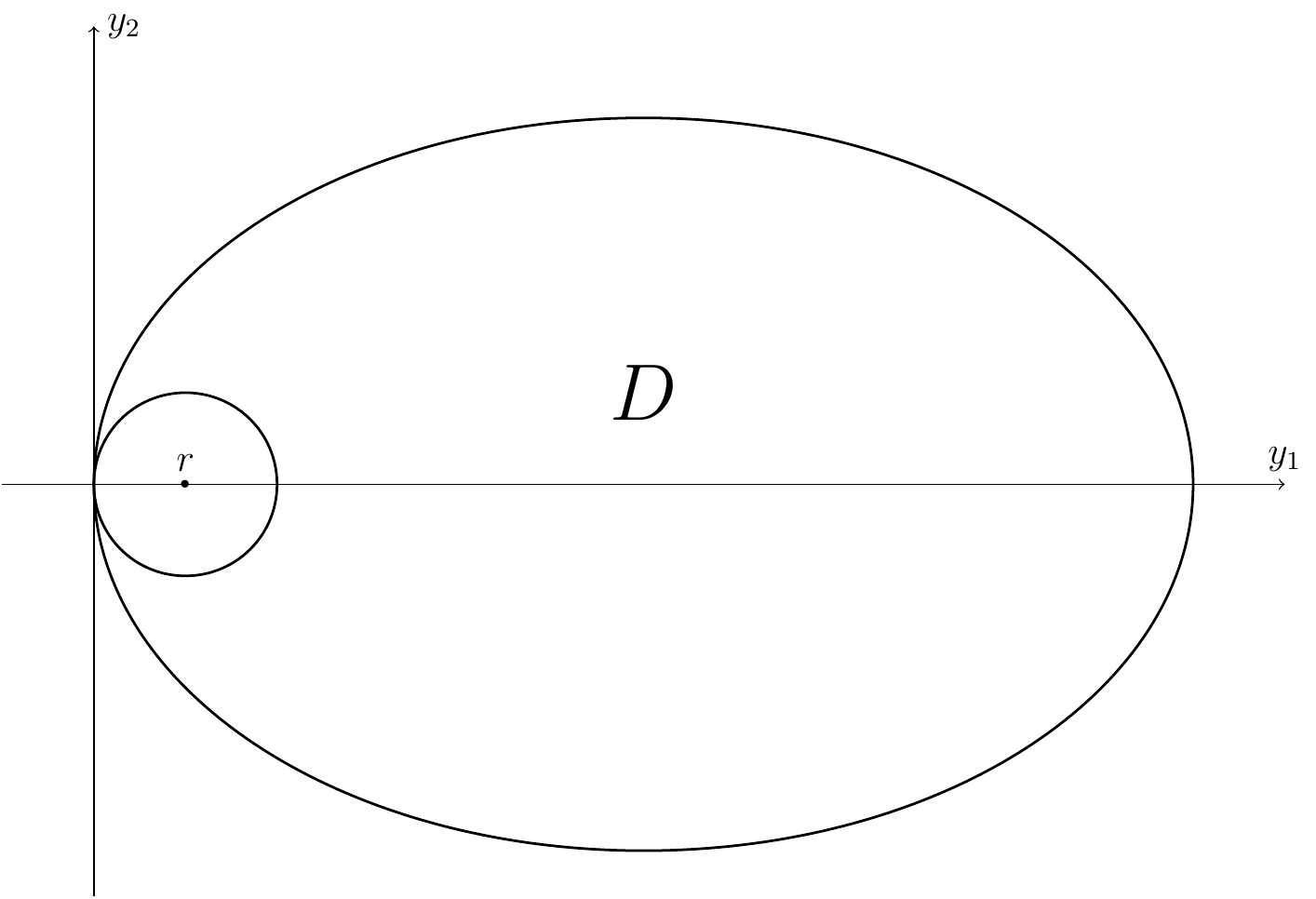}
\caption{}
\label{fig:1}
\end{figure}

\begin{lemma}
\label{ball}
Let $C_1 > 0$, $R_1 > 0$, $\kappa_2 \ge \kappa_1 > 0$ put $\Lambda = \{C_1,R_1,\kappa_1,\kappa_2\}$ and let us fix a Cartesian coordinate system $CS$ in $\R^2$. Fix $D \in F(\Lambda)$ and $x_0 \in \partial D$. Choose a new Cartesian coordinate system $CS_{x_0}$ with origin at $x_0$ obtained by translation and rotation of $CS$ such that the positive coordinate halflines $y_1$, $y_2$ are in the directions $\vec{n}(x_0)$, $\vec{T}(x_0)$ respectively.

From now on all points and vectors are in this new coordinate system $CS_{x_0}$, in particular $\vec{n}(0,0) = (1,0) = {e_1}$, $\vec{T}(0,0) = (0,1) = {e_2}$. For any $y \in \partial D$ define $\alpha(y) \in (-\pi,\pi]$ such that $\vec{T}(y) = \sin \alpha(y) {e_1} + \cos \alpha(y) {e_2}$ (this is an angle between ${e_2}$ and $\vec{T}(y)$).

There exists $r_0 = r_0(\Lambda) \le \tilde{R} \wedge (1/2)$, $c_1 = c_1(\Lambda)$, $c_2 = c_2(\Lambda)$, $c_3 = c_3(\Lambda)$, $c_4 = c_4(\Lambda)$, $c_5 = c_5(\Lambda)$, $c_6 = c_6(\Lambda)$, $f: [-r_0,r_0] \to [0,\infty)$ such that $f \in C^{2,1}[-r_0,r_0]$, $f(0) = 0$, $f'(0) = 0$, $c_4 r_0 \le 1/4$ and for any fixed $r \in (0,r_0]$ we have (see Figure 1)

1. $\{(y_1,y_2): \, (y_1 - r)^2 + y_2^2 < r^2\}  \subset D$,
$$
W := \{(y_1,y_2): \,  y_2 \in [-r,r], y_1 \in (f(y_2),r]\} 
= D \cap \{(y_1,y_2): \,  y_1 \in [-r,r], y_2 \in [-r,r]\}.
$$

2. For any $y \in W$ we have $\alpha(y) \in [-\pi/4,\pi/4]$ and 
$$
c_1 |y_2| \le |\sin \alpha(y)| \le c_2 |y_2|.
$$

3. For any $y_2 \in [-r,r]$ we have 
$$
c_3 y_2^2 \le f(y_2) \le c_4 y_2^2.
$$ 

4. For any $y \in W$ we have ${e_1} = \cos \alpha(y) \vec{n}(y) + \sin \alpha(y) \vec{T}(y)$, ${e_2} = - \sin \alpha(y) \vec{n}(y) + \cos \alpha(y) \vec{T}(y)$. For any $\psi \in C^2(D)$ and $y \in W$ we have
\begin{eqnarray*}
\psi_1(y) &=& \cos \alpha(y) \frac{\partial \psi}{\partial \vec{n}}(y) 
+ \sin \alpha(y) \frac{\partial \psi}{\partial \vec{T}}(y),\\
\psi_2(y) &=& -\sin \alpha(y) \frac{\partial \psi}{\partial \vec{n}}(y) 
+ \cos \alpha(y) \frac{\partial \psi}{\partial \vec{T}}(y),
\end{eqnarray*}
\begin{eqnarray*}
\psi_{11}(y) &=& \cos^2 \alpha(y) \frac{\partial^2 \psi}{\partial \vec{n}^2}(y) 
+ \sin^2 \alpha(y) \frac{\partial^2 \psi}{\partial \vec{T}^2}(y) 
+ 2 \sin \alpha(y) \cos \alpha(y) \frac{\partial^2 \psi}{\partial \vec{n} \partial \vec{T}}(y),\\
\psi_{22}(y) &=& \cos^2 \alpha(y) \frac{\partial^2 \psi}{\partial \vec{T}^2}(y) 
+ \sin^2 \alpha(y) \frac{\partial^2 \psi}{\partial \vec{n}^2}(y) 
- 2 \sin \alpha(y) \cos \alpha(y) \frac{\partial^2 \psi}{\partial \vec{n} \partial \vec{T}}(y),\\
\psi_{12}(y) &=& (\cos^2 \alpha(y) - \sin^2 \alpha(y)) \frac{\partial^2 \psi}{\partial \vec{n} \partial \vec{T}}(y)
- \sin \alpha(y) \cos \alpha(y) \left(\frac{\partial^2 \psi}{\partial \vec{n}^2}(y) - \frac{\partial^2 \psi}{\partial \vec{T}^2}(y) \right). 
\end{eqnarray*}

5. For any $y \in \{(y_1,y_2) \in W: \, y_2 > 0\}$ we have
$$
c_5 (f^{-1}(y_1) - y_2) f^{-1}(y_1) \le \delta_D(y) 
\le c_6 (f^{-1}(y_1) - y_2) f^{-1}(y_1),
$$
where $f^{-1}: [0,f(r)] \to [0,r]$.
\end{lemma}
This lemma follows by elementary geometry and its proof is omitted.

In the sequel we will use the method of continuity (cf. \cite[page 20]{KL1987}, \cite{CF1985}). Roughly speaking, we will deform a convex bounded domain $D$ to a ball $B(0,1)$. To do this we will consider the following construction. Let $C_1 > 0$, $R_1 > 0$, $\kappa_2 \ge \kappa_1 > 0$. For any $D \in F(C_1,R_1,\kappa_1,\kappa_2)$ and $t \in [0,1]$ we define 
\begin{equation}
\label{construction}
D(t) = (1-t) D + t B(0,1).
\end{equation}
\begin{lemma}
\label{deformation}
For any $C_1 > 0$, $R_1 > 0$, $\kappa_2 \ge \kappa_1 > 0$ there exists $C'_1 > 0$, $R'_1 > 0$, $\kappa'_2 \ge \kappa'_1 > 0$ such that for any $D \in F(C_1,R_1,\kappa_1,\kappa_2)$ and any $t \in [0,1]$ we have $D(t) \in F(C'_1,R'_1,\kappa'_1,\kappa'_2)$.
\end{lemma}

\begin{proof}
This lemma seems to be standard, similar results are well known (cf. \cite[proof of Theorem 3.1]{CF1985}). Notation and most of the arguments are taken from Appendix in D. Gilbarg and N. Trudinger's book \cite{GT1977}, pages 381-384.

Clearly, $D(t)$ is a convex domain satisfying $B(0,R_1) \subset D(t) \subset B(0,1)$. Fix $t \in [0,1]$. Put $E(t) = (1-t) D$, we have $D(t) = E(t) \cup \{x \in E(t)^c: \, \delta_{E(t)}(x) < t\}$. In particular
\begin{equation}
\label{boundarydelta}
\partial D(t) = \{x \in E(t)^c: \, \delta_{E(t)}(x) = t\}.
\end{equation} 
By scaling for any $x \in \partial E(t)$ we have
\begin{equation}
\label{Etcurv}
\frac{\kappa_1}{1-t} \le \kappa_{E(t)}(x) \le \frac{\kappa_2}{1-t},
\end{equation}
where $\kappa_{E(t)}(x) > 0$ denotes the curvature of $\partial E(t)$ at $x$ (see the definition on page 381 in \cite{GT1977}).

We will now use \cite[Appendix]{GT1977} (mainly we will use arguments used in the proofs of Lemmas 1, 2 and not necessarily assertions of these lemmas). We will use arguments for the set $E(t)^c$. Fix $x_0 \in E(t)^c$ such that $\dist(x_0,\partial E(t)) = t$. Let $y_0 \in \partial E(t)$ be a point such that $|x_0 - y_0| = t$. By the arguments in Lemmas 1, 2 \cite[Appendix]{GT1977} $\delta_{E(t)}(x)$ is a $C^2$ function on $\text{int}(E(t)^c)$. Choose Cartesian coordinate system $(x_1,x_2)$ such that $x_2$-axis lies in the direction $x_0 - y_0$ and the origin is $x_0$ (i.e. $x_0$ has coordinates $(0,0)$). This coordinate system is obtained by translation and rotation of the original coordinate system.

By arguments as in Lemma 2 \cite[Appendix]{GT1977} $\nabla \delta_{E(t)}(x_0) = (0,1)$, $D_{12} \delta_{E(t)}(x_0) = 0$, $D_{22} \delta_{E(t)}(x_0) = 0$, $D_{11} \delta_{E(t)}(x_0) = \frac{\kappa_{E(t)}(y_0)}{1+\kappa_{E(t)}(y_0) \delta_{E(t)}(x_0)}$ (in the assertion of Lemma 2 there are minuses in front of curvatures, here we do not have minuses because we consider $E(t)^c$ and the curvature $\kappa_{E(t)}$ was chosen to be positive). 

Put $F(x_1,x_2) = \delta_{E(t)}(x_1,x_2) - t$. We have $F_1(0,0) = 0$, $F_2(0,0) = 1$, $F_{11}(0,0) = \frac{\kappa_{E(t)}(y_0)}{1+\kappa_{E(t)}(y_0) \delta_{E(t)}(x_0)}$, $F_{12}(0,0) = 0$, $F_{22}(0,0) = 0$. By the implicit function theorem there exists a $C^2$ function $\psi:(-\eta,\eta) \to \R$, $\eta > 0$ such that $F(x_1,\psi(x_1)) = 0$. Hence by (\ref{boundarydelta}) $\partial D(t)$ is locally $C^2$.

We have $\psi' = -F_1/F_2$, $\psi'' = (2F_1F_2F_{12} - (F_2)^2F_{11} - (F_1)^2F_{22})(F_2)^{-3}$, so $\psi'(0) = 0$, 
$$
\psi''(0) = - F_{11}(0) = 
\frac{-\kappa_{E(t)}(y_0)}{1+\kappa_{E(t)}(y_0) \delta_{E(t)}(x_0)}
= \frac{-\kappa_{E(t)}(y_0)}{1+\kappa_{E(t)}(y_0) t}.
$$
Hence the curvature of $\partial D(t)$ at $x_0$ satisfies
$$
\kappa_{D(t)}(x_0) = \frac{\kappa_{E(t)}(y_0)}{1+\kappa_{E(t)}(y_0) t} 
= \frac{1}{\frac{1}{\kappa_{E(t)}(y_0)}+ t}.
$$
By (\ref{Etcurv})
$$
\frac{1-t}{\kappa_2} \le \frac{1}{\kappa_{E(t)}(y_0)} \le \frac{1-t}{\kappa_1},
$$
so 
$$
\kappa_1 \wedge 1 
\le \frac{1}{\frac{1-t}{\kappa_1}+ t}
\le \frac{1}{\frac{1}{\kappa_{E(t)}(y_0)}+ t}
\le \frac{1}{\frac{1-t}{\kappa_2}+ t}
\le \kappa_2 \vee 1. 
$$
Hence the curvature of $\partial D(t)$ at $x_0$ is between $\kappa_1 \wedge 1$ and $\kappa_2 \vee 1$.

Now we will show that the curvature $\kappa_{D(t)}(x)$ is Lipschitz. For any $x \in \text{int}(E(t)^c)$ there exists a unique point $y = y(x) \in \partial E(t)$ such that $|x - y| = \delta_{E(t)}(x)$. By \cite[Appendix]{GT1977} the function $y(x)$ is $C^1$ on $\text{int}(E(t)^c)$. Let $\nu(y)$ be the unit inner normal vector of $E(t)^c$ at $y$. We have $x = y(x) + \nu(y(x)) \delta_{E(t)}(x)$.

Let $z_0 \in \text{int}(E(t)^c)$, let $y_0 = y(z_0)$ i.e. $y_0$ is a unique point such that $y_0 \in \partial E(t)$ and $|y_0 - z_0| = \delta_{E(t)}(z_0)$. We use a Cartesian coordinate system $(x_1,x_2)$ as above with the origin $z_0$ and such that $x_2$-axis lies in the direction $z_0 - y_0$ (this coordinate system is obtained by translation and rotation of the original coordinate system). 

Using the same notation as in \cite[Appendix]{GT1977} note that for $y = (y_1,y_2) \in \partial E(t)$ near $y_0$ we have $\nu(y_1,y_2) = \overline{\nu}(y_1)$. Let us denote $\overline{\nu}(y) = (\overline{\nu}_1(y), \overline{\nu}_2(y))$. We have $(y_1(z_0),y_2(z_0)) = y(z_0) = y_0 = (y_{01},y_{02}) = (0,y_{02})$, $\overline{\nu}_1(y_{01}) = 0$, $\overline{\nu}_2(y_{01}) = 1$.

For $x \in \text{int}(E(t)^c)$ near $z_0$ we have
$$
x = y(x) + \overline{\nu}(y_1(x)) \delta_{E(t)}(x).
$$
In particular
\begin{eqnarray}
\label{x1y1}
x_1 &=& y_1(x) + \overline{\nu}_1(y_1(x)) \delta_{E(t)}(x),\\
\label{x2y2}
x_2 &=& y_2(x) + \overline{\nu}_2(y_1(x)) \delta_{E(t)}(x)
\end{eqnarray}
By computing $D_1$ derivative of (\ref{x1y1}) we get 
$$
1 = D_1 y_1(x) + D \overline{\nu}_1 (y_1(x)) D_1 y_1(x) \delta_{E(t)}(x) 
+ \overline{\nu}_1 (y_1(x)) D_1 \delta_{E(t)}(x).
$$
Putting $x = z_0$ (recall that $y(z_0) = y_0$) we obtain
$$
1 = D_1 y_1(z_0) + D \overline{\nu}_1 (y_{01}) D_1 y_1(z_0) \delta_{E(t)}(z_0) 
+ \overline{\nu}_1 (y_{01}) D_1 \delta_{E(t)}(z_0).
$$
By \cite[(A6), page 382]{GT1977} we get $D \overline{\nu}_1 (y_{01}) = \kappa_{E(t)}(y_0)$. We also have $\overline{\nu}_1 (y_{01}) = 0$. Hence
$$
D_1 y_1(z_0) = (1 + \kappa_{E(t)}(y_0) \delta_{E(t)}(z_0))^{-1}.
$$
By computing $D_2$ derivative of (\ref{x1y1}) we get 
$$
0 = D_2 y_1(x) + D \overline{\nu}_1 (y_1(x)) D_2 y_1(x) \delta_{E(t)}(x) 
+ \overline{\nu}_1 (y_1(x)) D_2 \delta_{E(t)}(x).
$$
Putting $x = z_0$ we obtain
$$
0 = D_2 y_1(z_0) + D \overline{\nu}_1 (y_{01}) D_2 y_1(z_0) \delta_{E(t)}(z_0) 
+ \overline{\nu}_1 (y_{01}) D_2 \delta_{E(t)}(z_0).
$$
Hence $D_2 y_1(z_0) = 0$.

By computing $D_1$ derivative of (\ref{x2y2}) we get 
$$
0 = D_1 y_2(x) + D \overline{\nu}_2 (y_1(x)) D_1 y_1(x) \delta_{E(t)}(x) 
+ \overline{\nu}_2 (y_1(x)) D_1 \delta_{E(t)}(x).
$$
Putting $x = z_0$ we obtain
$$
0 = D_1 y_2(z_0) + D \overline{\nu}_2 (y_{01}) D_1 y_1(z_0) \delta_{E(t)}(z_0) 
+ \overline{\nu}_2 (y_{01}) D_1 \delta_{E(t)}(z_0).
$$
By \cite[line 6, page 383]{GT1977} we have $D_1 \delta_{E(t)}(z_0) = 0$. 
From the formula for $\overline{\nu}_2 (y_{1}) = \nu_2(y)$ in \cite[(A5), page 382]{GT1977} it follows that $D \overline{\nu}_2 (y_{01}) = 0$. Hence $D_1 y_2(z_0) = 0$.  
By computing $D_2$ derivative of (\ref{x2y2}) we get 
$$
1 = D_2 y_2(x) + D \overline{\nu}_2 (y_1(x)) D_2 y_1(x) \delta_{E(t)}(x) 
+ \overline{\nu}_2 (y_1(x)) D_2 \delta_{E(t)}(x).
$$
Putting $x = z_0$ we obtain
$$
1 = D_2 y_2(z_0) + D \overline{\nu}_2 (y_{01}) D_2 y_1(z_0) \delta_{E(t)}(z_0) 
+ \overline{\nu}_2 (y_{01}) D_2 \delta_{E(t)}(z_0).
$$
We have $\overline{\nu}_2 (y_{01}) = 1$. By \cite[line 6, page 383]{GT1977} we have $D_2 \delta_{E(t)}(z_0) = 1$. Hence $D_2 y_2(z_0) = 0$.

Finally we get 
\begin{equation}
\label{gradients}
|D_1 y_1(z_0)| \le 1, \quad \quad D_2 y_1(z_0) = 0, \quad \quad |\nabla y_2(z_0)| = 0.
\end{equation}

Let $x_0 \in E(t)^c$ be such that $\dist(x_0,\partial E(t)) = t$ i.e. $x_0 \in \partial D(t)$. Choose $z_0 = x_0$, we have $y(z_0) = y(x_0) = y_0$. 
For any $x \in \partial D(t)$ which are sufficiently close to $x_0$ we get
$$
|\kappa_{D(t)}(x_0) - \kappa_{D(t)}(x)| 
= \left|\frac{\kappa_{E(t)}(y_0)}{1+ t \kappa_{E(t)}(y_0)}
- \frac{\kappa_{E(t)}(y(x))}{1+ t \kappa_{E(t)}(y(x))}\right|.
$$
Using $\kappa_{E(t)}(y(x)) = \frac{1}{1 - t} \kappa_D \left(\frac{y(x)}{1-t} \right)$ this is equal to
\begin{eqnarray*}
\frac{(1-t)\left|\kappa_D\left(\frac{y_0}{1-t}\right) - \kappa_D\left(\frac{y(x)}{1-t}\right)\right|}{\left|1 - t + t \kappa_D\left(\frac{y_0}{1-t}\right)\right| \left|1 - t + t \kappa_D\left(\frac{y(x)}{1-t}\right)\right|}
&\le&
\frac{(1-t)C_1 \left|\frac{y_0 - y(x)}{1 - t}\right|}{(\kappa_1 \wedge 1)^2}\\
&\le& \frac{C_1}{(\kappa_1 \wedge 1)^2} |y_0 - y(x)|.
\end{eqnarray*}

We estimate now $|y(x_0) - y(x)|$ for $x \in \partial D(t)$ which are sufficiently close to $x_0$. We have
\begin{eqnarray*}
y(x_0) - y(x) &=& (y_1(x_0) - y_1(x), y_2(x_0) - y_2(x)),\\
|y_1(x_0) - y_1(x)| &=& |\nabla y_1(\xi)| |x - x_0|,\\
|y_2(x_0) - y_2(x)| &=& |\nabla y_2(\tilde{\xi})| |x - x_0|,
\end{eqnarray*}
where $\xi$, $\tilde{\xi}$ are points between $x$ and $x_0$. For $x \in \partial D(t)$ which are sufficiently close to $x_0$ we have $\delta_{E(t)}(\xi) \ge t/2$, $\delta_{E(t)}(\tilde{\xi}) \ge t/2$. It follows that $\xi \in \text{int}(E(t)^c)$, $\tilde{\xi} \in \text{int}(E(t)^c)$. Using (\ref{gradients}) in the appropriate way we get $|\nabla y_1(\xi)| \le 1$, $|\nabla y_2(\tilde{\xi})| \le 1$ (this follows by translation and rotation of a coordinate system). Hence $|y(x_0) - y(x)| \le \sqrt{2} |x- x_0|$. Therefore
$$
|\kappa_{D(t)}(x_0) - \kappa_{D(t)}(x)| 
\le \frac{C_1}{(\kappa_1 \wedge 1)^2} |y_0 - y(x)|
\le \frac{C_1 \sqrt{2}}{(\kappa_1 \wedge 1)^2} |x - x_0|.
$$
This holds for $x \in \partial D(t)$ which are sufficiently close to $x_0$ but by simple geometric arguments it can be extended to any $x \in \partial D(t)$ (with a different constant).
\end{proof}

Now we state some properties of the solution of (\ref{maineq1}-\ref{maineq2}) and its harmonic extension which will be needed in the rest of the paper. 

Let $D \subset \R^2$ be an open bounded set and $\vp^{(D)}$ be the solution of (\ref{maineq1}-\ref{maineq2}) for $D$. Then the following scaling property is well known \cite[(1.61)]{book2009}:
\begin{equation}
\label{scaling}
\vp^{(aD)}(ax) = a \vp^{(D)}(x), \quad x \in D, \, a > 0.
\end{equation}

For any open bounded sets $D_1, D_2 \subset \R^2$ put $d(D_1,D_2) = 
[\sup\{\dist(x,\partial D_2): \, x \in \partial D_1\}] \wedge 
[\sup\{\dist(x,\partial D_1): \, x \in \partial D_2\}]$. 
\begin{lemma}
\label{phiconv}
Let $\{D_n\}_{n = 0}^{\infty}$ be a sequence of bounded convex domains in $\R^2$ and $\vp^{(D_n)}$ be the solution of (\ref{maineq1}-\ref{maineq2}) for $D_n$. If $d(D_n,D_0) \to 0$ as $n \to \infty$ then for any $x \in D_0$ we have $\vp^{(D_n)}(x) \to \vp^{(D_0)}(x)$ as $n \to \infty$.  
\end{lemma}
This lemma seems to be well known and follows easily from (\ref{scaling}) so we omit its proof (in fact it holds not only for convex domains but we need it only in this case).

\begin{lemma}
\label{lowerhalfspace}
Let $C_1 > 0$, $R_1 > 0$, $\kappa_2 \ge \kappa_1 > 0$, $D \in F(C_1,R_1,\kappa_1,\kappa_2)$, $\vp$ be the solution of (\ref{maineq1}-\ref{maineq2}) for $D$ and $u$ the harmonic extension of $\vp$ given by (\ref{ext1}-\ref{ext2}). For any $(x_1,x_2,x_3) \in \R_+^3$ we have $H(u)(x_1,x_2,-x_3) = H(u)(x_1,x_2,x_3)$.
\end{lemma}
\begin{proof}
For $x = (x_1,x_2,x_3)$ put $\hat{x} = (x_1,x_2,-x_3)$. For $x \in \R_+^3$ we have $u_{ii}(\hat{x}) = u_{ii}(x)$ for $i = 1,2,3$, $u_{12}(\hat{x}) = u_{12}(x)$, $u_{13}(\hat{x}) = -u_{13}(x)$, $u_{23}(\hat{x}) = -u_{23}(x)$. Hence $H(u)(\hat{x}) = H(u)(x)$. 
\end{proof}

We will need the following formulas of derivatives of $K(x) = C_K x_3 (x_1^2+x_2^2+x_3^2)^{-3/2}$:
\begin{eqnarray*}
K_1(x) &=& -3 C_K x_3 x_1 (x_1^2+x_2^2+x_3^2)^{-5/2},\\
K_2(x) &=& -3 C_K x_3 x_2 (x_1^2+x_2^2+x_3^2)^{-5/2},\\
K_3(x) &=& C_K (x_1^2+x_2^2-2x_3^2) (x_1^2+x_2^2+x_3^2)^{-5/2}.
\end{eqnarray*}
\begin{eqnarray*}
K_{11}(x) &=& C_K x_3 (12 x_1^2 - 3 x_2^2 - 3 x_3^2)(x_1^2+x_2^2+x_3^2)^{-7/2},\\
K_{22}(x) &=& C_K x_3 (12 x_2^2 - 3 x_1^2 - 3 x_3^2)(x_1^2+x_2^2+x_3^2)^{-7/2},\\
K_{33}(x) &=& C_K x_3 (6 x_3^2 - 9 x_1^2 - 9 x_2^2) (x_1^2+x_2^2+x_3^2)^{-7/2}.\\
K_{12}(x) &=& 15 C_K x_3 x_1 x_2 (x_1^2+x_2^2+x_3^2)^{-7/2},\\
K_{13}(x) &=& C_K x_1 (12 x_3^2 - 3x_1^2 - 3 x_2^2) (x_1^2+x_2^2+x_3^2)^{-7/2},\\
K_{23}(x) &=& C_K x_2 (12 x_3^2 - 3x_1^2 - 3 x_2^2)(x_1^2+x_2^2+x_3^2)^{-7/2}.
\end{eqnarray*}

\begin{remark}
\label{constants}
All constants appearing in this paper are positive and finite. We write $C = C(a,\ldots,z)$ to emphasize that $C$ depends only on $a,\ldots,z$. We adopt the convention that constants denoted by $c$ (or $c_1$, $c_2$, etc.) may change their value from one use to the next.
\end{remark}

\begin{remark}
\label{constants1}
In Sections 3, 4 and in the proof of Proposition \ref{vepsilon} we use the following convention. Constants denoted by $c$ (or $c_1$, $c_2$, etc. ) depend on $\Lambda = \{C_1,R_1,\kappa_1,\kappa_2\}$, where $\Lambda = \{C_1,R_1,\kappa_1,\kappa_2\}$ appear in Definition \ref{classF} 
We write $f(x) \approx g(x)$ for $x \in A \subset \R^2$ to indicate that there exist constants $c_1 = c_1(\Lambda)$, $c_2 = c_2(\Lambda)$ such that for any $x \in A$ we have $c_1 g(x) \le f(x) \le c_2 g(x)$ (in particular, it may happen that both $f$, $g$ are positive on $A$ or both $f$, $g$ are negative on $A$). 
\end{remark}

\section{Estimates of derivatives of $\vp$ near $\partial D$}

In this section the behaviour of $\vp_{i,j}$ near $\partial D$ is studied. The section contains quite complicated and technical estimates. Some new methods are used see e.g. the proof of Lemma \ref{phi22}. Nevertheless, most of the technics used in this section are similar to the technics used in the papers by T. Kulczycki \cite{K1997} and Z.-Q. Chen, R. Song \cite{CS1998}. It should be mentioned that similar estimates of derivatives of $\alpha$-harmonic functions were simultaneously obtained by the author's student G. {\.Z}urek in his Master Thesis \cite{Z2014}.

In the whole section we fix $C_1 > 0$, $R_1 > 0$, $\kappa_2 \ge \kappa_1 > 0$, $D \in F(C_1,R_1,\kappa_1,\kappa_2)$ and $x_0 \in \partial D$. We put $\Lambda = \{C_1,R_1,\kappa_1,\kappa_1\}$. $\vp$ is the solution of (\ref{maineq1}-\ref{maineq2}) for $D$. Unless it is stated otherwise we fix the coordinate system $CS_{x_0}$ and notation as in Lemma \ref{ball} (see Figure 1). In particular $x_0$ is $(0,0)$ in $CS_{x_0}$ coordinates. Let us recall that in the whole section we use convention stated in Remark \ref{constants1}.

Let $r \in (0,r_0]$, $z = (r,0)$, $s \in (0,r]$, $B = B(z,s)$ (where $r_0$ is the constant from Lemma \ref{ball}). It is well known (see e.g. \cite[(1.50), (1.56), (1.57)]{book2009}) that 
\begin{equation}
\label{phiformula}
\vp(x) = h(x) + \int_{B^c} P(x,y) \vp(y) \, dy, \quad x \in B,
\end{equation} 
where $h(x) = C_B (s^2 - |x - z|^2)^{1/2}$, $x \in B$,
\begin{equation}
\label{Pformula}
P(x,y)= C_P \frac{(s^2 - |x - z|^2)^{1/2}}{(|y - z|^2 - s^2)^{1/2} |x-y|^2}, \quad x \in B, \, \, \, y \in (\overline{B})^c,
\end{equation}
$C_B = 2/\pi$, $C_P = \pi^{-2}$.

We have $h_1(x) = C_B (r - x_1) (s^2 - |x - z|^2)^{-1/2}$, $x \in B$. Put $P_i(x,y) = \frac{\partial}{\partial x_i} P(x,y)$, $i = 1,2$. For any $x \in B$, $y \in (\overline{B})^c$ we have $P_1(x,y) = A(x,y) + E(x,y)$ where
\begin{equation}
\label{Aformula}
A(x,y)= -C_P \frac{(s^2 - |x - z|^2)^{-1/2}(x_1 - r)}{(|y - z|^2 - s^2)^{1/2} |x-y|^2}, 
\end{equation}
\begin{equation}
\label{Eformula}
E(x,y)= -2C_P \frac{(s^2 - |x - z|^2)^{1/2}(x_1 - y_1)}{(|y - z|^2 - s^2)^{1/2} |x-y|^4}. 
\end{equation}

It is also well known (see e.g. \cite{CS1997}) that $\vp(y) \le c \delta_D^{1/2}(y)$.

\begin{lemma}
\label{phi1} There exists $r_1 \in (0,r_0/4]$, $r_1 = r_1(\Lambda)$ such that for any $x_1 \in (0,r_1]$ we have $\vp_1(x_1,0) \approx x_1^{-1/2}$.
\end{lemma}
\begin{proof}
Put $r = r_0$. We will use (\ref{phiformula}) for $s = r$, in particular $B = B(z,r)$. We have $h_1(x_1,0) = C_B (r - x_1) (2r - x_1)^{-1/2} x_1^{-1/2} \approx c x_1^{-1/2}$ for $x_1 \in (0,r/4]$. Put 
$$
k(x) = 1_{B}(x) \int_{B^c} P(x,y) \vp(y) \, dy + 1_{B^c}(x) \vp(x), \quad x \in \R^2.
$$
We have $k(x) \ge 0$ on $\R^2$, by (\ref{phiformula}) $k(x) \le \vp(x)$ on $B$ and $k$ is $1$-harmonic on $B$. For the definition and basic properties of $\alpha$-harmonic functions see \cite[pages 20-21, 61]{book2009}. The fact that $k$ is $1$-harmonic follows from \cite[page 61]{book2009}. By \cite[Lemma 3.2]{BKN2002} (cf. also \cite{KR2013}) $k_1(x_1,0) \le c x_1^{-1/2}$ for $x_1 \in (0,r]$. Hence $\vp_1(x_1,0) = h_1(x_1,0) + k_1(x_1,0) \le c x_1^{-1/2}$ for $x_1 \in (0,r/4]$.

What remains is to show that $\vp_1(x_1,0) \ge c x_1^{-1/2}$. For $x_1 \in (0,r]$ we have $\vp_1(x_1,0) = \int_{B^c} P_1((x_1,0),y) \vp(y) \, dy + h_1(x_1,0)$. We will estimate $\int_{B^c} P_1 \vp$.

Let $x_1 \in (0,f(r/2) \wedge f(-r/2)]$. Note that $f(r/2) \wedge f(-r/2) \ge c_3 r^2/4$, where $c_3$ and $r = r_0$ are constants from Lemma \ref{ball}, $c_3 r^2/4$ depends only on $\Lambda$. Let $p_1 \in (0,r/2]$ be such that $f(p_1) = x_1$, $p_2 \in [-r/2,0)$ be such that $f(p_2) = x_1$ (recall that $f$ is defined in Lemma \ref{ball}). By Lemma \ref{ball} $f(x_1) < c_4 x_1^2 \le (1/2) x_1$, $f(-x_1) \le (1/2) x_1$, so $p_1 > x_1$ and $|p_2| > x_1$. Let $f_1:[-r,r] \to \R$ be defined by $f_1(y_2) = r - (r^2 - y_2^2)^{1/2}$. Put
\begin{eqnarray*}
D_1 &=& \{(y_1,y_2): \, y_2 \in [-x_1,x_1], y_1 \in (f(y_2),f_1(y_2))\}, \\
D_2 &=& \{(y_1,y_2): \, y_2 \in (x_1,p_1] \cup [p_2,-x_1), y_1 \in (f(y_2),f_1(y_2) \wedge x_1)\}, \\
D_3 &=& D \setminus (D_1 \cup D_2 \cup B).
\end{eqnarray*}
Note that $\int_{D \setminus B} A((x_1,0),y) \vp(y) \, dy > 0$ and $\int_{D_3} E((x_1,0),y) \vp(y) \, dy > 0$.

Since $\vp(y) \le c \delta_D^{1/2}(y)$ we get $\vp(y) \le c x_1$ for $y \in D_1$ and 
$\vp(y) \le c y_2$ for $y \in D_2$. Note that for $y \in D_1 \cup D_2$ we have $|y-z|^2 - r^2 \approx f_1(y_2) - y_1$. Hence
\begin{eqnarray*}
\left| \int_{D_1} E((x_1,0),y) \vp(y) \, dy \right|
&\le& c x_1^{-3/2} \int_{D_1} \frac{dy}{(|y-z|^2 - r^2)^{1/2}} \\
&\approx& x_1^{-3/2} \int_{-x_1}^{x_1} \, dy_2 \int_{f(y_2)}^{f_1(y_2)} 
(f_1(y_2) -y_1)^{-1/2} \, dy_1 \\
&\approx& x_1^{1/2}.
\end{eqnarray*}
Note that $p_1 \le c \sqrt{x_1} \wedge (r/2)$, $|p_2| \le c \sqrt{x_1} \wedge (r/2)$ so 
\begin{eqnarray*}
\left| \int_{D_2} E((x_1,0),y) \vp(y) \, dy \right|
&\le& c x_1^{3/2} \int_{x_1}^{c \sqrt{x_1} \wedge (r/2)} \, dy_2 y_2^{-3} \int_{f(y_2)}^{f_1(y_2) \wedge x_1} (f_1(y_2) -y_1)^{-1/2} \, dy_1 \\
&\approx& x_1^{1/2},
\end{eqnarray*}
(we omit here $\int_{p_2}^{-x_1} \ldots$ because it can be estimated in the same way).

It follows that 
\begin{eqnarray*}
\vp_1(x_1,0) &=&
h_1(x_1,0) + \int_{D \setminus B} A \vp + \int_{D_1} E \vp + \int_{D_2} E \vp + \int_{D_3} E \vp \\
&\ge& c x_1^{-1/2} - c_1 x_1^{1/2} \ge c x_1^{-1/2}
\end{eqnarray*}
for sufficiently small $x_1$ (recall that we use convention from Remark \ref{constants} that a constant $c$ may change its value from one use to the next).
\end{proof}

\begin{lemma}
\label{phi2} There exists $r_1 \in (0,r_0/4]$, $r_1 = r_1(\Lambda)$ such that for any $x_1 \in (0,r_1]$ we have $|\vp_2(x_1,0)| \le c x_1^{1/2} |\log x_1|$.
\end{lemma}
\begin{proof}
Put $r = r_0$. We will use (\ref{phiformula}) for $s = r$, in particular $B = B(z,r)$. Let $x_1 \in (0,r/4]$. We have $\vp_2(x_1,0) = \int_{B^c} P_2((x_1,0),y) \vp(y) \, dy + h_2(x_1,0)$, $h_2(x_1,0) = 0$, $P_2((x_1,0),y) = 2 C_P \frac{(r^2 - |x - z|^2)^{1/2}y_2}{(|y - z|^2 - r^2)^{1/2} |x-y|^4}$, $y \in (\overline{B})^c$. Let $f_1$ be such as in the proof of Lemma \ref{phi1}. Put 
\begin{eqnarray*}
D_1 &=& \{(y_1,y_2): \, y_2 \in [-x_1,x_1], y_1 \in (f(y_2),f_1(y_2))\}, \\
D_2 &=& \{(y_1,y_2): \, y_2 \in (x_1,r/2] \cup [-r/2,-x_1), y_1 \in (f(y_2),f_1(y_2))\}, \\
D_3 &=& D \setminus (D_1 \cup D_2 \cup B).
\end{eqnarray*} 
Note that $\vp(y) \le cx_1$ for $y \in D_1$ and $\vp(y) \le c y_2$ for $y \in D_2$.
Similarly like in Lemma \ref{phi1} we obtain
$$
\left|\int_{D_1} P_2((x_1,0),y) \vp(y) \, dy\right| 
\le c x_1^{-3/2} \int_{D_1} \frac{dy}{(|y-z|^2 - r^2)^{1/2}} 
\approx x_1^{1/2}, 
$$
\begin{eqnarray*}
\left|\int_{D_2} P_2((x_1,0),y) \vp(y) \, dy\right| 
&\le& c x_1^{1/2} \int_{x_1}^{r/2} \, dy_2 y_2^{-2} \int_{f(y_2)}^{f_1(y_2)} (f_1(y_2) - y_1)^{-1/2} \, dy_1 \\
&\approx& x_1^{1/2} |\log x_1|.
\end{eqnarray*}
$$
\left|\int_{D_3} P_2((x_1,0),y) \vp(y) \, dy\right| 
\le c x_1^{1/2} \int_{D \setminus B} \frac{dy}{(|y-z|^2 - r^2)^{1/2}} 
\le c x_1^{1/2}. 
$$
It follows that $|\vp_2(x_1,0)| \le c x_1^{1/2} |\log x_1|$.
\end{proof}

By Lemmas \ref{phi1}, \ref{phi2} and \ref{ball} we obtain
\begin{corollary}
\label{nT}
There exists $r_1 \in (0,r_0/4]$, $r_1 = r_1(\Lambda)$ such that for any $y \in D$, $\delta_D(y) \le r_1$ we have
\begin{eqnarray}
\label{phin}
\frac{\partial \vp}{\partial \vec{n}}(y) 
&\approx& \delta_D^{-1/2}(y),\\
\label{phit}
\left| \frac{\partial \vp}{\partial \vec{T}}(y) \right| 
&\le& c \delta_D^{1/2}(y) |\log \delta_D(y)|,\\
\label{phigradient}
|\nabla \vp(y)|
&\le& c \delta_D^{-1/2}(y).
\end{eqnarray}
\end{corollary}

\begin{lemma}
\label{12harmonic}
$\vp_1$, $\vp_2$ are singular $1$-harmonic on $D$.
\end{lemma}

\begin{remark}
\label{12harmonic1}
$\vp_{11}$, $\vp_{22}$ are not $1$-harmonic on $D$ because they are not locally integrable on $\R^2$ (see Corollary \ref{nT1}).
\end{remark}

For the definition of singular $\alpha$-harmonic functions a reader is referred to \cite[page 61]{book2009}. The definition of $\alpha$-harmonic functions on an open set $U \subset \R^d$ demands that the function is defined on the whole $\R^d$. $\vp_1$, $\vp_2$ are well defined on $D$ and also on $D^c \setminus \partial D$. $\vp_1$, $\vp_2$ are not well defined on $\partial D$ but $\partial D$ has Lebesgue measure zero. One may formally defined $\vp_1 = \vp_2 = 0$ on $\partial D$.

\begin{proof}[Proof of Lemma \ref{12harmonic}]
It is enough to show the Lemma for $\vp_1$. Fix $x \in D$, put $\delta_D(x) = 2s$. For any $z \in B(x,s/2)$ (see \cite[page 9]{book2009}) we have
$$
-(-\Delta)^{1/2} \vp(z) = 
\int_{\R^2} \left(\vp(z+y) -\vp(z) - y \nabla \vp(z) 1_{B(0,s)}(y)\right) \nu(y) \, dy,
$$
where $\nu(y) = (2 \pi)^{-1} |y|^{-3}$. 

Put $F_1(z,y) = \vp(z+y) -\vp(z) - y \nabla \vp(z)$, $F_2(z,y) = \vp(z+y) -\vp(z)$,
\begin{eqnarray*}
D_h &=& \{y \in \R^2: \, \delta_D(y) < 2h\},\\
U_{1,h} &=& \{y \in \R^2: \, y + x \in D_h \},\\
U_{2,h} &=& \{y \in \R^2: \, y + x \in D \setminus D_h \} \setminus B(0,s),\\
U_{3,h} &=& \{y \in \R^2: \, y + x \in D^c \setminus D_h \}.
\end{eqnarray*}
For any $h \in (0,s/2)$ define
$$
L(h) = \frac{-(-\Delta)^{1/2} \vp(x + e_1 h) + (-\Delta)^{1/2} \vp(x)}{h}
$$
By (\ref{maineq1}-\ref{maineq2}) for any $h \in (0,s/2)$ we have $L(h) = 0$.
On the other hand we have
\begin{eqnarray*}
L(h) &=& 
\frac{1}{h} \int_{B(0,s)} \left(F_1(x + e_1 h,y) - F_1(x,y)\right) \nu(y) \, dy\\
&& + \frac{1}{h} \int_{U_{1,h}} \left(F_2(x + e_1 h,y) - F_2(x,y)\right) \nu(y) \, dy\\
&& + \frac{1}{h} \int_{U_{2,h}} \left(F_2(x + e_1 h,y) - F_2(x,y)\right) \nu(y) \, dy\\  
&& + \frac{1}{h} \int_{U_{3,h}} \left(F_2(x + e_1 h,y) - F_2(x,y)\right) \nu(y) \, dy\\ 
&=& \text{I} + \text{II} + \text{III} + \text{IV}, 
\end{eqnarray*}
$$
\text{I} = (2 \pi)^{-1} \int_{B(0,s)} 
\left(\vp_1(x+ e_1 \xi +y) 
- \vp_1(x+ e_1 \xi)
- y \left(\nabla \vp_1\right) (x+ e_1 \xi) \right) |y|^{-3} \, dy,
$$
where $\xi \in (0,h)$. We have
$$
\left|\vp_1(x+ e_1 \xi +y) 
- \vp_1(x+ e_1 \xi)
- y \left(\nabla \vp_1\right) (x+ e_1 \xi)\right| \le \tilde{C} |y|^2.
$$
In this proof $\tilde{C}$ denotes a constant which depends on $x$ and $\vp$ but not on $h$ (we use the convention that it may change its value from one use to the next). It follows that
$$
\lim_{h \to 0^+} \text{I} =
\int_{B(0,s)} 
\left(\vp_1(x+y) 
- \vp_1(x)
- y \left(\nabla \vp_1\right) (x) \right) 
\nu(y)\, dy.
$$
We also have
\begin{eqnarray*}
\text{II} &=& 
\frac{1}{2 \pi h} \int_{U_{1,h}} 
\left(\vp(x+ e_1 h +y) - \vp(x+y) \right) |y|^{-3} \, dy\\
&& - \frac{1}{2 \pi h} \int_{U_{1,h}} 
\left(\vp(x+ e_1 h) - \vp(x) \right) |y|^{-3} \, dy\\
&=& \text{IV} - \text{V}.
\end{eqnarray*}
Note that $|\vp(x + y)| \le \tilde{C} h^{1/2}$, $|\vp(x + e_1 h + y)| \le \tilde{C} h^{1/2}$ and $\int_{U_{1,h}} |y|^{-3} \, dy \le \tilde{C} h$ for $h \in (0,s/2)$, $y \in U_{1,h}$ so $|\text{IV}| \le \tilde{C} h^{1/2}$. Note also that $|\vp(x+ e_1 h) - \vp(x)| \le \tilde{C} h$ for $h \in (0,s/2)$ so $|\text{V}| \le \tilde{C} h$, $|\text{II}| \le \tilde{C} h^{1/2}$.

We also have
$$
\text{III} = 
(2 \pi)^{-1} \int_{\R^2} 1_{U_{2,h}}(y) 
\left(\vp_1(x+ e_1 \xi +y) 
- \vp_1(x+ e_1 \xi) \right) |y|^{-3} \, dy,
$$
where $\xi \in (0,h)$. Note that for $h \in (0,s/2)$, $y \in U_{2,h}$ and $\xi \in (0,h)$ we have $\delta_D(x+y+e_1 \xi) \approx \delta_D(x + y)$ so
$$
1_{U_{2,h}}(y) \left| \vp_1(x+ e_1 \xi +y) 
- \vp_1(x+ e_1 \xi) \right|
\le \tilde{C} (\delta_D^{-1/2}(x+y) + 1).
$$
It follows that
$$
\lim_{h \to 0^+} \text{III} =
\int_{ \{y: \, y+x \in D\} \setminus B(0,s)} 
\left(\vp_1(x+y) 
- \vp_1(x) \right) 
\nu(y)\, dy.
$$
We also get
\begin{eqnarray*}
\text{IV} &=& 
(2 \pi)^{-1} \int_{U_{3,h}} 
- \vp_1(x+ e_1 \xi) |y|^{-3} \, dy\\
&=& 
(2 \pi)^{-1} \int_{U_{3,h}} 
\left(\vp_1(x +y) 
- \vp_1(x+ e_1 \xi) \right) |y|^{-3} \, dy,
\end{eqnarray*}
where $\xi \in (0,h)$, because $\vp_1(x +y) = 0$ for $y \in U_{3,h}$. Hence
$$
\lim_{h \to 0^+} \text{IV} =
\int_{ \{y: \, y+x \in D^c\}} 
\left(\vp_1(x+y) 
- \vp_1(x) \right) 
\nu(y)\, dy.
$$
It follows that 
$0 = \lim_{h \to 0^+} L(h) = -(-\Delta)^{-1/2} \vp_1(x)$.
\end{proof}

\begin{lemma}
\label{phi22} There exists $r_2 \in (0,r_0/4]$, $r_2 = r_2(\Lambda)$ such that for any $x_1 \in (0,r_2]$ we have $\vp_{22}(x_1,0) \approx -x_1^{-1/2}$.
\end{lemma}
\begin{proof}
Put $r = r_0$. Let $r_1$ be the constant from Corollary \ref{nT}. In this proof we take $s \in (r - (r_1/2)^2,r)$, i.e. $0 < r - s < (r_1/2)^2$. Recall that $z = (r,0)$, $B = B(z,s)$ and $P$ is given by (\ref{Pformula}). For any $x_1 \in (r - s, r]$ by Lemma \ref{12harmonic} we have $\vp_2(x_1,0) = \int_{D \setminus B} P((x_1,0),y) \vp_2(y) \, dy$. It follows that $\vp_{22}(x_1,0) = \int_{D \setminus B} P_2((x_1,0),y) \vp_2(y) \, dy$. We have $P_2((x_1,0),y) = 2 C_P \frac{(s^2 - |x - z|^2)^{1/2}y_2}{(|y - z|^2 - s^2)^{1/2} |x-y|^4}$. Take $x_1 = \sqrt{r - s}$ (we have $\sqrt{r - s} < r_1/2$). Let $f_1:[-s,s] \to \R$ be defined by $f_1(y_2) = r - \sqrt{s^2 - y_2^2}$. Put 
\begin{eqnarray*}
D_1 &=& \{(y_1,y_2): \, y_2 \in [-x_1,x_1], y_1 \in (f(y_2),f_1(y_2))\}, \\
D_2 &=& \{(y_1,y_2): \, y_2 \in (x_1,r_1/2] \cup [-r_1/2,-x_1), y_1 \in (f(y_2),f_1(y_2))\}, \\
D_3 &=& D \setminus (D_1 \cup D_2 \cup B).
\end{eqnarray*}
By Lemma \ref{ball} we have for $y \in D_1 \cup D_2$
$$
\vp_2(y) = \cos \alpha(y) \frac{\partial \vp}{\partial \vec{T}}(y) 
-\sin \alpha(y) \frac{\partial \vp}{\partial \vec{n}}(y).
$$
Note that by definition of $s$ we have $\delta_D(y) < r_1$ for $y \in D_1 \cup D_2$. By Corollary \ref{nT} we get for $y \in D_1 \cup D_2$
$$
\left|\frac{\partial \vp}{\partial \vec{T}}(y)\right| 
\le c (y_1 - f(y_2))^{1/2} |\log(y_1 - f(y_2))|
$$
and
$$
\frac{\partial \vp}{\partial \vec{n}}(y) \approx (y_1 - f(y_2))^{-1/2}.
$$
Hence
$$
\left|\cos \alpha(y)\frac{\partial \vp}{\partial \vec{T}}(y)\right| 
\le c (y_1 - f(y_2))^{1/2} |\log(y_1 - f(y_2))|
$$
and 
$$
- \sin \alpha(y) \frac{\partial \vp}{\partial \vec{n}}(y) \approx -y_2 (y_1 - f(y_2))^{-1/2}.
$$
Note also that for $y \in D_1 \cup D_2$ we have $(|y-z|^2 -s^2)^{1/2} \approx (-y_1+f_1(y_2))^{1/2}$. Recall that we have chosen $x_1 = \sqrt{r-s}$. It follows that 
\begin{eqnarray*}
&& -\int_{D_1} P_2((x_1,0),y) \sin \alpha(y) \frac{\partial \vp}{\partial \vec{n}}(y) \, dy \\
&\approx& -x_1^{-7/2} \int_{-x_1}^{x_1} \, dy_2 y_2^2 \int_{f(y_2)}^{f_1(y_2)} \, dy_1 (-y_1+f_1(y_2))^{-1/2} (y_1 - f(y_2))^{-1/2} \approx - x_1^{1/2},
\end{eqnarray*}
because $\int_a^b (x-a)^{-1/2} (b-x)^{-1/2} \, dx = \text{const}$.

Similarly,
\begin{eqnarray*}
&& -\int_{D_2} P_2((x_1,0),y) \sin \alpha(y) \frac{\partial \vp}{\partial \vec{n}}(y) \, dy \\
&\approx& -x_1^{1/2} \int_{x_1}^{r_1/2} \, dy_2 y_2^{-2} \int_{f(y_2)}^{f_1(y_2)} \, dy_1 (-y_1+f_1(y_2))^{-1/2} (y_1 - f(y_2))^{-1/2} \approx - x_1^{1/2}.
\end{eqnarray*}
On the other hand we have
\begin{eqnarray*}
&& \left|\int_{D_1} P_2((x_1,0),y) \cos \alpha(y) \frac{\partial \vp}{\partial \vec{T}}(y) \, dy \right| \\
&\le& c x_1^{-7/2} \int_{-x_1}^{x_1} \, dy_2 y_2 \int_{f(y_2)}^{f_1(y_2)} \, dy_1 (-y_1+f_1(y_2))^{-1/2} (y_1 - f(y_2))^{1/2} |\log(y_1 - f(y_2))| \\
&\le& c x_1^{1/2} |\log x_1|,
\end{eqnarray*}
\begin{eqnarray*}
&& \left|\int_{D_2} P_2((x_1,0),y) \cos \alpha(y) \frac{\partial \vp}{\partial \vec{T}}(y) \, dy \right|\\
&\le&  c x_1^{1/2} \int_{x_1}^{r_1/2} \, dy_2 y_2^{-3} \int_{f(y_2)}^{f_1(y_2)} \, dy_1 (-y_1+f_1(y_2))^{-1/2} (y_1 - f(y_2))^{1/2} |\log(y_1 - f(y_2))| \\
&\le& c x_1^{1/2} |\log x_1|^2, 
\end{eqnarray*}
$$
\left|\int_{D_3} P_2((x_1,0),y) \vp_2(y)\, dy \right|
\le c x_1^{1/2} \int_{D_3} \delta_B^{-1/2}(y) \delta_D^{-1/2}(y) \, dy \le c x_1^{1/2}.
$$
It follows that 
$$
-c_1 x_1^{-1/2} - c_2 x_1^{1/2} |\log x_1|^2 \le \vp_{22}(x_1,0) \le
-c_3 x_1^{-1/2} + c_4 x_1^{1/2} |\log x_1|^2,
$$
where $x_1 = \sqrt{r-s}$. It is very important that $c_1$, $c_2$, $c_3$, $c_4$ do not depend on $s$. Hence there exists $r_2 \in (0,r/4]$, $r_2 = r_2(\Lambda)$ such that for any $x_1 \in (0,r_2]$ we have $\vp_{22}(x_1,0) \approx - x_1^{-1/2}$.
\end{proof}

\begin{lemma}
\label{phi11}
There exists $r_2 \in (0,r_0/4]$, $r_2 = r_2(\Lambda)$ such that for any $x_1 \in (0,r_2]$ we have $\vp_{11}(x_1,0) \approx -x_1^{-3/2}$.
\end{lemma}
\begin{proof}
First we show that $|\vp_{11}(x_1,0)| \le c x_1^{-3/2}$, $x_1 \in (0,r_2]$. We will use similar notation as in Lemma \ref{phi22}. Put $r = r_0$. Let $r_1$ be the constant from Corollary \ref{nT}. We take $s \in (r - (r_1/2)^2,r)$, $z = (r,0)$, $B = B(z,s)$ and $P$ is given by (\ref{Pformula}). For any $x_1 \in (r - s, r]$ by Lemma \ref{12harmonic} we have $\vp_1(x_1,0) = \int_{D \setminus B} P((x_1,0),y) \vp_1(y) \, dy$. It follows that 
\begin{eqnarray*}
&& \vp_{11}(x_1,0) = \int_{D \setminus B} P_1((x_1,0),y) \vp_1(y) \, dy\\
&=&  \int_{D \setminus B} A((x_1,0),y) \vp_1(y) \, dy + \int_{D \setminus B} E((x_1,0),y) \vp_1(y) \, dy,
\end{eqnarray*}
where $A$, $E$ are given by (\ref{Aformula}), (\ref{Eformula}).

Take $x_1 = \sqrt{r - s}$ (we have $\sqrt{r - s} < r_1/2 \le r/8$). By (\ref{phigradient}) $|\vp_1(y)| \le c \delta_D^{-1/2}(y)$, $y \in D$. We have
$$
\int_{D \setminus B} A((x_1,0),y) \vp_1(y) \, dy
= \frac{r - x_1}{s^2 - (x_1 - r)^2} 
\int_{D \setminus B} P((x_1,0),y) \vp_1(y) \, dy,
$$ 
$$
\left|\int_{D \setminus B} P((x_1,0),y) \vp_1(y) \, dy\right| 
= |\vp_1(x_1,0)| \le c x_1^{-1/2}
$$
and $\frac{r - x_1}{s^2 - (x_1 - r)^2} \approx x_1^{-1}$ so
$$
\left|\int_{D \setminus B} A((x_1,0),y) \vp_1(y) \, dy\right| 
\le c x_1^{-3/2}
$$
for $x_1 = \sqrt{r-s}$.

Let $f_1$, $D_1$, $D_2$, $D_3$ be such as in the proof of Lemma \ref{phi22}. Using  $|\vp_1(y)| \le c \delta_D^{-1/2}(y)$ and similar arguments as in the proof of Lemma \ref{phi22} we get the following estimates 
\begin{eqnarray}
\label{phi11D1}
&& \left|\int_{D_1} E((x_1,0),y) \vp_1(y) \, dy\right| \\
\nonumber
&\le& c x_1^{-5/2} \int_{-x_1}^{x_1} \, dy_2 \int_{f(y_2)}^{f_1(y_2)} \, dy_1 (-y_1+f_1(y_2))^{-1/2} (y_1 - f(y_2))^{-1/2} 
\le c x_1^{-3/2}, 
\end{eqnarray}
\begin{eqnarray}
\label{phi11D2}
&& \left|\int_{D_2} E((x_1,0),y) \vp_1(y) \, dy\right| \\
\nonumber
&\le&  c x_1^{1/2} \int_{x_1}^{r_1/2} \, dy_2 \, y_2^{-4} \int_{f(y_2)}^{f_1(y_2)} \, dy_1 (-y_1+f_1(y_2))^{-1/2} (y_1 - f(y_2))^{-1/2} (x_1 + y_1) \\
\nonumber
&\le& c x_1^{-3/2}, 
\end{eqnarray}
(here we used the estimate $y_1 \le c y_2^2$).
$$
\left|\int_{D_3} E((x_1,0),y) \vp_1(y)\, dy \right|
\le c x_1^{1/2} \int_{D_3} \delta_B^{-1/2}(y) \delta_D^{-1/2}(y) \, dy \le c x_1^{1/2}.
$$
It follows that $|\vp_{11}(x_1,0)| \le c x_1^{-3/2}$, where $c$ does not depend on $s$ and $x_1 = \sqrt{r - s}$. Since $s \in (r - (r_1/2)^2,r)$ we get $|\vp_{11}(x_1,0)| \le c x_1^{-3/2}$, $x_1 \in (0,r_1/2]$.

Now we will show that $\vp_{11}(x_1,0) \le -c x_1^{-3/2}$ for $x_1 \in (0,r_2]$. Here we will use notation similar to the notation used in the proof of Lemma \ref{phi1}. We will use (\ref{phiformula}) for $s = r$, in particular $B = B(z,r)$. By (\ref{phiformula}) we get for $x_1 \in (0,r]$
\begin{eqnarray*}
&& \vp_{11}(x_1,0)
= h_{11}(x_1,0) + \int_{D \setminus B} P_{11}((x_1,0),y) \vp(y) \, dy\\
&=& h_{11}(x_1,0) 
+ \int_{D \setminus B} \frac{\partial A}{\partial x_1}((x_1,0),y) \vp(y) \, dy
+ \int_{D \setminus B} \frac{\partial E}{\partial x_1}((x_1,0),y) \vp(y) \, dy.
\end{eqnarray*}
One easily gets $h_{11}(x_1,0) \approx - x_1^{-3/2}$ for $x_1 \in (0,r/4]$. For $x \in B$, $y \in (\overline{B})^c$ we have
\begin{eqnarray*}
\frac{\partial A}{\partial x_1}(x,y) &=&
\frac{-C_{P}(r^2 - |x - z|^2)^{-3/2}(x_1 - r)^2}{(|y - z|^2 - r^2)^{1/2} |x-y|^2}
+ \frac{-C_{P}(r^2 - |x - z|^2)^{-1/2}}{(|y - z|^2 - r^2)^{1/2} |x-y|^2} \\
&+& \frac{-2 C_{P}(r^2 - |x - z|^2)^{-1/2}(r - x_1)(x_1 - y_1)}{(|y - z|^2 - r^2)^{1/2} |x-y|^4}\\
&=& A^{(1)}(x,y) + A^{(2)}(x,y) + A^{(3)}(x,y),
\end{eqnarray*}
\begin{eqnarray*}
\frac{\partial E}{\partial x_1}(x,y) &=&
\frac{-2 C_{P}(r^2 - |x - z|^2)^{-1/2}(r - x_1) (x_1 - y_1)}{(|y - z|^2 - r^2)^{1/2} |x-y|^4}
+ \frac{-2C_{P}(r^2 - |x - z|^2)^{1/2}}{(|y - z|^2 - r^2)^{1/2} |x-y|^4} \\
&+& \frac{8C_{P}(r^2 - |x - z|^2)^{1/2}(x_1 - y_1)^2}{(|y - z|^2 - r^2)^{1/2} |x-y|^6}\\
&=& E^{(1)}(x,y) + E^{(2)}(x,y) + E^{(3)}(x,y).
\end{eqnarray*}

Let $x_1 \in (0,r/8]$, $y \in (\overline{B})^c$. We have $A^{(1)}(x,y) \le 0$, $A^{(2)}(x,y) \le 0$. We also have $A^{(3)}(x,y) \ge 0$ iff $y_1 \ge x_1$. Let $f_1$ be such as in the proof of Lemma \ref{phi1}. Let $p'_1 > 0$ be such that $f_1(p'_1) = x_1$, $p'_2 < 0$ be such that $f_1(p'_2) = x_1$ (we have $p'_2 = - p'_1$). Note that $p'_1 \approx \sqrt{x_1}$, $|p'_2| \approx \sqrt{x_1}$. 
Note also that $f_1(r/2) = r (1 - \sqrt{3}/2) > r/8$ and $f_1(p'_1) = x_1 \le r/8$ so $p'_1 < r/2$.
Put 
\begin{eqnarray*}
D'_1 &=& \{(y_1,y_2): \, y_2 \in [p'_2,p'_1], y_1 \in (f(y_2),f_1(y_2))\}, \\
D'_2 &=& \{(y_1,y_2): \, y_2 \in (p'_1,r/2] \cup [-r/2,p'_2), y_1 \in (f(y_2),f_1(y_2))\}, \\
D'_3 &=& D \setminus (D'_1 \cup D'_2 \cup B).
\end{eqnarray*}
We have $\int_{D'_1} A^{(3)}((x_1,0),y) \vp(y) \, dy \le 0$. Note that for $y \in D'_2$ we have $y_1 \le f_1(y_2) \le c y_2^2$, $\vp(y) \le c \delta_D^{1/2}(y) \le c (y_2^2)^{1/2} = c y_2$. Hence
\begin{eqnarray*}
\int_{D'_2} A^{(3)}((x_1,0),y) \vp(y) \, dy 
&\le& c x_1^{-1/2} \int_{c \sqrt{x_1}}^{r/2} \, dy_2 \, y_2^{-4} \int_{f(y_2)}^{f_1(y_2)} \, dy_1 \, (y_1 - f_1(y_2))^{-1/2} y_1 \vp(y) \\
&\le& c x_1^{-1/2} \int_{c \sqrt{x_1}}^{r/2} \, dy_2 \le c x_1^{-1/2},
\end{eqnarray*}
$$
\left|\int_{D'_3} A^{(3)}((x_1,0),y) \vp(y) \, dy \right|
\le c x_1^{-1/2} \int_{D'_3} \delta_B^{-1/2}(y) \, dy \le c x_1^{-1/2}.
$$
Note that $E^{(1)}(x,y) = A^{(3)}(x,y)$ and $E^{(2)}(x,y) \le 0$. To estimate $\int_{D \setminus B} E^{(3)} \vp$ we put  
\begin{eqnarray*}
D''_1 &=& \{(y_1,y_2): \, y_2 \in [-x_1,x_1], y_1 \in (f(y_2),f_1(y_2))\}, \\
D''_2 &=& \{(y_1,y_2): \, y_2 \in (x_1,r/2] \cup [-r/2,-x_1), y_1 \in (f(y_2),f_1(y_2))\}, \\
D''_3 &=& D \setminus (D''_1 \cup D''_2 \cup B).
\end{eqnarray*}
Note that for $y \in D''_1$ we have $(x_1 - y_1)^2 \le x_1^2$, $\vp(y) \le c \delta_D^{1/2}(y) \le c x_1$ so 
\begin{eqnarray*}
\int_{D''_1} E^{(3)}((x_1,0),y) \vp(y) \, dy 
&\le& c x_1^{-7/2} \int_{-x_1}^{x_1} \, dy_2 \int_{f(y_2)}^{f_1(y_2)} \, dy_1 \, (y_1 - f_1(y_2))^{-1/2} \vp(y) \\
&\le&  c x_1^{-1/2}.
\end{eqnarray*}
Note that for $y \in D''_2$ we have $(x_1 - y_1)^2 \le x_1^2 + y_1^2 \le x_1^2 + c y_2^4$ and $\vp(y) \le c \delta_D^{1/2}(y) \le c y_2$ so
\begin{eqnarray*}
&& \int_{D''_2} E^{(3)}((x_1,0),y) \vp(y) \, dy \\
&\le& c x_1^{1/2} \int_{x_1}^{r/2} \, dy_2 \, y_2^{-6} (x_1^2 + y_2^4) \int_{f(y_2)}^{f_1(y_2)} \, dy_1 \, (y_1 - f_1(y_2))^{-1/2} \vp(y) \\
&\le& c x_1^{5/2} \int_{x_1}^{r/2} y_2^{-4} \, dy_2 
+ c x_1^{1/2} \int_{x_1}^{r/2} \, dy_2 
\le  c x_1^{-1/2}.
\end{eqnarray*}
We also have $\int_{D''_3} E^{(3)}((x_1,0),y) \vp(y) \, dy \le  c x_1^{1/2}$.

It follows that for sufficiently small $x_1$ we have $\vp_{11}(x_1,0) \le -c x_1^{-3/2}$.
\end{proof}

\begin{lemma}
\label{phi12} There exists $r_2 \in (0,r_0/4]$, $r_2 = r_2(\Lambda)$ such that for any $x_1 \in (0,r_2]$ we have $|\vp_{12}(x_1,0)| \le c x_1^{-1/2} |\log x_1|$.
\end{lemma}
\begin{proof}
We will use similar notation as in Lemma \ref{phi22}. Put $r = r_0$. Let $r_1$ be the constant from Corollary \ref{nT}. We take $s \in (r - (r_1/2)^2,r)$. Recall that $z = (r,0)$, $B = B(z,s)$ and $P$ is given by (\ref{Pformula}). For any $x_1 \in (r - s, r]$ by Lemma \ref{12harmonic} we have $\vp_2(x_1,0) = \int_{D \setminus B} P((x_1,0),y) \vp_2(y) \, dy$. It follows that 
\begin{eqnarray*}
&& \vp_{12}(x_1,0)
= \int_{D \setminus B} P_{1}((x_1,0),y) \vp_2(y) \, dy\\
&=&  
\int_{D \setminus B} A((x_1,0),y) \vp_2(y) \, dy
+ \int_{D \setminus B} E((x_1,0),y) \vp_2(y) \, dy.
\end{eqnarray*}
Take $x_1 = \sqrt{r - s}$ (we have $\sqrt{r - s} < r_1/2 \le r/8$). 
We have
$$
\int_{D \setminus B} A((x_1,0),y) \vp_2(y) \, dy 
= \frac{r - x_1}{(s^2 - (x_1 - r)^2)} \int_{D \setminus B} P((x_1,0),y) \vp_2(y) \, dy.
$$
By Lemma \ref{phi2} we get 
$$
\left| \int_{D \setminus B} P((x_1,0),y) \vp_2(y) \, dy \right| 
= |\vp_2(x_1,0)| \le c x_1^{1/2} |\log x_1|.
$$
Since $(r - x_1)(s^2 - (x_1 - r)^2)^{-1} \approx x_1^{-1}$ we obtain
$$
\left| \int_{D \setminus B} A((x_1,0),y) \vp_2(y) \, dy \right|
\le c x_1^{-1/2} |\log x_1|,
$$
for $x_1 = \sqrt{r - s}$.

Let $f_1$, $D_1$, $D_2$, $D_3$ be such as in the proof of 
Lemma \ref{phi22}. 
By Lemma \ref{ball} we have for $y \in D_1 \cup D_2$
$$
\vp_2(y) = \cos \alpha(y) \frac{\partial \vp}{\partial \vec{T}}(y) 
-\sin \alpha(y) \frac{\partial \vp}{\partial \vec{n}}(y).
$$
By the arguments from the proof of Lemma \ref{phi22} we have for $y \in D_1 \cup D_2$
\begin{eqnarray*}
\left|\cos \alpha(y)\frac{\partial \vp}{\partial \vec{T}}(y)\right| 
&\le& c (y_1 - f(y_2))^{1/2} |\log(y_1 - f(y_2))| \\
&\le& c y_1^{1/2} |\log y_1|,
\end{eqnarray*} 
$$
\left|\sin \alpha(y) \frac{\partial \vp}{\partial \vec{n}}(y) \right| \le c y_2 (y_1 - f(y_2))^{-1/2}.
$$
Similarly like in the proofs of Lemmas \ref{phi22} and \ref{phi11} we obtain the following estimates
\begin{eqnarray*}
&& \left| \int_{D_1} E((x_1,0),y) \cos \alpha(y) \frac{\partial \vp}{\partial \vec{T}}(y) \, dy \right| \\
&\le& c x_1^{-5/2} \int_{-x_1}^{x_1} \, dy_2 \int_{f(y_2)}^{f_1(y_2)} \, dy_1 (-y_1+f_1(y_2))^{-1/2}  y_1^{1/2} |\log y_1| 
\le c x_1^{1/2} |\log x_1|.
\end{eqnarray*}
Here we used the following facts $y_1^{1/2} |\log y_1| \le c y_2 |\log y_2| \le c x_1 |\log x_1|$, $\int_{f(y_2)}^{f_1(y_2)}  (-y_1+f_1(y_2))^{-1/2} \, dy_1 \le c f_1^{1/2}(y_2) \le c y_2 \le c x_1$.

Using similar arguments we get 
\begin{eqnarray*}
&& \left| \int_{D_2} E((x_1,0),y) \cos \alpha(y) \frac{\partial \vp}{\partial \vec{T}}(y) \, dy \right| \\
&\le& c x_1^{1/2} \int_{x_1}^{r/2} \, dy_2 \, y_2^{-4} \int_{f(y_2)}^{f_1(y_2)} \, dy_1 (-y_1+f_1(y_2))^{-1/2} y_1^{1/2} |\log y_1| (x_1 + y_1)\\ 
&\le& c x_1^{1/2} |\log x_1|.
\end{eqnarray*}
By the same arguments as in (\ref{phi11D1}), (\ref{phi11D2}) one can easily obtain 
$$
\left| \int_{D_1} E((x_1,0),y) y_2 (y_1 - f(y_2))^{-1/2} \, dy \right|
\le c x_1^{-1/2},
$$
$$
\left| \int_{D_2} E((x_1,0),y) y_2 (y_1 - f(y_2))^{-1/2} \, dy \right|
\le c x_1^{-1/2} + c x_1^{1/2} |\log x_1|,
$$
We also have
$$
\left|\int_{D_3} E((x_1,0),y) \vp_2(y)\, dy \right|
\le c x_1^{1/2} \int_{D_3} \delta_B^{-1/2}(y) \delta_D^{-1/2}(y) \, dy \le c x_1^{1/2}.
$$
It follows that $|\vp_{12}(x_1,0)| \le c x_1^{-1/2} |\log x_1|$,
where $c$ does not depend on $s$ and $x_1 = \sqrt{r-s}$. Since $s \in (r - (r_1/2)^2,r)$ we get $|\vp_{12}(x_1,0)| \le c x_1^{-1/2} |\log x_1|$, $x_1 \in (0,r_1/2]$.
\end{proof}

By Lemmas \ref{ball}, \ref{phi22}, \ref{phi11}, \ref{phi12}  and Corollary \ref{nT}  we obtain
\begin{corollary}
\label{nT1}
There exists $r_2 \in (0,r_0/4]$, $r_2 = r_2(\Lambda)$ such that for any $y \in D$, $\delta_D(y) \le r_2$ we have (\ref{phin}), (\ref{phit}), (\ref{phigradient}) and
\begin{eqnarray*}
\frac{\partial^2 \vp}{\partial \vec{n}^2}(y) 
&\approx& -\delta_D^{-3/2}(y),\\
\frac{\partial^2 \vp}{\partial \vec{T}^2}(y) 
&\approx& -\delta_D^{-1/2}(y),\\
\left|\frac{\partial^2 \vp}{\partial \vec{n} \partial \vec{T}}(y) \right|
&\le& c \delta_D^{-1/2}(y) |\log(\delta_D(y))|.
\end{eqnarray*}
\end{corollary}

The coordinate system and notation in the following lemma is the same as in the whole section.

\begin{lemma}
\label{phiderivatives}
There exists $r_3 \in (0,r_0/4]$, $r_3 = r_3(\Lambda)$ such that for any $y = (y_1,y_2) \in B((r_3,0),r_3)$ we have
\begin{eqnarray}
\label{phi2der}
|\vp_2(y)| &\le&
c (y_1^{1/2} |\log y_1| + |y_2| y_1^{-1/2}),\\ 
\label{phi12der}
|\vp_{12}(y)| &\le&
c (y_1^{-1/2} |\log y_1| + |y_2| y_1^{-3/2}),\\
\label{phi22der}
|\vp_{22}(y)| &\approx& -y_1^{-1/2} 
\end{eqnarray}
and for any $y = (y_1,y_2) \in W_{r_3}$ we have
\begin{equation}
\label{derphi1}
\vp_1(y) \approx \delta_D^{-1/2}(y),
\end{equation}
where $W_{r_3} = \{(y_1,y_2): \, y_2 \in [-r_3,r_3], y_1 \in (f(y_2),r_3]\}$.
\end{lemma}
\begin{proof}
We may assume that $y_2 > 0$.
Let $r \in (0,r_2]$ where $r_2$ is the constant from Corollary \ref{nT1} (recall that $r_2 \le r_0/4$). Let $y = (y_1,y_2) \in B((r,0),r)$ with $y_2 > 0$. By Lemma \ref{ball} we have $\sin \alpha(y) \approx y_2$, $\cos \alpha(y) \approx c$. We also have $\delta_D(y) \approx y_1$ and $y_2^2 \le c y_1$.

By Corollary \ref{nT1} we get $\frac{\partial \vp}{\partial \vec{n}}(y) \approx -\delta_D^{-1/2}(y) \approx - y_1^{-1/2}$,
$\left| \frac{\partial \vp}{\partial \vec{T}}(y) \right| \le c \delta_D^{1/2}(y) |\log(\delta_D(y))| \le c y_1^{1/2} |\log y_1|$. Using this and the formula for $\vp_2$ from Lemma \ref{ball} we get (\ref{phi2der}).

By Corollary \ref{nT1} we have 
$$
\left| \frac{\partial^2 \vp}{\partial \vec{n} \partial \vec{T}}(y) \right| \le  c \delta_D^{-1/2}(y) |\log(\delta_D(y))| 
\le c y_1^{-1/2} |\log y_1|,
$$
$$
\left| \frac{\partial^2 \vp}{\partial \vec{n}^2}(y) -
\frac{\partial^2 \vp}{\partial \vec{T}^2}(y) \right| 
\le  c \delta_D^{-3/2}(y) 
\le c y_1^{-3/2}.
$$
Using this and the formula for $\vp_{12}$ from Lemma \ref{ball} we get (\ref{phi12der}).

By Corollary \ref{nT1} we have $\frac{\partial^2 \vp}{\partial \vec{T}^2}(y) \approx -\delta_D^{-1/2}(y)  \approx - y_1^{-1/2}$, $\frac{\partial^2 \vp}{\partial \vec{n}^2}(y) \approx -\delta_D^{-3/2}(y) \approx - y_1^{-3/2}$. $\sin^2 \alpha(y) \approx y_2^2 \le c y_1$,
$$
\left| \sin \alpha(y) \cos \alpha(y) \frac{\partial^2 \vp}{\partial \vec{n} \partial \vec{T}}(y) \right|    
\le c y_2 y_1^{-1/2} |\log y_1| \le c |\log y_1|.
$$
Using this and the formula for $\vp_{22}$ from Lemma \ref{ball} we get (\ref{phi22der}) for sufficiently small $r$.

By (\ref{phin}), (\ref{phit}) and the formula for $\vp_1$ from Lemma \ref{ball} we get (\ref{derphi1}) for sufficiently small $r$.
\end{proof}

We have $(-\Delta)^{1/2} \vp(x) = 1$ for $x \in D$. We need to estimate $(-\Delta)^{1/2} \vp(x) $ for $x \in (\overline{D})^c$. For such $x$ we have  $(-\Delta)^{1/2} \vp(x) = -(2 \pi)^{-1} \int_D \frac{\vp(y)}{|y - x|^3} \, dy$.

\begin{lemma}
\label{DeltaSqrt}
Let $x = (-x_1,0)$, $x_1 > 0$. We have
$$
\left|(-\Delta)^{1/2} \vp(x)\right| \approx \delta_D^{-1/2}(x) (1 + |x|)^{-5/2}.
$$
\end{lemma}
\begin{proof}
Put $r = r_0$.
When $x_1 \in (-\infty,-r/2)$ we have
$$
\int_D \frac{\vp(y)}{|y - x|^3} \, dy \approx |x|^{-3} \approx \delta_D^{-1/2}(x) (1 + |x|)^{-5/2}.
$$
When $x_1 \in [-r/2,0)$ we have
\begin{eqnarray*}
\int_D \frac{\vp(y)}{|y - x|^3} \, dy &\approx&
\int_{D \cap B(0,\delta_D(x))} \delta_D^{-5/2}(x) \, dy +
\int_{D \cap (B(0,r/2) \setminus B(0,\delta_D(x)))} |y|^{-5/2} \, dy \\
&& + 
\int_{D \cap B^c(0,r/2)} |y|^{-5/2} \, dy 
\approx \delta_D^{-1/2}(x).
\end{eqnarray*}
\end{proof}
By Lemma \ref{DeltaSqrt} we obtain immediately
\begin{corollary}
\label{DeltaSqrt1}
For any $x \in (\overline{D})^c$ we have
$$
\left|(-\Delta)^{1/2} \vp(x)\right| \approx \delta_D^{-1/2}(x) (1 + |x|)^{-5/2}.
$$
\end{corollary}

\section{Estimates of derivatives of $u$ near $\partial D \times \{0\}$}

In this section we study the behaviour of $u_{i,j}$ near $\partial D \times \{0\}$.

\begin{figure}
\centering
\includegraphics[scale=0.7]{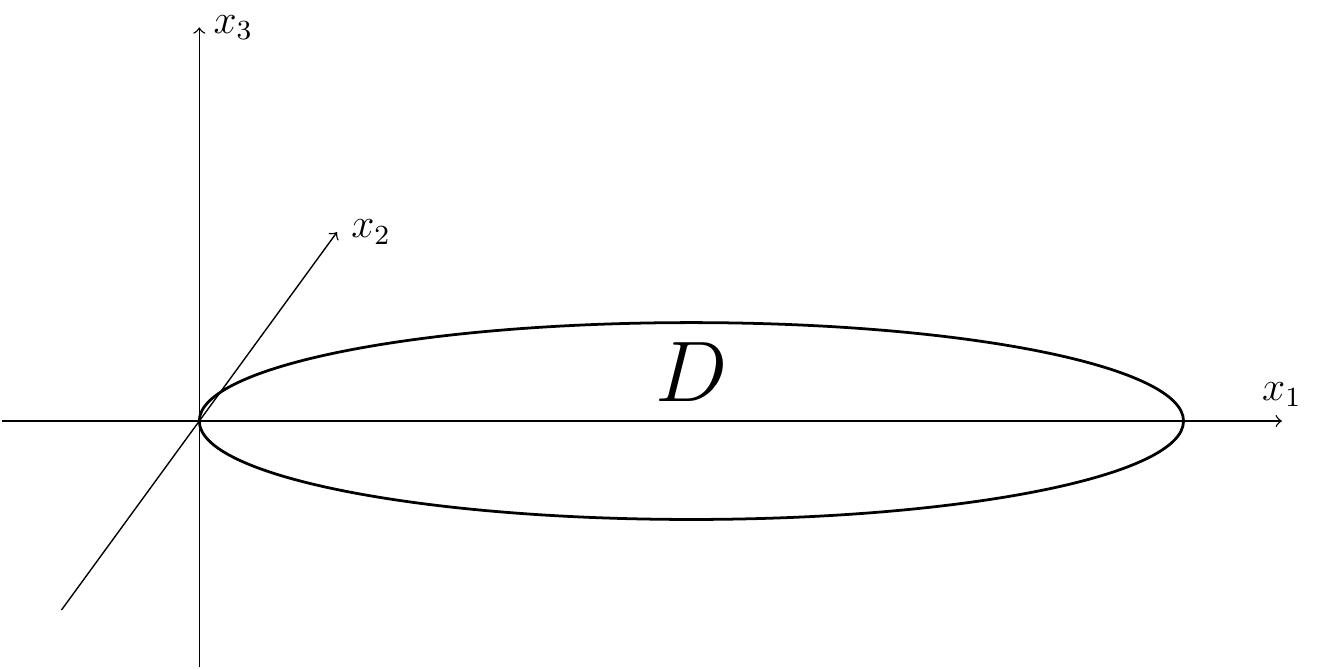}
\caption{}
\label{fig:2}
\end{figure}

In the whole section we fix $C_1 > 0$, $R_1 > 0$, $\kappa_2 \ge \kappa_1 > 0$, $D \in F(C_1,R_1,\kappa_1,\kappa_2)$ and $x_0 \in \partial D$. We put $\Lambda = \{C_1,R_1,\kappa_1,\kappa_1\}$. $\vp$ is the solution of (\ref{maineq1}-\ref{maineq2}) for $D$ and $u$ is the harmonic extension of $\vp$ given by (\ref{ext1}-\ref{ext2}). Unless it is otherwise stated we fix a $2$-dimensional coordinate system $CS_{x_0}$ and notation as in Lemma \ref{ball} (see Figure 1). In particular $x_0$ is $(0,0)$ in $CS_{x_0}$ coordinates. To study $u$ we also use a $3$-dimensional Cartesian coordinate system $0 x_1 x_2 x_3$, see Figure 2, which is formed (roughly speaking) by adding $0 x_3$ axis to the above $2$-dimensional coordinate system. Let us recall that in the whole section we use convention stated in Remark \ref{constants1}.

Put $r = r_2 \wedge r_3 \wedge f(r_0/4) \wedge f(-r_0/4)$, where $r_0$, $r_2$, $r_3$ are the constant from Lemma \ref{ball}, Corollary \ref{nT1} and Lemma \ref{phiderivatives}. Note that $f(r_0/4) \wedge f(-r_0/4) \ge c_3 r_0^2/16$, where $c_3$ is a constant from Lemma \ref{ball}, $c_3 r_0^2/16$ depends only on $\Lambda$. For any $h \in (0,r]$ we put (see Figure 3):
\begin{eqnarray*}
S_1(h) &=& \{(x_1,x_2,x_3): \, x_1 = -h, x_2 = 0, x_3 \in (0,h/4]\},\\
S_2(h) &=& \{(x_1,x_2,x_3): \, x_1 = -h, x_2 = 0, x_3 \in (h/4,h]\} \\
&& \cup \, \, \{(x_1,x_2,x_3): \, x_1 \in (-h,0], x_2 = 0, x_3 = h \},\\
S_3(h) &=& \{(x_1,x_2,x_3): \, x_1 \in (0,h], x_2 = 0, x_3 = h\} \\
&& \cup \, \, \{(x_1,x_2,x_3): \, x_1 = h, x_2 = 0, x_3 \in (h/4,h]\},\\
S_4(h) &=& \{(x_1,x_2,x_3): \, x_1 = h, x_2 = 0, x_3 \in (0,h/4]\}.
\end{eqnarray*}

\begin{figure}
\centering
\includegraphics[scale=0.7]{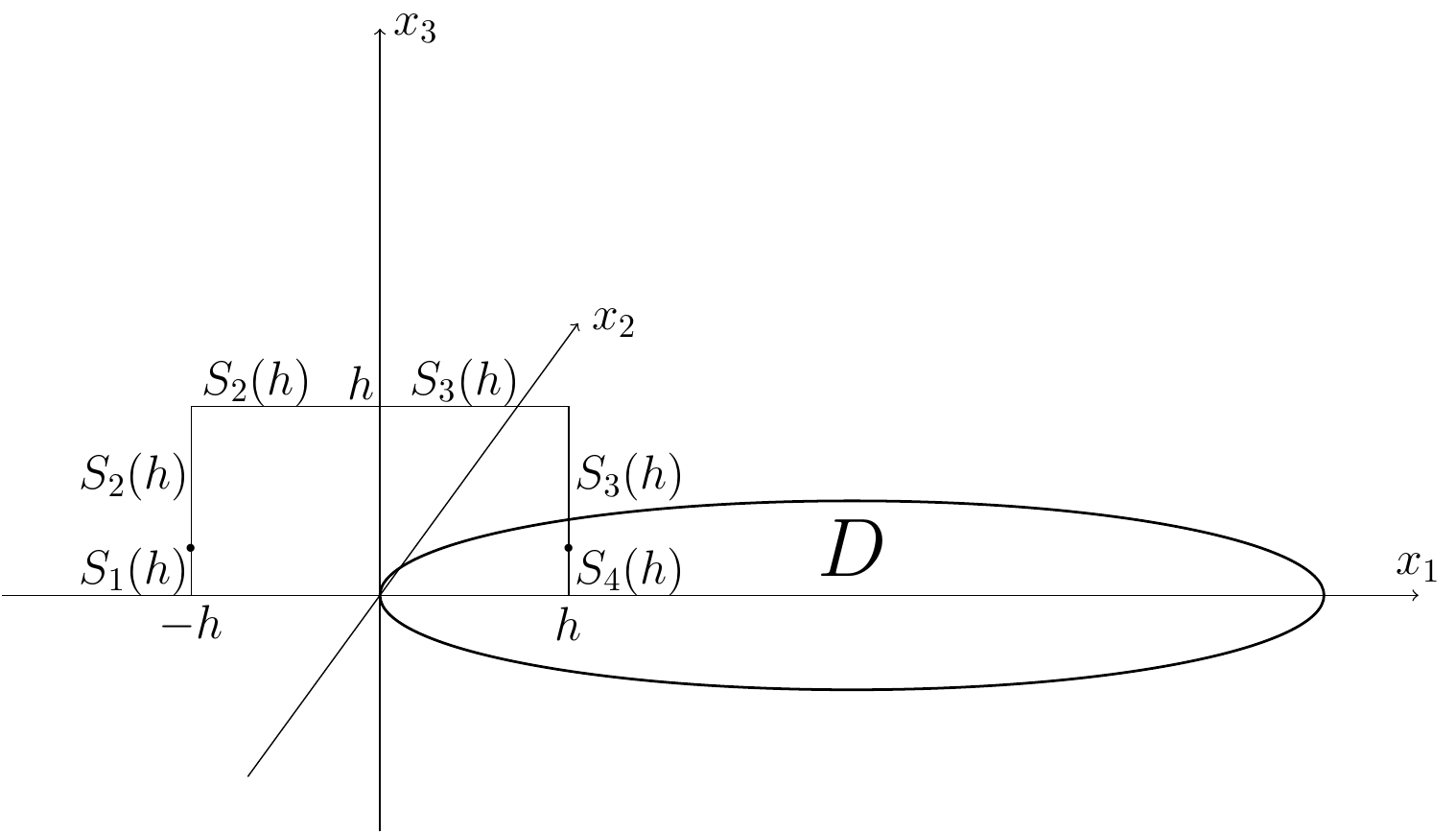}
\caption{}
\label{fig:3}
\end{figure}

The main result of this section is the following proposition.

\begin{proposition}
\label{Hessianboundary}
There exists $h_0 \in (0,r/4]$, $h_0 = h_0(\Lambda)$ such that for any $h \in (0,h_0]$ we have
\begin{eqnarray*}
&& u_{22}(x) \approx -x_3 h^{-3/2} \,\,\, \text{for} \,\,\, x \in S_1(h) \cup S_2(h) \cup S_3(h), \,\,\,
u_{22}(x) \approx -h^{-1/2} \,\,\, \text{for} \,\,\, x \in S_4(h),\\
&& u_{11}(x) \approx h^{-3/2} \,\,\, \text{for} \,\,\, x \in S_2(h), \quad
u_{11}(x) \approx -h^{-3/2} \,\,\, \text{for} \,\,\, x \in S_4(h), \\
&& |u_{11}(x)| \le c x_3 h^{-5/2} \,\,\, \text{for} \,\,\, x \in S_1(h) \cup S_3(h),\\
&& u_{33}(x) \approx -h^{-3/2} \,\,\, \text{for} \,\,\, x \in S_2(h), \quad
u_{33}(x) \approx h^{-3/2} \,\,\, \text{for} \,\,\, x \in S_4(h), \\
&& |u_{33}(x)| \le c x_3 h^{-5/2} \,\,\, \text{for} \,\,\, x \in S_1(h) \cup S_3(h),
\end{eqnarray*}
\begin{eqnarray*}
&& u_{13}(x) \approx h^{-3/2} \,\,\, \text{for} \,\,\, x \in S_1(h), \quad
u_{13}(x) \approx -h^{-3/2} \,\,\, \text{for} \,\,\, x \in S_3(h), \\
&& |u_{13}(x)| \le c h^{-3/2} \,\,\, \text{for} \,\,\, x \in S_2(h) \cup S_4(h), \quad
u_{13}(x) \le - c x_3 h^{-5/2} \,\,\, \text{for} \,\,\, x \in S_4(h),\\
&& |u_{12}(x)| \le c x_3 h^{-3/2} |\log h| \,\,\, \text{for} \,\,\, x \in S_1(h), \\
&& |u_{12}(x)| \le c h^{-1/2} |\log h| \,\,\, \text{for} \,\,\, x \in S_2(h) \cup S_3(h) \cup S_4(h),\\
&& |u_{23}(x)| \le c h^{-1/2} |\log h| \,\,\, \text{for} \,\,\, x \in S_1(h) \cup S_2(h) \cup S_3(h), \\
&& |u_{23}(x)| \le c h^{-3/4} |\log h| \,\,\, \text{for} \,\,\, x \in S_4(h).
\end{eqnarray*}
\end{proposition}
\begin{proof}
Let $h \in (0,r/8]$. Let us define $f_1: \, [-r,r] \to \R$ by $f_1(y_2) = r - \sqrt{r^2 - y_2^2}$ and $g_1: \, [-r,r] \to \R$ by $g_1(y_1) = \sqrt{r^2 - (y_1 - r)^2}$. 

\vskip 5pt
{\bf{Step 1.}} Estimate $u_{22}(x) \approx -x_3 h^{-3/2}$ for $x \in S_1(h) \cup S_2(h) \cup S_3(h)$.

We have
\begin{equation}
\label{u22step1}
u_{22}(x) = \int_D K_2(x_1-y_1,-y_2,x_3) \vp_2(y_1,y_2) \, dy_1 \, dy_2.
\end{equation}
Put 
\begin{eqnarray*}
D_1 &=& \{(y_1,y_2): \, y_1 \in [f_1(h),h], y_2 \in [-g_1(y_1),g_1(y_1)]\},\\
D_2 &=& \{(y_1,y_2): \, y_1 \in (h,r], y_2 \in [-g_1(y_1),g_1(y_1)]\},\\
D_3 &=& \{(y_1,y_2): \, y_2 \in [-h,h], y_1 \in (f(y_2),f_1(h))\},\\
D_4 &=& \{(y_1,y_2): \, y_2 \in [-r/2,-h] \cup [h,r/2], y_1 \in (f(y_2),f_1(y_2))\},\\
D_5 &=& D \setminus (D_1 \cup D_2 \cup D_3 \cup D_4).
\end{eqnarray*}
For $i = 1,2,3,4$ we also put $D_{i+} = \{(y_1,y_2) \in D_i: \, y_2 > 0$, $D_{i-} = \{(y_1,y_2) \in D_i: \, y_2 < 0\}$.

Note that $f_1(h) \le h^2/r \le h/4$. 

We will estimate (\ref{u22step1}). The most important is $\int_{D_1 \cup D_2} K_2 \vp_2$. By Lemma \ref{phiderivatives} for $y \in D_{1+} \cup D_{2+}$ we have
$\vp_2(y_1,y_2) - \vp_2(y_1,-y_2) = 2 y_2 \vp_{22}(y_1,\xi) \approx - y_2 y_1^{-1/2}$, where $\xi \in (-y_2,y_2)$. It follows that 
\begin{eqnarray*}
&& \int_{D_1 \cup D_2} K_2(x_1-y_1,-y_2,x_3) \vp_2(y_1,y_2) \, dy_1 \, dy_2\\
&=& c x_3 \int_{D_{1+} \cup D_{2+}} \frac{y_2}{((x_1 - y_1)^2 + y_2^2 + x_3^2)^{5/2}}
(\vp_2(y_1,y_2) - \vp_2(y_1,-y_2)) \, dy_1 \, dy_2\\
&\approx& c x_3 \int_{D_{1+} \cup D_{2+}} \frac{-y_2^2 y_1^{-1/2}}{((x_1 - y_1)^2 + y_2^2 + x_3^2)^{5/2}} \, dy_1 \, dy_2.
\end{eqnarray*}
We have 
\begin{eqnarray*}
&& \int_{D_{1+}} \frac{-y_2^2 y_1^{-1/2}}{((x_1 - y_1)^2 + y_2^2 + x_3^2)^{5/2}} \, dy_1 \, dy_2\\
&\approx& \frac{1}{h^5} \int_{f_1(h)}^h dy_1 \, y_1^{-1/2} \int_0^{h} dy_2 \, (-y_2^2)
+ \int_{f_1(h)}^h dy_1 \, y_1^{-1/2} \int_h^{g_1(y_1)} dy_2 \, \frac{-y_2^2}{y_2^5}\\
&\approx& - h^{-3/2}.
\end{eqnarray*}
We also have
\begin{eqnarray*}
&& \int_{D_{2+}} \frac{-y_2^2 y_1^{-1/2}}{((x_1 - y_1)^2 + y_2^2 + x_3^2)^{5/2}} \, dy_1 \, dy_2\\
&\approx& \int_{h}^r dy_1 \int_0^{y_1} dy_2 \, \frac{-y_2^2 y_1^{-1/2}}{y_1^5}
+ \int_{h}^r dy_1 \int_{y_1}^{g_1(y_1)} dy_2 \, \frac{-y_2^2 y_1^{-1/2}}{y_2^5}\\
&\approx& - h^{-3/2}.
\end{eqnarray*}
It follows that $\int_{D_1 \cup D_2} K_2 \vp_2 \approx - x_3 h^{-3/2}$.

Now we will estimate $\int_{D_3 \cup D_4} K_2 \vp_2$. It is sufficient to estimate $\int_{D_{3+} \cup D_{4+}} K_2 \vp_2$. The estimate $\int_{D_{3-} \cup D_{4-}} K_2 \vp_2$ is the same. By Lemma \ref{ball} and Corollary \ref{nT1} we get for $y \in D_{3+} \cup D_{4+}$
\begin{eqnarray*}
|\vp_2(y)| &=& \left|\cos \alpha(y) \frac{\partial \vp}{\partial \vec{T}}(y)-\sin \alpha(y) \frac{\partial \vp}{\partial \vec{n}}(y) \right|\\
&\le& c \delta_D^{1/2}(y) |\log \delta_D(y)| + c y_2 \delta_D^{-1/2}(y)\\
&\le& c(f^{-1}(y_1) - y_2)^{1/2} (f^{-1}(y_1))^{1/2} |\log((f^{-1}(y_1) - y_2) f^{-1}(y_1))|\\
&+& c y_2 (f^{-1}(y_1) - y_2)^{-1/2} (f^{-1}(y_1))^{-1/2}.
\end{eqnarray*}
It follows that 
\begin{eqnarray*}
&& \left|\int_{D_{3+}} K_2(x_1-y_1,-y_2,x_3) \vp_2(y_1,y_2) \, dy_1 \, dy_2 \right|\\
&\le& \frac{c x_3}{h^5} \int_0^{f_1(h)} \, dy_1 \int_0^{f^{-1}(y_1)} \, dy_2
y_2 |\vp_2(y_1,y_2)|\\
&\le& \frac{c x_3}{h^5} \int_0^{f_1(h)} \, dy_1 \int_0^{f^{-1}(y_1)} \, dy_2
(f^{-1}(y_1) - y_2)^{1/2} (f^{-1}(y_1))^{1/2} \\
&& \quad \quad \quad \quad \quad \quad \quad \quad \quad \quad \quad \quad \times
|\log((f^{-1}(y_1) - y_2) f^{-1}(y_1))| y_2\\
&+& \frac{c x_3}{h^5} \int_0^{f_1(h)} \, dy_1 \int_0^{f^{-1}(y_1)} \, dy_2
(f^{-1}(y_1) - y_2)^{-1/2} (f^{-1}(y_1))^{-1/2} y_2^2.
\end{eqnarray*}
By substituting $w = f^{-1}(y_1) - y_2$ and using $y_2 = f^{-1}(y_1) - w \le f^{-1}(y_1)$, $f^{-1}(y_1) \approx y_1^{1/2}$, $f_1(h) \le c h^2$ this is bounded from above by
\begin{eqnarray*}
&& \frac{c x_3}{h^5} \int_0^{f_1(h)} \, dy_1 \int_0^{f^{-1}(y_1)} \, dw
w^{1/2} (f^{-1}(y_1))^{3/2} |\log(w f^{-1}(y_1))| \\
&+& \frac{c x_3}{h^5} \int_0^{f_1(h)} \, dy_1 \int_0^{f^{-1}(y_1)} \, dw 
w^{-1/2} (f^{-1}(y_1))^{3/2}\\
&\le& c x_3 |\log h| + c x_3 h^{-1}.
\end{eqnarray*}
In the same way we get 
\begin{eqnarray*}
&& \left|\int_{D_{4+}} K_2(x_1-y_1,-y_2,x_3) \vp_2(y_1,y_2) \, dy_1 \, dy_2 \right|\\
&\le& c x_3 \int_h^{r/2} \, dy_2 \int_{f(y_2)}^{f_1(y_2)} \, dy_1
\frac{y_2}{y_2^5} |\vp_2(y_1,y_2)|\\
&\le& c x_3 \int_{f(h)}^{f_1(r/2)} \, dy_1 \int_{g_1(y_1)}^{f^{-1}(y_1)} \, dy_2
y_2^{-4}(f^{-1}(y_1) - y_2)^{1/2} (f^{-1}(y_1))^{1/2} \\
&& \quad \quad \quad \quad \quad \quad \quad \quad \quad \quad \quad \quad \times
|\log((f^{-1}(y_1) - y_2) f^{-1}(y_1))| \\
&+& c x_3 \int_{f(h)}^{f_1(r/2)} \, dy_1 \int_{g_1(y_1)}^{f^{-1}(y_1)} \, dy_2
y_2^{-3}(f^{-1}(y_1) - y_2)^{-1/2} (f^{-1}(y_1))^{-1/2}.
\end{eqnarray*}
Similarly like in the estimate $\int_{D_{3+}} K_2 \vp_2$ using substitution $w = f^{-1}(y_1) - y_2$ we obtain that it is bounded from above by $c x_3 |\log h|^2 + c x_3 h^{-1}$. We also have 
$$
\left|\int_{D_{5}} K_2(x_1-y_1,-y_2,x_3) \vp_2(y_1,y_2) \, dy_1 \, dy_2 \right|
\le c x_3 \int_{D_5} \delta_D^{-1/2}(y) \, dy \le c x_3.
$$
It follows that $u_{22}(x) = \int_{D} K_2 \vp_2 \approx -x_3 h^{-3/2}$ for $x \in S_1(h) \cup S_2(h) \cup S_3(h)$ and sufficiently small h.

\vskip 5pt
{\bf{Step 2.}} Estimate $u_{22}(x) \approx -h^{-1/2}$ for $x \in S_4(h)$.

We have
\begin{equation*}
u_{22}(x) = \int_D K_2(x_1-y_1,-y_2,x_3) \vp_2(y_1,y_2) \, dy_1 \, dy_2.
\end{equation*}
Put $A = B((h,0),h/2)$, $A_+ = \{y \in A: \, y_2 > 0\}$, $A_{1+} = \{y \in B((h,0),x_3): \, y_2 > 0\}$, $A_{2+} = A_+ \setminus A_{1+}$. By the same argument as in Step 1 we obtain $\int_{D \setminus A} K_2 \vp_2 \approx -x_3 h^{-3/2}$. Similarly like in Step 1 for $y \in A$ we obtain 
$\vp_2(y_1,y_2) - \vp_2(y_1,-y_2) \approx - y_2 y_1^{-1/2} \approx - y_2 h^{-1/2}$. Note that for $x \in S_4(h)$ we have $x = (h,0,x_3)$, where $x_3 \in (0,h/4]$. It follows that 
\begin{eqnarray*}
&& \int_{A} K_2(x_1-y_1,-y_2,x_3) \vp_2(y_1,y_2) \, dy_1 \, dy_2\\
&=& \int_{A_+} K_2(x_1-y_1,-y_2,x_3) (\vp_2(y_1,y_2)-\vp_2(y_1,-y_2)) \, dy_1 \, dy_2\\
&\approx& - x_3 h^{-1/2} 
 \int_{A_{1+} \cup A_{2+}} \frac{y_2^2}{((h - y_1)^2 + y_2^2 + x_3^2)^{5/2}}
\, dy_1 \, dy_2 \\
&\approx& \frac{-h^{-1/2}}{x_3^4} \int_0^{x_3} \rho^3 \, d\rho 
- x_3 h^{-1/2} \int_{x_3}^{h/2} \rho^{-2} \, d\rho \approx - h^{-1/2}.
\end{eqnarray*}

\vskip 5pt
{\bf{Step 3.}} Estimate $|u_{11}(x)| \le c x_3 h^{-5/2}$, $|u_{33}(x)| \le c x_3 h^{-5/2}$, $|u_{13}(x)| \le c h^{-3/2}$ for $x \in S_1(h) \cup S_2(h) \cup S_3(h)$.

We have
\begin{equation*}
u_{11}(x) = \int_D K_{11}(x_1-y_1,-y_2,x_3) \vp(y_1,y_2) \, dy_1 \, dy_2,
\end{equation*}
Put $D_1 = D \cap B(0,h)$. For $y \in D_1$ we have $\vp(y) \le c h^{1/2}$, for $y \in D \setminus D_1$ we have $\vp(y) \le c (\dist(0,y))^{1/2}$. It follows that 
\begin{eqnarray*}
\left| \int_{D_1} K_{11} \vp  \right| 
&\le& c x_3 \frac{h^2}{h^7} h^{1/2} \int_{D_1} \, dy 
\approx c x_3 h^{-5/2},\\
\left| \int_{D \setminus D_1} K_{11} \vp  \right| 
&\le& c x_3 \int_{h}^{\infty} \frac{\rho^2}{\rho^7} \rho^{1/2} \rho \, d\rho 
\approx c x_3 h^{-5/2}.
\end{eqnarray*}

Since $u_{11}(x) + u_{22}(x) + u_{33}(x) = 0$ and by Step 1 $u_{22}(x) \approx - x_3 h^{-3/2}$ for $x \in S_1(h) \cup S_2(h) \cup S_3(h)$ we get $|u_{33}(x)| \le c x_3 h^{-5/2}$.

Similarly we have
\begin{equation*}
u_{13}(x) = \int_D K_{13}(x_1-y_1,-y_2,x_3) \vp(y_1,y_2) \, dy_1 \, dy_2,
\end{equation*}
\begin{eqnarray*}
\left| \int_{D_1} K_{13} \vp  \right| 
&\le& c h \frac{h^2}{h^7} h^{1/2} \int_{D_1} \, dy 
\approx c h^{-3/2},\\
\left| \int_{D \setminus D_1} K_{13} \vp  \right| 
&\le& c \int_{h}^{\infty} \frac{\rho^3}{\rho^7} \rho^{1/2} \rho \, d\rho 
\approx c h^{-3/2}.
\end{eqnarray*}

\vskip 5pt
{\bf{Step 4.}} Estimate  $u_{13}(x) \approx h^{-3/2}$ for $x \in S_1(h)$.

We have
\begin{equation*}
u_{13}(x) = \int_D K_{3}(x_1-y_1,-y_2,x_3) \vp_1(y_1,y_2) \, dy_1 \, dy_2,
\end{equation*}
$$
K_{3}(x_1-y_1,-y_2,x_3) = C_K \frac{(x_1 - y_1)^2 + y_2^2 - 2 x_3^2}{((x_1 - y_1)^2 + y_2^2 + x_3^2)^{5/2}}. 
$$
Put $D_1 = \{(y_1,y_2): \, y_2 \in (-r,r), y_1 \in (f(y_2),r)\}$. By Lemma \ref{phiderivatives} we get $\vp_1(y) \approx \delta_D^{-1/2}(y)$ for $y \in D_1$. We also have $K_3(x_1-y_1,-y_2,x_3) \ge 0$ for $y \in D_1$ and $x \in S_1(h)$. Let $\beta(y)$ be the acute angle between $0y$ and $y_1$ axis. Put $D_2 = \{(y_1,y_2): \, |y| \in (h,r), \beta(y) \in [0,\pi/6)\}$. Clearly, $D_2 \subset D_1$. For $y \in D_2$ we have
$\vp_1(y) \approx \delta_D^{-1/2}(y) \approx |y|^{-1/2}$ and $K_{3}(x_1-y_1,-y_2,x_3) \ge c |y|^{-3}$. It follows that 
$$
\int_{D_1} K_3 \vp_1 \ge \int_{D_2} |y|^{-7/2} \, dy \approx h^{-3/2}.
$$
We also have
$$
\left|\int_{D \setminus D_1} K_3 \vp_1 \right| 
\le c \int_{D \setminus D_1} \delta_D^{-1/2}(y) \, dy \le c.
$$
Hence $u_{13}(x) \ge c h^{-3/2}$ for $x \in S_1(h)$ and sufficiently small $h$. By Step 3 $|u_{13}(x)| \le c h^{-3/2}$ so $u_{13}(x) \approx  h^{-3/2}$.

\vskip 5pt
{\bf{Step 5.}} Estimates  $u_{11}(x) \approx h^{-3/2}$, $u_{33}(x) \approx -h^{-3/2}$ for $x \in S_2(h)$.

Step 5 is similar to Step 4. We have
\begin{equation*}
u_{11}(x) = \int_D K_{1}(x_1-y_1,-y_2,x_3) \vp_1(y_1,y_2) \, dy_1 \, dy_2,
\end{equation*}
$$
K_{1}(x_1-y_1,-y_2,x_3) = 3 C_K \frac{x_3 (y_1 - x_1)}{((x_1 - y_1)^2 + y_2^2 + x_3^2)^{5/2}}. 
$$
Let $D_1$, $D_2$ be such as in Step 4. We have $K_1(x_1-y_1,-y_2,x_3) \ge 0$ for $y \in D_1$ and $x \in S_2(h)$. For $y \in D_2$ and $x \in S_2(h)$ we have
$K_{1}(x_1-y_1,-y_2,x_3) \ge c h |y|^{-4}$. It follows that 
$$
\int_{D_1} K_1 \vp_1 \ge c h \int_{D_2} |y|^{-9/2} \, dy \approx h^{-3/2}.
$$
We also have $\left|\int_{D \setminus D_1} K_1 \vp_1 \right| \le c$.
Hence $u_{11}(x) \ge c h^{-3/2}$ for $x \in S_2(h)$ and sufficiently small $h$. By Step 3 $|u_{11}(x)| \le c h^{-3/2}$ so $u_{11}(x) \approx  h^{-3/2}$. Since $u_{11}(x) + u_{22}(x) + u_{33}(x) = 0$ and by Step 1 $u_{22}(x) \approx - h^{-1/2}$ for $x \in S_2(h)$ we get $u_{33}(x) \approx -h^{-3/2}$.

\vskip 5pt
{\bf{Step 6.}} Estimates  $|u_{13}(x)| \le c h^{-3/2}$ for $x \in S_4(h)$, $u_{13}(x) \approx -h^{-3/2}$ for $x \in S_3(h)$, $u_{13}(x) \le - c x_3 h^{-5/2}$ for $x \in S_4(h)$.

We have
\begin{equation*}
u_{13}(x) = \int_{\R^2} K_{1}(x_1-y_1,-y_2,x_3) u_3(y_1,y_2,0) \, dy_1 \, dy_2,
\end{equation*}
$$
K_{1}(x_1-y_1,-y_2,x_3) = 3 C_K \frac{x_3 (y_1 - x_1)}{((x_1 - y_1)^2 + y_2^2 + x_3^2)^{5/2}}. 
$$
For $y \in D$ we have $u_3(y_1,y_2,0) = -1$ and for $y \in (\overline{D})^c$ by Corollary \ref{DeltaSqrt1} 
$$
u_3(y_1,y_2,0) = -(-\Delta)^{1/2}\vp(y) \approx (1 + |y|^{-5/2}) \delta_D^{-1/2}(y).
$$

Put 
\begin{eqnarray*}
A_1 &=& \{y \in B(0,h): \, y_1 \le 0\},\\
A_2 &=& \{y \in B(0,r) \setminus B(0,h): \, y_1 < 0, |y_2| \le |y_1|\},\\
A_3 &=& \{y \in B(0,r) \setminus B(0,h): \, y_1 \le 0, |y_2| \ge |y_1|\},\\
A_4 &=& \{y: \, y_2 \in [-h,h], y_1 \in (0,f(y_2)]\},\\
A_5 &=& \{y: \, y_2 \in (h,r] \cup [-r,h), y_1 \in (0,f(y_2)]\},\\
A_6 &=& D^c \setminus (A_1 \cup A_2 \cup A_3 \cup A_4 \cup A_5).
\end{eqnarray*}
Clearly $A_1, A_2, A_3, A_4, A_5, A_6 \subset D^c$. We also put $D_1 = B((0,h),h/2)$.

Let $x \in S_3(h) \cup S_4(h)$. We have 
$$
\left| \int_{A_1} K_1 u_3 \right| \le c h^{-3} \int_{A_1} \delta_D^{-1/2}(y) \, dy \le c h^{-3/2},
$$
$$
\int_{A_2} K_1 u_3 \approx  - x_3 \int_{A_2} |y|^{-9/2} \, dy \approx -x_3 h^{-5/2},
$$
$$
\left| \int_{A_3} K_1 u_3 \right| \le c h \int_{h/\sqrt{2}}^r \, dy_2 \int_{-y_2}^0 \, dy_1 \, |y_1|^{-1/2} y_2^{-4} 
\le c h^{-3/2}.
$$
For $x \in S_3(h) \cup S_4(h)$ and $y \in A_4$ we estimate $|y_1 - x_1| \le y_1 + h \le c h$, $f(y_2) \le c y_2^2$. Hence
$$
\left| \int_{A_4} K_1 u_3 \right| \le c x_3 h^{-4} \int_{-h}^h \, dy_2 \int_{0}^{f(y_2)} \, dy_1 \, (-y_1 + f(y_2))^{-1/2} \le c x_3 h^{-2}. 
$$
For $x \in S_3(h) \cup S_4(h)$ and $y \in A_5$ we estimate $|y_1 - x_1| \le y_1 + h \le c |y_2|$, $f(y_2) \le c y_2^2$. Hence
$$
\left| \int_{A_5} K_1 u_3 \right| \le c x_3 \int_{h}^r \, dy_2 \int_{0}^{f(y_2)} \, dy_1 \, (-y_1 + f(y_2))^{-1/2} y_2^{-4} \le c x_3 h^{-2}. 
$$
We also have 
$$
\left| \int_{A_6} K_1 u_3 \right| \le c x_3 \int_{A_6} |y|^{-13/2} \delta_D^{-1/2}(y) \, dy \le c x_3. 
$$
For $x \in S_3(h) $ we have
$$
\left| \int_{D_1} K_1 u_3 \right| = \left| \int_{D_1} K_1 \right|  \le 
c x_3 h^{-4} \int_{D_1} \, dy \approx x_3 h^{-2}.
$$
For $x \in S_4(h) $ we have
$$
\left| \int_{D_1} K_1 u_3 \right| = 
c x_3 \int_{D_1} \frac{y_1 - h}{((y_1 - h)^2 + y_2^2 + x_3^2)^{5/2}} \, dy_1 \, dy_2 = 0.
$$
For $x \in S_3(h) \cup S_4(h)$ we also have
$$
\left| \int_{D \setminus D_1} K_1 u_3 \right|  \le 
c x_3  \int_{D \setminus D_1} ((y_1 - h)^2 + y_2^2)^{-2} \, dy 
\le c x_3 h^{-2}.
$$
It follows that for $x \in S_3(h) \cup S_4(h)$
$$
|u_{13}(x)| = \left| \int_{\R^2} K_1 u_3 \right| \le c h^{-3/2},
$$
(for $x \in S_3(h)$ such estimate follows also from Step 3).

Now note that $K_{1}(x_1-y_1,-y_2,x_3) \le 0$ and $u_3(y_1,y_2,0) \ge 0$ for $x \in S_3(h) \cup S_4(h)$ and $y \in A_1 \cup A_3$. So $\int_{A_1 \cup A_3} K_1 u_3 \le 0$. It follows that for $x \in S_3(h) \cup S_4(h)$ we have
$$
u_{13}(x) = \int_{\R^2} K_1 u_3 \le \int_{A_2 \cup A_4 \cup A_5 \cup A_6 \cup D} K_1 u_3 \le -c x_3 h^{-5/2} + c_1 x_3 h^{-2}.
$$
Hence for $x \in S_3(h)$ and sufficiently small $h$ we have $u_{13}(x) \approx - h^{-3/2}$. For $x \in S_4(h)$ and sufficiently small $h$ we have $u_{13}(x) \le - c x_3 h^{-5/2}$.

\vskip 5pt
{\bf{Step 7.}} Estimates  $u_{33}(x) \approx h^{-3/2}$, $u_{11}(x) \approx -h^{-3/2}$ for $x \in S_4(h)$.

We have
\begin{equation*}
u_{33}(x) = \int_{\R^2} K_{3}(x_1-y_1,-y_2,x_3) u_3(y_1,y_2,0) \, dy_1 \, dy_2,
\end{equation*}
$$
K_{3}(x_1-y_1,-y_2,x_3) = C_K \frac{(x_1-y_1)^2 + y_2^2 - 2 x_3^2}{((x_1 - y_1)^2 + y_2^2 + x_3^2)^{5/2}}. 
$$
For $x \in S_4(h)$ and $y \in D^c$ we have $K_3(x_1-y_1,-y_2,x_3) > 0$, $u_3(y_1,y_2,0) \approx (1 + |y|^{-5/2}) \delta_D^{-1/2}(y)$. For $y \in D$ we have $u_3(y_1,y_2,0) = -1$. Let $A_1, A_2, A_3, A_4, A_5, A_6, D_1$ be such as in Step 6. We have 
\begin{eqnarray*}
\left| \int_{A_1 \cup A_4} K_3 u_3 \right| 
&\le& \frac{c}{h^{3}} \int_{A_1 \cup A_4} \delta_D^{-1/2}(y) \, dy \\
&\le& \frac{c}{h^{3}}  \int_{0}^h \, dy_2 \int_{-h}^{f(y_2)} \, dy_1 \, (-y_1 + f(y_2))^{-1/2} \approx  h^{-3/2},
\end{eqnarray*}
$$
\int_{A_2} K_3 u_3 \approx  \int_{A_2} |y|^{-7/2} \, dy \approx h^{-3/2},
$$
$$
\left| \int_{A_3 \cup A_5} K_3 u_3 \right| 
\le c  \int_{h/\sqrt{2}}^r \, dy_2 \int_{-y_2}^{f(y_2)} \, dy_1 \, \frac{(-y_1 + f(y_2))^{-1/2}}{y_2^3} 
\approx  h^{-3/2},
$$
$$
\left| \int_{A_6} K_3 u_3 \right| \le c \int_{A_6} |y|^{-11/2} \delta_D^{-1/2}(y) \, dy \le c,
$$
$$
\left| \int_{D \setminus D_1} K_3 u_3 \right|  \le 
c  \int_{D \setminus D_1} ((y_1 - h)^2 + y_2^2)^{-3/2} \, dy 
\le c h^{-1}.
$$

The integral over $D_1$ we compute directly. Recall that $D_1 = B((h,0),h/2)$ and $x = (x_1,x_2,x_3) \in S_4(h)$ so $x_1 = h$, $x_2 =0$, $x_3 \in (0,h/4]$. We have 
\begin{equation}
\label{iD1}
\int_{D_1} K_{3}(x_1-y_1,-y_2,x_3) u_3(y_1,y_2,0) \, dy_1 \, dy_2
= C_K \int_{D_1} \frac{(h-y_1)^2 + y_2^2 - 2 x_3^2}{((h - y_1)^2 + y_2^2 + x_3^2)^{5/2}} \, dy_1 \, dy_2.
\end{equation}
Let us introduce polar coordinates $h - y_1 = \rho \cos \theta$, $y_2 = \rho \sin \theta$. Then (\ref{iD1}) equals $2 \pi C_K \int_0^{h/2} \frac{\rho^2 - 2 x_3^2}{(\rho^2 + x_3^2)^{5/2}} \rho \, d\rho$. By substitution $t = \rho^2$ this is equal to $\pi C_K \int_0^{h^2/4} \frac{t - 2 x_3^2}{(t + x_3^2)^{5/2}} \, dt$. By elementary calculations this is equal to $\frac{-\pi C_K h^2}{2(h^2/4 + x_3^2)^{3/2}}$. Hence $\left|\int_{D_1} K_{3} u_3 \right| \le c/h$.

It follows that $|u_{33}(x)| \le c h^{-3/2}$. Since for $x \in S_4(h)$ and $y \in (\overline{D})^c$ we have $K_{3}(x_1-y_1,-y_2,x_3) > 0$ and $u_3(y_1,y_2,0) > 0$ we get 
$$
u_{33}(x) = \int_{\R^2} K_{3} u_3 \ge \int_{A_2 \cup D} K_{3} u_3 \ge 
\int_{A_2} K_{3} u_3 - \left| \int_{D} K_{3} u_3\right| \ge c h^{-3/2} - c_1 h^{-1}.
$$
It follows that $u_{33}(x) \approx h^{-3/2}$ for $x \in S_4(h)$ and sufficiently small $h$. Since $u_{11}(x) + u_{22}(x) + u_{33}(x) = 0$ and by Step 2 $u_{22}(x) \approx - h^{-1/2}$ for $x \in S_4(h)$ we get $u_{11}(x) \approx - h^{-3/2}$.

\vskip 5pt
{\bf{Step 8.}} Estimate $|u_{12}(x)| \le c x_3 h^{-3/2} |\log h|$ for $x \in S_1(h) \cup S_2(h) \cup S_3(h)$.

We have
\begin{equation}
\label{u12step8}
u_{12}(x) = \int_D K_{12}(x_1-y_1,-y_2,x_3) \vp(y_1,y_2) \, dy_1 \, dy_2,
\end{equation}
$$
K_{12}(x_1-y_1,-y_2,x_3) = -15 C_K \frac{x_3 (x_1-y_1) y_2 }{((x_1 - y_1)^2 + y_2^2 + x_3^2)^{7/2}}. 
$$
Let $D_1, D_2, D_3, D_4, D_5$ and $D_{i+}$, $D_{i-}$ for $i = 1,2,3,4$ be such as in Step 1.
We have
$$
\int_{D_1 \cup D_2} K_{12} \vp = 
- c x_3 \int_{D_{1+} \cup D_{2+}}
\frac{(x_1-y_1) y_2 }{((x_1 - y_1)^2 + y_2^2 + x_3^2)^{7/2}}
(\vp(y_1,y_2) - \vp(y_1,-y_2)) \, dy_1 \, dy_2.
$$
For $y \in D_{1+} \cup D_{2+}$ by Lemma \ref{phiderivatives} we get $|\vp(y_1,y_2) - \vp(y_1,-y_2)| = |2 y_2 \vp_{2}(y_1,\xi)| \le c y_2 (y_2 y_1^{-1/2} + y_1^{1/2} |\log y_1|)$, where $\xi \in (-y_2,y_2)$. Hence
\begin{eqnarray*}
\left| \int_{D_1} K_{12} \vp \right| 
&\le& 
c x_3 \int_{D_{1+}}
\frac{|x_1-y_1| }{((x_1 - y_1)^2 + y_2^2 + x_3^2)^{7/2}}
(y_2^3 y_1^{-1/2} + y_2^2 y_1^{1/2} |\log y_1|) \, dy_1 \, dy_2\\
&\le& 
c x_3 h^{-6} \int_0^h \, dy_1 \int_0^h \, dy_2 (y_2^3 y_1^{-1/2} + y_2^2 y_1^{1/2} |\log y_1|)\\
&+& c x_3 h \int_0^h \, dy_1 \int_{h}^{c_1 y_1^{1/2}} \, dy_2 (y_2^{-4} y_1^{-1/2} + y_2^{-5} y_1^{1/2} |\log y_1|)\\
&\le& 
c x_3 h^{-3/2} |\log h|.
\end{eqnarray*}
Note that for $y \in D_2$ we have $|x_1 - y_1| \le c y_1$. We obtain
\begin{eqnarray*}
\left| \int_{D_2} K_{12} \vp \right| 
&\le& 
c x_3 \int_{D_{2+}}
\frac{|x_1-y_1|}{((x_1 - y_1)^2 + y_2^2 + x_3^2)^{7/2}}
(y_2^3 y_1^{-1/2} + y_2^2 y_1^{1/2} |\log y_1|) \, dy_1 \, dy_2\\
&\le& 
c x_3 \int_h^r \, dy_1 \int_0^{y_1} \, dy_2 (y_2^3 y_1^{-13/2} + y_2^2 y_1^{-11/2} |\log y_1|)\\
&+& c x_3 \int_h^r \, dy_1 \int_{y_1}^{r} \, dy_2 (y_2^{-4} y_1^{1/2} + y_2^{-5} y_1^{3/2} |\log y_1|)\\
&\le& 
c x_3 h^{-3/2} |\log h|.
\end{eqnarray*}
Note that for $y \in D_3 \cup D_4$ we have $\vp(y) \le c \delta_D^{1/2}(y) \le c y_2$. Note also that $|x_1 - y_1| \le 2 h$ for $y \in D_3$ and $|x_1 - y_1| \le h + y_1$ for $y \in D_4$. We get
$$
\left| \int_{D_3} K_{12} \vp \right| \le
c x_3 h^{-5} \int_0^h \, dy_2 \int_0^{f_1(h)} \, dy_1 y_2 \le
c x_3 h^{-1}, 
$$
$$
\left| \int_{D_4+} K_{12} \vp \right| \le
c x_3  \int_h^r \, dy_2 \int_0^{c_1 y_2^{2}} \, dy_1 (h+y_1) y_2^{-5} 
\le c x_3 h^{-1}.
$$
The estimate of $\left| \int_{D_4-} K_{12} \vp \right|$ is the same so $\left| \int_{D_4} K_{12} \vp \right| \le c x_3 h^{-1}$.
Note that for $y \in D_5$ we have $|x_1 - y_1| \le c y_1$ and $\vp(y) \le c$. Hence
$$
\left| \int_{D_5} K_{12} \vp \right| \le
c x_3  \int_{B^c(0,c_1 r^2)} \frac{y_1 |y_2|}{(y_1^2 + y_2^2)^{7/2}} \, dy_1 \, dy_2 \le c x_3.
$$

\vskip 5pt
{\bf{Step 9.}} Estimate $|u_{12}(x)| \le c h^{-1/2} |\log h|$ for $x \in S_4(h)$.

We have
\begin{equation*}
u_{12}(x) = \int_D K_{12}(x_1-y_1,-y_2,x_3) \vp(y_1,y_2) \, dy_1 \, dy_2.
\end{equation*}

Put $A = B((h,0),h/2)$. By the same argument as in Step 8 we obtain $\left| \int_{D \setminus A} K_{12} \vp \right| \le c x_3 h^{-3/2} |\log h|$. We have
$$
\left| \int_{A} K_{12} \vp \right|= 
\left| c x_3 \int_{A}
\frac{(y_1 - h) y_2 }{((y_1 - h)^2 + y_2^2 + x_3^2)^{7/2}} \vp(y_1,y_2)  
\, dy_1 \, dy_2\right| .
$$
By substitution $z_1 = y_1 - h$, $z_2 = y_2$ this is equal to
\begin{equation}
\label{Step9intW}
\left| c x_3 \int_{B(0,h/2)}
\frac{z_1 z_2}{(z_1^2 + z_2^2 + x_3^2)^{7/2}} \vp(z_1 + h,z_2)  
\, dz_1 \, dz_2\right| 
=
\left| c x_3 \int_{W}
\frac{z_1 z_2 g(z_1,z_2)}{(z_1^2 + z_2^2 + x_3^2)^{7/2}}   
\, dz_1 \, dz_2\right|,
\end{equation}
where $g(z_1,z_2) = \vp(z_1 + h,z_2) - \vp(-z_1 + h,z_2) - \vp(z_1 + h,-z_2) + 
\vp(-z_1 + h,-z_2)$ and $W = \{z \in B(0,h/2): \, z_1 \ge 0, z_2 \ge 0\}$. Note that for $z \in W$ we have $g(z_1,z_2) = 4 z_1 z_2 \vp_{12}(\xi_1 + h,\xi_2)$, where $\xi_1 \in (-z_1,z_1)$, $\xi_2 \in (-z_2,z_2)$. By Lemma \ref{phiderivatives} we have for $z \in W$ and $\xi_1$, $\xi_2$ as above
$$
|\vp_{12}(\xi_1 + h,\xi_2)| \le c h^{-1/2} |\log h| + c z_2 h^{-3/2}.
$$
It follows that (\ref{Step9intW}) is bounded from above by
\begin{equation}
\label{Step9intW1}
c x_3 \int_{W}
\frac{z_1^2 z_2^2 (h^{-1/2} |\log h| + z_2 h^{-3/2})}{(z_1^2 + z_2^2 + x_3^2)^{7/2}} \, dz_1 \, dz_2.
\end{equation}
Put $W_1 = \{z: \, z_1 \in [0,x_3], z_2 \in [0,x_3]\}$, $W_2 = \{z \in B(0,h/2) \setminus B(0,x_3): z_1 \ge 0, \, z_2 \ge 0\}$. We have $W \subset W_1 \cup W_2$. (\ref{Step9intW1}) is bounded from above by 
\begin{eqnarray*}
&& 
c x_3 \int_{W_1}
\frac{z_1^2 z_2^2 (h^{-1/2} |\log h| + z_2 h^{-3/2})}{x_3^7} \, dz_1 \, dz_2 \\
&+&
c x_3 \int_{W_2}
\frac{z_1^2 z_2^2 (h^{-1/2} |\log h| + z_2 h^{-3/2})}{(z_1^2 + z_2^2)^{7/2}} \, dz_1 \, dz_2 \\
&\le&
c h^{-1/2} |\log h|.
\end{eqnarray*}

\vskip 5pt
{\bf{Step 10.}} Estimate $|u_{23}(x)| \le c h^{-1/2} |\log h|$ for $x \in S_1(h) \cup S_2(h) \cup S_3(h)$ and $|u_{23}(x)| \le c h^{-3/4} |\log h|$ for $x \in S_4(h)$.

For $x \in S_1(h) \cup S_2(h) \cup S_3(h)$ we have
\begin{equation*}
u_{23}(x) = \int_D K_{23}(x_1-y_1,-y_2,x_3) \vp(y_1,y_2) \, dy_1 \, dy_2.
\end{equation*}
The proof of the estimate $\left| \int_{D} K_{23} \vp \right| \le c h^{-1/2} |\log h|$ is very similar to the proof of the estimate $\left| \int_{D} K_{12} \vp \right| \le c x_3 h^{-3/2} |\log h|$ in Step 8 and it is omitted.

Now we estimate $|u_{23}(x)|$ for $x \in S_4(h)$. Put $p = (-r,0)$, recall that $z = (r,0)$. We have 
\begin{eqnarray*}
u_{23}(x) &=&
\int_{\R^2} K_2(x_1-y_1,-y_2,x_3) u_3(y_1,y_2,0) \, dy_1 \, dy_2\\
&=& \int_{B(0,r/4) \cap B(p,r)} K_2 u_3
+ \int_{(D \cap B(0,r/4)) \setminus (B(p,r) \cup B(z,r))} K_2 u_3\\
&+& \int_{(D^c \cap B(0,r/4)) \setminus (B(p,r) \cup B(z,r))} K_2 u_3
+\int_{B(0,r/4) \cap B(z,r)} K_2 u_3\\
&+& \int_{B^c(0,r/4)} K_2 u_3
= \text{I} + \text{II} +\text{III} +\text{IV} +\text{V}.
\end{eqnarray*}
Note that $u_3(y_1,y_2,0) = -(-\Delta)^{1/2} \vp(y_1,y_2)$ for $(y_1,y_2) \in \R^2 \setminus \partial D$. 

Put $A = B(0,r/4) \cap B(p,r)$. For $y \in A$ by Corollary \ref{DeltaSqrt1} we get $|(-\Delta)^{1/2} \vp(y)| \le c \delta_D^{-1/2}(y) \le c |y_1|^{-1/2}$. It follows that
\begin{eqnarray*}
\left| \text{I} \right| &\le&
c x_3 \int_A \frac{y_2 |y_1|^{-1/2}}{((h - y_1)^2 +y_2^2 + x_3^2)^{5/2}} \, dy_1 \, dy_2\\
&\le& c x_3 \int_0^h \, dy_2 \, \int_{-r/4}^{-f_1(y_2)} \, dy_1 \, 
\frac{y_2 |y_1|^{-1/2}}{h^5} +
c x_3 \int_h^{r/4} \, dy_2 \, \int_{-r/2}^{-f_1(y_2)} \, dy_1 \, 
\frac{y_2 |y_1|^{-1/2}}{y_2^5}\\
&\le& c x_3 h^{-3}.
\end{eqnarray*}
We also have 
\begin{equation*}
\left| \text{II} \right| 
\le c x_3 \int_0^h \, dy_2 \, \int_{0}^{f_1(y_2)} \, dy_1 \, 
y_2 h^{-5} +
c x_3 \int_h^{r/2} \, dy_2 \, \int_{0}^{f_1(y_2)} \, dy_1 \, 
y_2 y_2^{-5} 
\le c x_3 h^{-1}.
\end{equation*}
For $y \in (D^c \cap B(0,r/4)) \setminus (B(p,r) \cup B(z,r))$ by Corollary \ref{DeltaSqrt1} we get $|(-\Delta)^{1/2} \vp(y)| \le c \delta_D^{-1/2}(y) \approx (f(y_2) - y_1)^{-1/2}$. Hence
\begin{equation*}
\left| \text{III} \right| 
\le c x_3 \int_0^{r/4} \, dy_2 \, \int_{-f_1(y_2)}^{f(y_2)} \, dy_1 \, 
(f(y_2) - y_1)^{-1/2} \frac{y_2}{h^5 \vee y_2^5}.
\end{equation*}
For $y_2 \in (0,r/4)$ we have
$$
\int_{-f_1(y_2)}^{f(y_2)}  (f(y_2) - y_1)^{-1/2} \, dy_1 
= \int_0^{f_1(y_2) + f(y_2)} z^{-1/2} \, dz \le c y_2.
$$
It follows that 
\begin{equation*}
\left| \text{III} \right| 
\le c x_3 \int_0^{h} \frac{y_2^2}{h^5} \, dy_2 
+ c x_3 \int_h^{r/4} \frac{y_2^2}{y_2^5} \, dy_2
\le \frac{c x_3}{h^2}.
\end{equation*}
Clearly
$$
\text{IV} =
\int_{B(0,r/4) \cap B(z,r)} \frac{- c x_3 y_2}{((h - y_1)^2 +y_2^2 + x_3^2)^{5/2}} \, dy_1 \, dy_2 = 0.
$$
Using Corollary \ref{DeltaSqrt1} we get
$$
|\text{V}| \le
c x_3 \int_{D} \, dy + c x_3 \int_{D^c} \frac{\delta_D(y)^{-1/2}}{(1 + |y|)^{5/2}} \, dy \le c x_3. 
$$
It follows that for $x \in S_4(h)$ we have
\begin{equation}
\label{u23formula1}
|u_{23}(x)| \le \left| \text{I} + \text{II} +\text{III} +\text{IV} +\text{V} \right|
\le \frac{c x_3}{h^3}.
\end{equation}

On the other hand we have for $x \in S_4(h)$ 
\begin{equation*}
u_{23}(x) = \int_D K_{23}(x_1-y_1,-y_2,x_3) \vp(y_1,y_2) \, dy_1 \, dy_2.
\end{equation*}
Put $W = B((h,0),h/2)$, $W_+ = \{y \in W: \, y_2 > 0\}$. For $x \in S_4(h)$ one may show $\left| \int_{D \setminus W} K_{23} \vp \right| \le c h^{-1/2} |\log h|$. The proof of this inequality is omitted. It is very similar to the proof of the estimate $\left| \int_{D \setminus W} K_{12} \vp \right| \le c x_3 h^{-3/2} |\log h|$  see Step 9 and Step 8.

We have 
\begin{eqnarray}
\nonumber
\int_W K_{23} \vp &=&
-c \int_W \frac{12 x_3^2 - 3 (y_1 - h)^2 - 3 y_2^2}{((y_1 - h)^2 + y_2^2 + x_3^2)^{7/2}} y_2 \vp(y_1,y_2) \, dy_1 \, dy_2\\
\label{23intW}
&=&
-c \int_{W_+} \frac{12 x_3^2 - 3 (y_1 - h)^2 - 3 y_2^2}{((y_1 - h)^2 + y_2^2 + x_3^2)^{7/2}} y_2 (\vp(y_1,y_2) - \vp(y_1,-y_2)) \, dy_1 \, dy_2.
\end{eqnarray}
For $y \in W_+$ we have $\vp(y_1,y_2) - \vp(y_1,-y_2) = 2 y_2 \vp_2(y_1,\xi_2)$ where $\xi_2 \in (-y_2,y_2)$ and $\vp_2(y_1,\xi_2) = \vp_2(h,0) + (y_1-h,\xi_2) \circ \nabla\vp_2(\xi')$, where $\xi'$ is a point between $(h,0)$ and $(y_1,\xi_2)$. It follows that (\ref{23intW}) equals
\begin{eqnarray*}
&-& c \vp_2(h,0) \int_{W_+} \frac{12 x_3^2 - 3 (y_1 - h)^2 - 3 y_2^2}{((y_1 - h)^2 + y_2^2 + x_3^2)^{7/2}} 2 y_2^2 \, dy_1 \, dy_2\\
&-& c \int_{W_+} \frac{12 x_3^2 - 3 (y_1 - h)^2 - 3 y_2^2}{((y_1 - h)^2 + y_2^2 + x_3^2)^{7/2}} 2 y_2^2 (y_1-h,\xi_2) \circ \nabla\vp_2(\xi')\, dy_1 \, dy_2
= \text{I} + \text{II}.
\end{eqnarray*}
Put $V = B(0,h/2)$, $V_+ = \{z \in V: \, z_2 > 0\}$. By substitution $z_1 = y_1 - h$, $z_2 = y_2$ we obtain
\begin{eqnarray*}
\text{I} &=& -c \vp_2(h,0) \int_{V_+} \frac{12 x_3^2 - 3 z_1^2 - 3 z_2^2}{(z_1^2 + z_2^2 + x_3^2)^{7/2}} 2 z_2^2 \, dy_1 \, dy_2\\
&=& -c \vp_2(h,0) \int_{V} \frac{12 x_3^2 - 3 z_1^2 - 3 z_2^2}{(z_1^2 + z_2^2 + x_3^2)^{7/2}}  z_2^2 \, dy_1 \, dy_2.
\end{eqnarray*}
By symmetry of $z_1$, $z_2$ the above integral equals
$$
\frac{1}{2} \int_{V} \frac{12 x_3^2 - 3 z_1^2 - 3 z_2^2}{(z_1^2 + z_2^2 + x_3^2)^{7/2}}  (z_1^2 + z_2^2) \, dy_1 \, dy_2.
$$
Let us introduce polar coordinates $z_1 = \rho \cos \theta$, $z_2 = \rho \sin \theta$. Then the above expression equals $\pi \int_0^{h/2} \frac{12 x_3^2 - 3 \rho^2}{(\rho^2 + x_3^2)^{7/2}} \rho^3 \, d\rho$. By elementary calculation this is equal to $(3 \pi/16) h^4 (x_3^2 + h^2/4)^{-5/2}$. By Lemma \ref{phi2} $\vp_2(h,0) \le c h^{1/2} |\log h|$. Hence $|\text{I}| \le c h^{-1/2} |\log h|$.

Now  we estimate $\text{II}$. For $y \in W_+$ and $\xi_2$, $\xi'$ as above we have
\begin{equation}
\label{xiprime}
(y_1-h,\xi_2) \circ \nabla\vp_2(\xi') 
= (y_1 - h) \vp_{12}(\xi') + \xi_2 \vp_{22}(\xi').
\end{equation}
For any $w \in W$ by Lemma \ref{phiderivatives} we get
$|\vp_{12}(w)| \le c h^{-1/2} |\log h|$, $|\vp_{22}(w)| \le c h^{-1/2}$ so (\ref{xiprime}) is bounded from above by $c |y_1 - h| h^{-1/2} |\log h| + c |y_2| h^{-1/2}$. Put $B_+((h,0),x_3) = \{y \in B((h,0),x_3): \, y_2 > 0\}$. It follows that 
\begin{eqnarray*}
|\text{II}| &\le& 
\frac{c}{x_3^5} \int_{B_+((h,0),x_3)} |y - (h,0)|^3 h^{-1/2} |\log h| \, dy \\
&+& c \int_{W_+ \setminus B_+((h,0),x_3)} |y - (h,0)|^{-2} h^{-1/2} |\log h| \, dy 
\le c h^{-1/2} |\log h| |\log x_3|. 
\end{eqnarray*}
Hence for $x \in S_4(h)$ we have
\begin{equation}
\label{u23formula2}
|u_{23}(x)| \le \left| \int_{D \setminus W} K_{23} \vp \right| + |\text{I}| + |\text{II}| \le c h^{-1/2} |\log h| |\log x_3|.
\end{equation}

For any $\beta > 0$ and $x \in S_4(h)$ we get by (\ref{u23formula1})  $|u_{23}(x)|^{\beta} \le c_1^{\beta} x_3^{\beta} h^{-3\beta}$. Using this and (\ref{u23formula2}) we get $|u_{23}(x)|^{1 + \beta} \le c c_1^{\beta} x_3^{\beta} |\log x_3| h^{-3\beta - 1/2} |\log h|$. Putting $\beta = 1/9$ we obtain $|u_{23}(x)| \le c h^{-3/4} |\log h|^{9/10} \le c h^{-3/4} |\log h|$.
\end{proof}

\begin{lemma}
\label{onD}
For any $(x_1,x_2) \in D$ we have $u_{13}(x_1,x_2,0) = u_{23}(x_1,x_2,0) = 0$ and $u_{33}(x_1,x_2,0) > 0$.
\end{lemma}
\begin{proof}
The equalities $u_{13}(x_1,x_2,0) = u_{23}(x_1,x_2,0) = 0$ for $(x_1,x_2) \in D$ follows easily from (\ref{Steklov}). For $(x_1,x_2) \in \text{int}(D^c)$ we have
$$
u_{3}(x_1,x_2,0) = -(-\Delta)^{1/2}\vp(x) = 
\frac{1}{2 \pi} \int_D \frac{\vp(y)}{|y - x|^3} \, dy > 0.
$$
By Corollary \ref{DeltaSqrt1} we have $f(x_1,x_2) = u_3(x_1,x_2,0) \in L^1(\R^2)$. By the normal derivative lemma (\cite[Lemma 2.33]{ES1992}) we get $u_{33}(x_1,x_2,0) > 0$ for $(x_1,x_2) \in D$.
\end{proof}

\section{Harmonic extension for a ball}

The aim of this section is to show the following result.
\begin{proposition}
\label{extensionball}
Let $\vp$ be the solution of (\ref{maineq1}-\ref{maineq2}) for the ball $B(0,1) \subset \R^2$ and $u$ be the harmonic extension of $\vp$ given by (\ref{ext1}-\ref{ext2}). We have
\begin{equation}
\label{ballHessian}
H(u)(x) > 0, \quad  x \in \R^3 \setminus \{B^c(0,1) \times \{0\}\}.
\end{equation}
\end{proposition}

Let us recall that $H(u)(x)$ is the determinant of the Hessian matrix of $u$ in $x$. Recall also that the solution of (\ref{maineq1}-\ref{maineq2}) for the ball $B(0,1)$ is given by an explicit formula $\vp(x) = C_{B} (1 - |x|)^{1/2}$, $C_B = 2/\pi$. Hence for $x = (x_1,x_2,x_3)$, where $x_3 > 0$ the function $u$ is given by an explicit formula $u(x) = \int_{B(0,1)} K(x_1-y_1,x_2-y_2,x_3) \vp(y_1,y_2) \, dy_1 \, dy_2$. Applying this it is easy to check numerically that (\ref{ballHessian}) holds (e.g. using Mathematica).  Unfortunately, it seems very hard to prove formally (\ref{ballHessian}) using directly the explicit formula for $u$.

Instead, to show (\ref{ballHessian}) we use a "trick": we add an auxiliary function $w$ to the function $u$ and we use H. Lewy's Theorem \ref{HL}. First, we briefly present the idea of the proof. We define
$$
\Psi^{(b)}(x) = (1 - b) u(x) + b w(x), \quad b \in [0,1],
$$
where $w$ is an appropriately chosen auxiliary function given by
\begin{equation}
\label{defw}
w(x) = K(x_1,x_2,x_3 + \sqrt{3/2}).
\end{equation} 
Note that for any $q \ge 0$ the set $\{(x_1,x_2,x_3): \, K_{33}(x_1,x_2,x_3 + q) = 0, x_3 > - q\} = \{(x_1,x_2,x_3): \, x_1^2 + x_2^2 = (2/3)(x_3 + q)^2, x_3 > - q\}$. The function $w$ is chosen so that $w_{33}(x) = 0$ for $x \in \partial B(0,1) \times \{0\}$ i.e. for $x = (x_1,x_2,0)$ where $x_1^2 + x_2^2 = 1$. Such a choice helps to control $H(\Psi^{(b)})(x)$ near $\partial B(0,1) \times \{0\}$. One can directly check that $\Psi^{(1)} = w$ satisfies $H(\Psi^{(1)})(x) > 0$ for $x \in \R_+^3 \cup B(0,1) \times \{0\}$ (recall that $\R^3_+ = \{(x_1,x_2,x_3): \, x_3 > 0\}$). If $\Psi^{(0)} = u$ does not satisfy $H(\Psi^{(0)})(x) > 0$ for $x \in \R_+^3 \cup B(0,1) \times \{0\}$ one can show that there exists $b \in [0,1)$ for which $H(\Psi^{(b)})(x) \ge 0$ for $x \in \R_+^3 \cup B(0,1) \times \{0\}$ and such that there exists $x_0 \in \R^3_+$ for which $H(\Psi^{(b)})(x_0) = 0$. This gives contradiction with Theorem \ref{HL}. If $\Psi^{(0)} = u$ does not satisfy $H(\Psi^{(0)})(x) > 0$ for $x \in \R_-^3$ one can use Lemma \ref{lowerhalfspace} and again obtain contradiction. This finishes the presentation of the idea of the proof.

\begin{lemma}
\label{uplusaw}
Let $w$ be given by (\ref{defw}) and $v = u + aw$, $a \ge 0$. There exists $M_1 \ge 10$ and $h_1 \in (0,1/2]$ such that for any $a \ge 0$ we have
$$
H(v)(x) > 0, \quad \quad x \in A_1 \cup A_2 \cup A_3 \cup A_4,
$$
where
\begin{eqnarray*}
A_1 &=& \{(x_1,x_2,x_3): \, x_1^2 + x_2^2 \in [(1-h_1)^2,(1+h_1)^2], x_3 \in (0,h_1]\}, \\
A_2 &=& \{(x_1,x_2,x_3): \, x_1^2 + x_2^2 \in [(1+h_1)^2,M_1^2], x_3 \in (0,h_1]\}, \\
A_3 &=& \{(x_1,x_2,0): \, x_1^2 + x_2^2 <1 \},\\
A_4 &=& \{(x_1,x_2,x_3) \in \R_+^3: \, x_1^2 + x_2^2 \ge M_1^2 \,\,\, \text{or} \,\,\, x_3 \ge M_1\}. 
\end{eqnarray*}
\end{lemma}
\begin{proof}
First note that for any fixed $x_3 > 0$ the function $(x_1,x_2) \to v(x_1,x_2,x_3)$ is radial so it is enough to show the assertion for $x \in (A_1 \cup A_2 \cup A_3 \cup A_4) \cap L$, where $L = \{(x_1,x_2,x_3): \, x_2 = 0, x_1 \le 0\}$. Put $A'_i = A_i \cap L$, $i = 1,2,3,4$. For $x \in A'_1 \cup A'_2 \cup A'_3 \cup A'_4$ we have $v_{12}(x) = v_{23}(x) = 0$ and $v_{22}(x) < 0$. Hence $H(v)(x) = v_{22}(x) f(a,x)$, where
\begin{equation}
\label{fax}
f(a,x) = 
\left|
\begin{array}{cc} 
v_{11} & v_{13} \\       
v_{13} & v_{33}
\end{array}
\right| 
= 
\left|
\begin{array}{cc} 
u_{11} + a w_{11} & u_{13} + a w_{13}\\       
u_{13} + a w_{13} & u_{33} + a w_{33}
\end{array}
\right|
\end{equation}
and it is enough to show $f(a,x) < 0$ for $x \in A'_1 \cup A'_2 \cup A'_3 \cup A'_4$.

We will consider 4 cases: $x \in A'_1$, $x \in A'_2$, $x \in A'_3$, $x \in A'_4$. 

\vskip 5pt
{\bf{Case 1.}} $x \in A'_1$.

Put $q_0 = \sqrt{3/2}$ and $z_0 = (-1,0,0)$. Note that $w_{33}(z_0) = 0$, $w_{11}(z_0) = C_K q_0 (12 - 3q_0^2) (1+ q_0^2)^{-7/2} \approx 9.185 C_K (1+ q_0^2)^{-7/2}$, $w_{13}(z_0) = -C_K  (12q_0^2 - 3) (1+ q_0^2)^{-7/2} =- 15 C_K (1+ q_0^2)^{-7/2}$. Let us denote $w_{11}(x) = p_1(x)$, $w_{13}(x) = p_2(x)$. It is clear that for sufficiently small $h_1$ and $x \in A'_1$ we have
\begin{equation}
\label{p1p2}
\sqrt{\frac{9}{10}} |p_2(x)| > |p_1(x)|.
\end{equation}
Let $h_0$ be the constant from Proposition \ref{Hessianboundary}. For any $h \in (0,h_0]$ put
\begin{eqnarray*}
T_1(h) &=& \{(-1+h,0,x_3): \, x_3 \in (0,h/4]\},\\
T_2(h) &=& \{(-1+h,0,x_3): \, x_3 \in (h/4,h]\} 
\cup \{(x_1,0,h): \, x_1 \in [-1,-1+h)\},\\
T_3(h) &=& \{(x_1,0,h): \, x_1 \in [-\sqrt{2/3}h-1,-1]\},\\
T_4(h) &=& \{(x_1,0,h): \, x_1 \in [-1-h,-\sqrt{2/3}h-1)\}
\cup \{(-1-h,0,x_3): \, x_3 \in (0,h) \}.
\end{eqnarray*}
Note that the value $-\sqrt{2/3}h-1$ in the definition of $T_3(h)$, $T_4(h)$ is chosen so that $w_{33}(-\sqrt{2/3}h-1,0,h) = 0$. Note also that $w_{33}(x) \ge 0$ for $x \in T_1(h) \cup T_2(h) \cup T_3(h)$ and $w_{33}(x) < 0$ for $x \in T_4(h)$. 

We will consider 4 subcases: $x \in T_1(h)$, $x \in T_2(h)$, $x \in T_3(h)$, $x \in T_4(h)$.

\vskip 2pt
{\bf{Subcase 1a.}} $x \in T_1(h)$. 

By (\ref{fax}), Proposition \ref{Hessianboundary} and definition of $w$ we have
\begin{equation*}
f(a,x) = 
\left|
\begin{array}{cc} 
-b_1(x)h^{-3/2} + p_1(x) a & -b_2(x)h^{-3/2} - p_2(x) a\\       
-b_2(x)h^{-3/2} - p_2(x) a & \eps(x) a + b_1(x)h^{-3/2} + b_3(x) h^{-1/2}
\end{array}
\right|,
\end{equation*}
where $0 < B'_1 \le b_1(x) \le B_1$, $0 \le b_2(x) \le B_2$, $0 < B'_3 \le b_3(x) \le B_3$, $0 < P'_1 \le p_1(x) \le P_1$, $0 < P'_2 \le p_2(x) \le P_2$, 
$0 \le \eps(x) \le E(h) \le E(h_0)$, $\lim_{h \to 0^+} E(h) = 0$. More precisely, estimates of $b_1(x)$, $b_2(x)$ follow from estimates of $u_{11}(x)$, $u_{13}(x)$ for $S_4(h)$ in Proposition \ref{Hessianboundary}, estimates of $b_3(x)$ follow from $u_{33}(x) = -u_{11}(x) - u_{22}(x)$ and estimates of $u_{11}(x)$, $u_{22}(x)$ for $S_4(h)$ in Proposition \ref{Hessianboundary}. Estimates of $p_1(x)$, $p_2(x)$ follow from formulas of $w_{11}(z_0)$, $w_{13}(z_0)$ and continuity of $w_{11}(x)$, $w_{13}(x)$ near $z_0$. Estimates of $\eps(x)$ and $\lim_{h \to 0^+} E(h) = 0$ follow from equality $w_{33}(z_0) = 0$ and continuity of $w_{33}(x)$ near $z_0$.

Hence
\begin{eqnarray*}
&& f(a,x) =
- \eps(x) b_1(x)a h^{-3/2} - b_1^2(x) h^{-3} - b_1(x)b_3(x) h^{-2} + \eps(x) p_1(x) a^2\\
&& +  b_1(x) p_1(x) a h^{-3/2} + p_1(x) b_3(x) a h^{-1/2} - b_2^2(x) h^{-3} - p_2^2(x) a^2 - 2 b_2(x) p_2(x) a h^{-3/2}.
\end{eqnarray*}
Note that for sufficiently small $h$ we have
$$
p_1(x) b_3(x) a h^{-1/2} <  p_1(x) b_1(x) a h^{-3/2}.
$$
For sufficiently small $h$, using this and (\ref{p1p2}) we get
\begin{eqnarray*}
(9/10)p_2^2(x)a^2 + b_1^2(x) h^{-3} 
&>& p_1^2(x)a^2 + b_1^2(x) h^{-3} \\
&\ge& 2 b_1(x) p_1(x) a h^{-3/2} \\
&>& b_1(x) p_1(x) a h^{-3/2} + b_3(x) p_1(x) a h^{-1/2}.
\end{eqnarray*}
For sufficiently small $h$ we also have $p_1(x) \eps(x) a^2 < (1/10) p_2^2(x)a^2$. It follows that for sufficiently small $h_1 > 0$ and for all $0 < h \le h_1$, $a \ge 0$, $x \in T_1(h)$ we have $f(a,x) < 0$.

\vskip 2pt
{\bf{Subcase 1b.}} $x \in T_2(h)$. 

By (\ref{fax}), Proposition \ref{Hessianboundary} and definition of $w$ we have
\begin{equation*}
f(a,x) = 
\left|
\begin{array}{cc} 
b_1(x)h^{-3/2} + p_1(x) a & -b_2(x)h^{-3/2} - p_2(x) a\\       
-b_2(x)h^{-3/2} - p_2(x) a & \eps(x) a - b_1(x)h^{-3/2} + b_3(x) h^{-1/2}
\end{array}
\right|,
\end{equation*}
where $-B_1 \le b_1(x) \le B_1$, $0 < B'_2 \le b_2(x) \le B_2$, $0 < B'_3 \le b_3(x) \le B_3$, $0 < P'_1 \le p_1(x) \le P_1$, $0 < P'_2 \le p_2(x) \le P_2$, 
$0 \le \eps(x) \le E(h) \le E(h_0)$, $\lim_{h \to 0^+} E(h) = 0$. More precisely, estimates of $b_1(x)$, $b_2(x)$ follow from estimates of $u_{11}(x)$, $u_{13}(x)$ for $S_3(h)$ in Proposition \ref{Hessianboundary}, estimates of $b_3(x)$ follow from $u_{33}(x) = -u_{11}(x) - u_{22}(x)$ and estimates of $u_{11}(x)$, $u_{22}(x)$ for $S_3(h)$ in Proposition \ref{Hessianboundary}. Estimates of $p_1(x)$, $p_2(x)$, $\eps(x)$ and $\lim_{h \to 0^+} E(h) = 0$ follow by the same arguments as in Subcase 1a.
Hence
\begin{eqnarray*}
&&f(a,x) =
\eps(x) b_1(x) a h^{-3/2} - b_1^2(x) h^{-3} + b_1(x)b_3(x) h^{-2} + \eps(x) p_1(x) a^2 \\
&& - b_1(x) p_1(x) a h^{-3/2} + p_1(x) b_3(x) a h^{-1/2} - b_2^2(x) h^{-3} - p_2^2(x) a^2 - 2 b_2(x) p_2(x) a h^{-3/2}.
\end{eqnarray*}
Let us first assume that $b_1(x) \ge 0$. Then 
for sufficiently small $h$ we have
\begin{eqnarray*}
\eps(x) b_1(x) a h^{-3/2} &<& b_2(x) p_2(x) a h^{-3/2},\\
p_1(x) b_3(x) a h^{-1/2} &<& b_2(x) p_2(x) a h^{-3/2},\\
b_1(x)b_3(x) h^{-2} &<& b_2^2(x) h^{-3},\\
\eps(x) p_1(x) a^2 &<& p_2^2(x) a^2,
\end{eqnarray*}
which implies $f(a,x) < 0$.

Now let us assume that $b_1(x) < 0$. By (\ref{p1p2}) for sufficiently small $h$ we get
\begin{eqnarray*}
&& (9/10)p_2^2(x)a^2 + b_1^2(x) h^{-3} 
> p_1^2(x)a^2 + b_1^2(x) h^{-3} 
\ge |2 b_1(x) p_1(x) a h^{-3/2}|,\\
&& p_1(x) \eps(x) a^2 < (1/10) p_2^2(x)a^2,\\
&& p_1(x) b_3(x) a h^{-1/2} < 2 b_2(x) p_2(x) a h^{-3/2},
\end{eqnarray*}
which implies $f(a,x) < 0$.

It follows that for sufficiently small $h_1 > 0$ and for all $0 < h \le h_1$, $a \ge 0$, $x \in T_2(h)$ we have $f(a,x) < 0$.

\vskip 2pt
{\bf{Subcase 1c.}} $x \in T_3(h)$. 

By (\ref{fax}), Proposition \ref{Hessianboundary} and definition of $w$ we have
\begin{equation*}
f(a,x) = 
\left|
\begin{array}{cc} 
b_1(x)h^{-3/2} + p_1(x) a & -b_2(x)h^{-3/2} - p_2(x) a\\       
-b_2(x)h^{-3/2} - p_2(x) a & \eps(x) a - b_1(x)h^{-3/2} + b_3(x) h^{-1/2}
\end{array}
\right|,
\end{equation*}
where $0 < B'_1 \le b_1(x) \le B_1$, $-B_2 \le b_2(x) \le B_2$, $0 < B'_3 \le b_3(x) \le B_3$, $0 < P'_1 \le p_1(x) \le P_1$, $0 < P'_2 \le p_2(x) \le P_2$, 
$0 \le \eps(x) \le E(h) \le E(h_0)$, $\lim_{h \to 0^+} E(h) = 0$. More precisely, estimates of $b_1(x)$, $b_2(x)$ follow from estimates of $u_{11}(x)$, $u_{13}(x)$ for $S_2(h)$ in Proposition \ref{Hessianboundary}, estimates of $b_3(x)$ follow from $u_{33}(x) = -u_{11}(x) - u_{22}(x)$ and estimates of $u_{11}(x)$, $u_{22}(x)$ for $S_2(h)$ in Proposition \ref{Hessianboundary}. Estimates of $p_1(x)$, $p_2(x)$, $\eps(x)$ and $\lim_{h \to 0^+} E(h) = 0$ follow by the same arguments as in Subcase 1a.

For sufficiently small $h$ we have
\begin{eqnarray}
\label{b3b1}
b_3(x) h^{-1/2} &<& b_1(x) h^{-3/2}/2,\\
\label{b2b1eps}
\frac{2 B_2}{B'_1} \eps(x) &<& \frac{P'_2}{2}\\
\label{epsp1}
\eps(x) (p_1(x) + 2 \eps(x)) &<& \frac{p_2^2(x)}{4}.
\end{eqnarray}

If $\eps(x) a - b_1(x) h^{-3/2} + b_3(x) h^{-1/2} < 0$ then clearly $f(a,x) < 0$. So we may assume $\eps(x) a - b_1(x) h^{-3/2} + b_3(x) h^{-1/2} \ge 0$ which implies (see (\ref{b3b1}))
\begin{eqnarray}
\label{epsb1}
&& \eps(x) a \ge b_1(x) h^{-3/2} - b_3(x) h^{-1/2} > (b_1(x) h^{-3/2})/2,\\
\label{epsb1b3}
&& \eps(x) a > \eps(x) a - b_1(x) h^{-3/2} + b_3(x) h^{-1/2} \ge 0.
\end{eqnarray}
By (\ref{b2b1eps}) and (\ref{epsb1}) we get 
\begin{equation}
\label{b2h}
|b_2(x)| h^{-3/2} = \frac{2 |b_2(x)|}{b_1(x)} \frac{b_1(x) h^{-3/2}}{2} <
\frac{2 B_2}{B'_1} \eps(x) a < \frac{P'_2 a}{2} < \frac{p_2(x) a}{2}.
\end{equation}
By (\ref{epsb1}), (\ref{epsb1b3}), (\ref{b2h}), (\ref{epsp1}) we get
\begin{eqnarray*}
f(a,x) &\le&
(p_1(x) a + b_1(x) h^{-3/2}) \eps(x) a - \left(\frac{p_2(x) a}{2}\right)^2\\
&\le& (p_1(x) a + 2 \eps(x) a) \eps(x) a - \frac{p_2^2(x) a^2}{4} < 0.
\end{eqnarray*}

It follows that for sufficiently small $h_1 > 0$ and for all $0 < h \le h_1$, $a \ge 0$, $x \in T_3(h)$ we have $f(a,x) < 0$.

\vskip 2pt
{\bf{Subcase 1d.}} $x \in T_4(h)$. 

Note that for $x = (x_1,0,x_3) \in T_4(h)$ we have $w_{33}(x) < 0$. We also have
$$
u_{33}(x) = \int_{B(0,1)} K_{33}(x_1-y_1,x_2-y_2,x_3) \vp(y_1,y_2) \, dy_1 \, dy_2.
$$
Recall that $K_{33}(x_1-y_1,x_2-y_2,x_3) = C_K x_3 ((x_1 - y_1)^2 + (x_2 - y_2)^2 + x_3^2)^{-7/2} (6 x_3^2 - 9(x_1-y_1)^2 - 9(x_2 - y_2)^2)$. Hence to have $K_{33}(x_1-y_1,-y_2,x_3) < 0$ for all $(y_1,y_2) \in B(0,1)$ and $x_1 \le -1$ it is sufficient to have $6 x_3^2 - 9(x_1+1)^2 < 0$. Note that for $x = (x_1,0,x_3) \in T_4(h)$ we have $0 < x_3 < -\sqrt{3/2}(x_1+1)$, $x_1 < -1$. It follows that $6 x_3^2 - 9(x_1+1)^2 < 0$ and $u_{33}(x) < 0$. Hence $u_{33}(x) + aw_{33}(x) < 0$. Note that $u_{22}(x) + a w_{22}(x) < 0$ so $u_{11}(x) + a w_{11}(x) = -u_{22}(x) - a w_{22}(x) - u_{33}(x) - a w_{33}(x) > 0$. This and (\ref{fax}) implies that $f(a,x) < 0$ for any $a \ge 0$ and $x \in T_4(h)$.

\vskip 5pt
{\bf{Case 2.}} $x \in A'_2$.

This case follows from the same arguments as in subcase 1d.

\vskip 5pt
{\bf{Case 3.}} $x \in A'_3$.

Note that $w_{33}(x) > 0$ for $x \in A'_3$. Put $\overline{x}_3 = x_3 + \sqrt{3/2}$. We have 
$$
w_{11}(x) = C_K \ox_3 (x_1^2 + \ox_3^2)^{-7/2} (12 x_1^2 - 3 \ox_3^2).
$$
Note that 
$$
\{(x_1,0,x_3): \, w_{11}(x_1,0,x_3) = 0, x_1 \le 0, x_3 > -\sqrt{3/2}\}
= \{(x_1,0,x_3): \, x_3 + \sqrt{3/2} = -2 x_1\}.
$$
Put $T_1 = \left\{(x_1,0,0): \, x_1 \in \left[\frac{-\sqrt{3}}{2 \sqrt{2}},0\right]\right\}$, 
$T_2 = \left\{(x_1,0,0): \, x_1 \in \left(-1,\frac{-\sqrt{3}}{2 \sqrt{2}}\right)\right\}$. We have $A'_2 = T_1 \cup T_2$. Note that $w_{11}(-\sqrt{3}/(2\sqrt{2}),0,0)) = 0$, $w_{11}(x) \le 0$ for $x \in T_1$ and $w_{11}(x) > 0$ for $x \in T_2$. Note also that for $x = (x_1,0,0) \in A'_3$ we have $u(x) = \vp(x_1,0) = C_B (1 - x_1^2)^{1/2}$ so $u_{11}(x) < 0$. 

We will consider 2 subcases: $x \in T_1$, $x \in T_2$.

\vskip 2pt
{\bf{Subcase 3a.}} $x \in T_1$. 

Note that $w_{11}(x) \le 0$, $u_{11}(x) < 0$ so $u_{11}(x) + a w_{11}(x) < 0$ for $a \ge 0$. It follows that $u_{33}(x) + a w_{33}(x) > 0$ (because $u_{33} + a w_{33} = -(u_{11} + a w_{11} + u_{22} + a w_{22})$). Hence $f(a,x) < 0$.

\vskip 2pt
{\bf{Subcase 3b.}} $x \in T_2$.

For $(y_1,y_2) \in B(0,1)$ and $y = (y_1,y_2,0)$ we have $u(y) = \vp(y_1,y_2) = C_B (1 - y_1^2 - y_2^2)^{1/2}$. Therefore for $x \in T_2$ we obtain $u_{11}(x) = \vp_{11}(x_1,0) = -C_B (1 - x_1^2)^{-3/2}$, $u_{33}(x) = -\vp_{11}(x_1,0) - \vp_{22}(x_1,0) = C_B (1 - x_1^2)^{-3/2} (2 - x_1^2)$. Hence
\begin{equation}
\label{u33u11}
u_{33}(x) < 2 |u_{11}(x)|.
\end{equation}
For $x \in T_2$ we also have $-w_{22}(x) - w_{11}(x) = w_{33}(x) > 0$ so 
\begin{equation}
\label{w22w11}
|w_{22}(x)| >  |w_{11}(x)|.
\end{equation}
Note that for $x = (x_1,x_2,x_3) = (x_1,0,0) \in T_2$ we have $\frac{\ox_3}{|x_1|} = \frac{\sqrt{3/2}}{|x_1|}$ and  $\frac{\ox_3}{|x_1|} \in \left(\sqrt{\frac{3}{2}},2\right)$.

For $x \in T_2$ we have
$$
\frac{|w_{13}(x)|}{|w_{22}(x)|} = \frac{|x_1|}{\ox_3} \frac{(12 \ox_3^2 - 3 x_1^2)}{(3 x_1^2 + 3 \ox_3^2)} 
= \frac{|x_1|}{\ox_3} \left( 4 - \frac{5}{\left(\frac{\ox_3}{|x_1|}\right)^2 + 1} \right)
> \frac{2 |x_1|}{\ox_3} > 1,
$$
so
\begin{equation}
\label{w13w22}
|w_{13}(x)| >  |w_{22}(x)|.
\end{equation}
If $a = 0$ then by explicit formulas $f(a,x) < 0$. If $a > 0$ and $u_{11}(x) + a w_{11}(x) \le 0$ then $u_{33}(x) + a w_{33}(x) = -(u_{11}(x) + a w_{11}(x) + u_{22}(x) + a w_{22}(x)) > 0$ and $u_{13}(x) + a w_{13}(x) = a w_{13}(x) \ne 0$ (see (\ref{w13w22})) so $f(a,x) < 0$. So we may assume $a > 0$ and $u_{11}(x) + a w_{11}(x) > 0$.

Again by (\ref{fax}) and (\ref{u33u11}), (\ref{w13w22}) we get
\begin{equation*}
f(a,x) < 
\left|
\begin{array}{cc} 
u_{11}(x) + a w_{11}(x) & a |w_{22}(x)|\\       
a |w_{22}(x)|           & 2 |u_{11}(x)| - a w_{11}(x) - a w_{22}(x)
\end{array}
\right|.
\end{equation*}
Hence
\begin{eqnarray*}
f(a,x) &<& -2 |u_{11}(x)|^2 + 3 |u_{11}(x)| w_{11}(x) a - |u_{11}(x)| |w_{22}(x)| a \\
&&  - w_{11}^2(x) a^2 + w_{11}(x) |w_{22}(x)| a^2 - |w_{22}(x)|^2 a^2.
\end{eqnarray*}
By (\ref{w22w11}) this is bounded from above by
\begin{eqnarray*}
&& -2 |u_{11}(x)|^2 +2  |u_{11}(x)| |w_{11}(x)| a 
           - w_{11}^2(x) a^2 + w_{11}(x) |w_{22}(x)| a^2 - |w_{22}(x)|^2 a^2\\
&=&
-\left(\sqrt{2} |u_{11}(x)| - \frac{w_{11}(x) a}{\sqrt{2}}\right)^2
-\left(\frac{w_{11}(x) a}{\sqrt{2}} - \frac{|w_{22}(x)| a}{\sqrt{2}}\right)^2
-\left(\frac{|w_{22}(x)| a}{\sqrt{2}}\right)^2\\
&<& 0.
\end{eqnarray*}

\vskip 5pt
{\bf{Case 4.}} $x \in A'_4$.

Recall that $\ox_3 = x_3 + \sqrt{3/2}$ and put $\ox = (x_1,x_2,\ox_3)$. Recall also that $w(x) = K(\ox)$. We have
\begin{eqnarray*}
K_{11}(\ox) &=& 
C_K \ox_3 (x_1^2 + x_2^2 + \ox_3^2)^{-7/2} (12 x_1^2 - 3x_2^2 - 3\ox_3^2),\\
K_{13}(\ox) &=& 
C_K x_1 (x_1^2 + x_2^2 + \ox_3^2)^{-7/2} (12 \ox_3^2 - 3x_1^2 - 3x_2^2),\\
K_{33}(\ox) &=& 
C_K \ox_3 (x_1^2 + x_2^2 + \ox_3^2)^{-7/2} (6 \ox_3^2 - 9x_1^2 - 9x_2^2).
\end{eqnarray*}
For any $M \ge 10$ put
\begin{eqnarray*}
T_1(M) &=&
\{(x_1,0,x_3): \, \ox_3 = M, x_1 \le 0, \ox_3 \ge 3 |x_1|\},\\
T_2(M) &=&
\{(x_1,0,x_3): \, \ox_3 = M, x_1 \le 0, \sqrt{3/2} |x_1| \le \ox_3 < 3 |x_1|\},\\
T_3(M) &=&
\{(x_1,0,x_3): \, \ox_3 = M, x_1 \le 0,  |x_1| \le \ox_3 < \sqrt{3/2} |x_1|\}\\ 
&& \cup \{(x_1,0,x_3): \, x_1 = -M,  0 < \ox_3 < M\}.
\end{eqnarray*}
We will consider 3 subcases: $x \in T_1(M)$, $x \in T_2(M)$, $x \in T_3(M)$.

\vskip 2pt
{\bf{Subcase 4a.}} $x \in T_1(M)$.

Put $B = B(0,1) \subset \R^2$. We have
\begin{eqnarray*}
u_{11}(x) &=& 
\int_B (K_{11}(x_1-y_1,-y_2,x_3) - K_{11}(\ox)) \vp(y_1,y_2) \, dy_1 \, dy_2\\
&& + K_{11}(\ox) \int_B \vp(y_1,y_2) \, dy_1 \, dy_2,
\end{eqnarray*}
\begin{equation}
\label{K11x}
K_{11}(\ox) = \frac{C_K \ox_3 (12 x_1^2 - 3 \ox_3^2)}{(x_1^2 + \ox_3^2)^{7/2}}
< \frac{C_K \ox_3^3 \left(\frac{12}{9} - 3\right)}{(x_1^2 + \ox_3^2)^{7/2}}
< \frac{-c}{\ox_3^4}.
\end{equation}
For $(y_1,y_2) \in B$ we also have
$$
|K_{11}(x_1-y_1,-y_2,x_3) - K_{11}(\ox)| \le 
(|y_1| + |y_2| + |x_3 - \ox_3|) |\nabla K_{11}(\xi)| 
\le 4 |\nabla K_{11}(\xi)|,
$$
where $\xi$ is a point between $(x_1-y_1,-y_2,x_3)$ and $\ox = (x_1,0,\ox_3)$. For such $\xi$ we have
\begin{equation}
\label{K11xi}
|\nabla K_{11}(\xi)| \le \frac{c}{x_3^5}.
\end{equation}
By (\ref{K11x}), (\ref{K11xi}) for sufficiently large $M$ and all $x \in T_1(M)$ we have $u_{11}(x) < 0$. We also have $a w_{11}(x) = a K_{11}(\ox) < 0$ for $a \ge 0$, $x \in T_1(M)$. Hence $u_{11}(x) + a w_{11}(x) < 0$ which implies $f(a,x) < 0$. It follows that for sufficiently large $M_1 \ge 10$ and for all $M \ge M_1$, $a \ge 0$, $x \in T_1(M)$ we have $f(a,x) < 0$.

\vskip 2pt
{\bf{Subcase 4b.}} $x \in T_2(M)$.

First we need the following auxiliary lemma.

\begin{lemma}
\label{f24}
Let $f(y_1,y_3) = - 6 y_1^3 - 3 y_1^2 y_3 + 24 y_1 y_3^2 - 3 y_3^3$. For any $y_3 > 0$ and $y_1 \in [y_3/3,y_3]$ we have $f(y_1,y_3) > 4 y_3^3$.
\end{lemma}
\begin{proof}
The proof is elementary. Fix $y_3 > 0$ and put $g(y_1) = f(y_1,y_3)$. We have $g'(y_1) = -18 y_1^2 - 6 y_1 y_3 + 24 y_3^2$, $g'(y_1) = 0$ for $y_1 = (-8/6) y_3$ and $y_1 = y_3$ so $g$ is increasing for $y_1 \in [(-8/6)y_3,y_3]$. We also have $g(y_3/3) = (40/9) y_3^3$ so for any $y_1 \in [y_3/3,y_3]$ we have $g(y_1) > 4 y_3^3$. 
\end{proof}

Put $b = \int_B \vp(y_1,y_2) \, dy_1 \, dy_2$. For $x \in T_2(M)$ we have
\begin{equation*}
f(a,x) =
\left|
\begin{array}{cc} 
K_{11}(\ox) (a + b) + \eps_{11}(x) & K_{13}(\ox) (a + b) + \eps_{13}(x)\\       
K_{13}(\ox) (a + b) + \eps_{13}(x) & K_{33}(\ox) (a + b) + \eps_{33}(x)
\end{array}
\right|,
\end{equation*}
where 
$$
\eps_{ij}(x) = 
\int_B (K_{ij}(x_1-y_1,-y_2,x_3) - K_{ij}(\ox)) \vp(y_1,y_2) \, dy_1 \, dy_2
$$
for $(i,j) = (1,1)$ or $(1,3)$ or $(3,3)$. For $(y_1,y_2) \in B$ we have
$$
|K_{ij}(x_1-y_1,-y_2,x_3) - K_{ij}(\ox)| \le 
(|y_1| + |y_2| + |x_3 - \ox_3|) |\nabla K_{ij}(\xi)| 
\le 4 |\nabla K_{ij}(\xi)|,
$$
where $\xi$ is a point between $(x_1-y_1,-y_2,x_3)$ and $\ox = (x_1,0,\ox_3)$. We have $|\nabla K_{ij}(\xi)| \le c x_3^{-5}$, so 
\begin{equation}
\label{epsij}
|\eps_{ij}(x)| \le \frac{c b}{x_3^5}.
\end{equation}
Put 
\begin{equation*}
f_1(a,x) =
\left|
\begin{array}{cc} 
K_{11}(\ox) (a + b)  & K_{13}(\ox) (a + b) \\       
K_{13}(\ox) (a + b)  & K_{33}(\ox) (a + b)
\end{array}
\right|.
\end{equation*}
We have $|K_{ij}(\ox)| \le c x_3^{-4}$ so by (\ref{epsij}) we obtain
\begin{equation}
\label{ff1}
|f(a,x) - f_1(a,x)| \le c (a + b) b x_3^{-9}.
\end{equation}
On the other hand we have
\begin{eqnarray}
\nonumber
|f_1(a,x)| 
&\ge&
(a+b)^2 \left(K_{13}^2(\ox) - K_{11}(\ox) K_{33}(\ox) \right) \\
\nonumber
&\ge&
(a+b)^2 \left(K_{13}^2(\ox) - \left( \frac{K_{11}(\ox) + K_{33}(\ox)}{2} \right)^2 \right) \\
\label{K13K22}
&=&
(a+b)^2 \left(|K_{13}(\ox)|^2 - \left( \frac{|K_{22}(\ox)|}{2} \right)^2 \right).
\end{eqnarray}
We have
$$
|K_{13}(\ox)| -  \frac{|K_{22}(\ox)|}{2} 
= \frac{1}{2} C_K (|x_1|^2 + \ox_3^2)^{-7/2} (- 6 |x_1|^3 - 3 |x_1|^2 \ox_3 + 24 |x_1| \ox_3^2 - 3 \ox_3^3).
$$
By Lemma \ref{f24} we obtain
$$
|K_{13}(\ox)| -  \frac{|K_{22}(\ox)|}{2} \ge
\frac{1}{2} C_K (|x_1|^2 + \ox_3^2)^{-7/2} 4 \ox_3^3 \ge c x_3^{-4}.
$$
Using this and (\ref{K13K22}) we obtain
$$
|f_1(a,x)| \ge (a+b)^2 \left( |K_{13}(\ox)| -  \frac{|K_{22}(\ox)|}{2} \right)^2
\ge c (a + b)^2 x_3^{-8}.
$$
It follows that $f_1(a,x) < - c (a + b)^2 x_3^{-8}$. Using this and (\ref{ff1}) we obtain that for sufficiently large $M_1 \ge 10$ and for all $M \ge M_1$, $a \ge 0$, $x \in T_2(M)$ we have $f(a,x) < 0$.

\vskip 2pt
{\bf{Subcase 4c.}} $x \in T_3(M)$.

This subcase follows from the same arguments as in subcase 1d.
\end{proof}

\begin{proof}[proof of Proposition \ref{extensionball}]
On the contrary assume that there exists $z = (z_1,z_2,z_3) \in \R^3 \setminus (B^c(0,1) \times \{0\})$ such that $H(u)(z) \le 0$. By Lemma \ref{lowerhalfspace} we may assume that $z_1 \ge 0$. By an explicit formula for $\vp$ and Lemma \ref{onD} we may assume that $z_1 > 0$.  Define
$$
\Psi^{(b)}(x) = (1-b) u(x) + b w(x), \quad \quad b \in [0,1],
$$
where $w$ is given by (\ref{defw}). By direct computation for any $x = (x_1,x_2,x_3) \in \R^3$ with $x_3 > -\sqrt{3/2}$ we have
$$
H(w)(x) = C_K^3 \frac{27 (x_3 + \sqrt{3/2})(x_1^2 + x_2^2 + 2 (x_3 + \sqrt{3/2})^2)}{(x_1^2 + x_2^2 + (x_3 + \sqrt{3/2})^2)^{15/2}} > 0.
$$
Recall that $\R_+^3 = \{(x_1,x_2,x_3) \in \R^3: \, x_3 > 0\}$ and put $\Omega = \R_+^3 \setminus (A_1 \cup A_2 \cup A_4)$, where $A_1$, $A_2$, $A_4$ are sets from Lemma \ref{uplusaw}. By this lemma we obtain that $z \in \Omega$ and $H(\Psi^{(b)})(x) > 0$ for all $b \in [0,1]$ and $x \in \partial \Omega$. Note that $\Psi^{(0)} = u$ and $\Psi^{(1)} = w$, $H(\Psi^{(0)})(z) < 0$, $H(\Psi^{(1)})(x) > 0$ for all $x \in \overline{\Omega}$. Clearly, all second partial derivatives of $\Psi^{(b)}$ are uniformly Lipschitz continuos on $\overline{\Omega}$ that is
$$
\exists c \,\, \forall b \in [0,1] \,\, \forall x,y \in \overline{\Omega} \,\,
\forall i,j \in \{1,2,3\} \quad
\left|\Psi_{ij}^{(b)}(x) - \Psi_{ij}^{(b)}(x) \right| \le c |x - y|.
$$
It follows that there exists $b_0 \in [0,1)$ such that $H(\Psi^{(b_0)})(z_0) = 0$ for some $z_0 \in \Omega$ and $H(\Psi^{(b_0)})(x) \ge 0$ for all $x \in \overline{\Omega}$. This gives contradiction with Theorem \ref{HL}.
\end{proof}

\section{Concavity of $\vp$}

In this section we prove the main result of this paper Theorem \ref{mainthm}. This is done by using the method of continuity, H. Lewy's Theorem \ref{HL} and results from Sections 3, 4, 5.

For any $\eps \ge 0$ we define 
\begin{equation}
\label{ve}
v^{(\eps)}(x) = u(x) + \eps \left(-\frac{x_1^2}{2} -\frac{x_2^2}{2} + x_3^2 \right), \quad \quad x \in \R^3 \setminus (D^c \times \{0\}),
\end{equation}
where $u$ is the harmonic extension of $\vp$ given by (\ref{ext1}-\ref{ext2}) and $\vp$ is the solution of (\ref{maineq1}-\ref{maineq2}) for an open bounded set $D \subset \R^2$. When $D$ is not fixed we will sometimes write $v^{(\eps,D)}$ instead of $v^{(\eps)}$.

\begin{lemma}
\label{lowerhalfspace1}
Let $C_1 > 0$, $R_1 > 0$, $\kappa_2 \ge \kappa_1 > 0$, $D \in F(C_1,R_1,\kappa_1,\kappa_2)$, $\vp$ be the solution of (\ref{maineq1}-\ref{maineq2}) for $D$ and $u$ the harmonic extension of $\vp$ given by (\ref{ext1}-\ref{ext2}). For any $\eps \ge 0$ let $v^{(\eps)}$ be given by (\ref{ve}). For any $(x_1,x_2,x_3) \in \R_+^3$ we have $H(v^{(\eps)})(x_1,x_2,-x_3) = H(v^{(\eps)})(x_1,x_2,x_3)$. 
\end{lemma}
The proof of this lemma is similar to the proof of Lemma \ref{lowerhalfspace} and it is omitted.

\begin{proposition}
\label{vepsilon}
Fix $C_1 > 0$, $R_1 > 0$, $\kappa_2 \ge \kappa_1 > 0$ and $D \in F(C_1,R_1,\kappa_1,\kappa_2)$. Denote $\Lambda = \{C_1,R_1,\kappa_1,\kappa_1\}$. Let $\vp$ be the solution of (\ref{maineq1}-\ref{maineq2}) for $D$, $u$ the harmonic extension of $\vp$ and $v^{(\eps)}$ given by (\ref{ve}). For $M \ge 10$, $h \in (0,1/2]$, $\eta \in (0,1/2]$ we define
\begin{eqnarray*}
U_1(M) 
&=&
\{x \in \R^3: \, x_1^2 + x_2^2 \le M^2, x_3 = M \,\, \text{or} \,\, x_3 = -M\}\\
&& 
\cup \{x \in \R^3: \, x_1^2 + x_2^2 = M^2, x_3 \in [-M,M] \setminus \{0\}\},\\
U_2(h) 
&=&
\{x \in \R^3:\, (x_1,x_2) \in D, \delta_D((x_1,x_2)) \le h, x_3 \in [-h,h]\}\\
&& 
\cup \{x \in \R^3:\, (x_1,x_2) \notin D, \delta_D((x_1,x_2)) \le h, x_3 \in [-h,h] \setminus \{0\}\},\\
U_3(M,h,\eta) 
&=&
\{ x \in \R^3: \, (x_1,x_2) \notin D, \delta_D((x_1,x_2)) \ge h, x_1^2 + x_2^2 \le M^2,  \\
&& \quad \quad  \quad \quad \quad \quad  \quad \quad \quad \quad  \quad \quad \quad \quad  \quad
 x_3 \in [-\eta,\eta] \setminus \{0\}\}.
\end{eqnarray*}
Then we have
\begin{eqnarray*}
&&\exists c_1 = c_1(\Lambda) \in (0,1] \,\, \exists M_0 \ge 10 \,\, \exists h_1 = h_1(\Lambda) \in (0,1/2] \,\, \forall M \ge M_0 \, \, \forall \eps \in (0,c_1 M^{-7}]\\
&& \exists \eta = \eta(\Lambda,M,\eps) \in (0,1/2] \,\, \exists C = C(\Lambda,M,\eps) > 0 \, \, \forall x \in U_{1}(M) \cup U_2(h_1) \cup U_3(M,h_1,\eta)\\
&& \quad \quad \quad \quad \quad  \quad \quad \quad H(v^{(\eps)})(x) \ge C.
\end{eqnarray*}
\end{proposition}
\begin{proof}
In the whole proof we use convention stated in Remark \ref{constants1}. We have
$H(v^{(\eps)})(x) = W_1(x) + W_2(x) + W_3(x)$, where
\begin{eqnarray*}
W_1(x) &=& v_{12}^{(\eps)}(x) 
\left(v_{13}^{(\eps)}(x)v_{23}^{(\eps)}(x)-v_{12}^{(\eps)}(x)v_{33}^{(\eps)}(x)\right),\\
W_2(x) &=& -v_{23}^{(\eps)}(x) 
\left(v_{11}^{(\eps)}(x)v_{23}^{(\eps)}(x)-v_{13}^{(\eps)}(x)v_{12}^{(\eps)}(x)\right),\\
W_3(x) &=& v_{22}^{(\eps)}(x) f(\eps,x),\\
f(\eps,x) &=& v_{11}^{(\eps)}(x)v_{33}^{(\eps)}(x)-(v_{13}^{(\eps)}(x))^2.
\end{eqnarray*}
The proof consists of 3 parts.

\vskip 5pt
{\bf{Part 1.}} Estimates on $U_1(M)$.

We may assume in this part that $x_2 = 0$, $x_3 > 0$, $x_1 \le 0$.

By formulas $u_{ij}(x) = 
\int_D K_{ij}(x_1-y_1,x_2-y_2,x_3) \vp(y_1,y_2) \, dy_1 \, dy_2$ and explicit formulas for $K_{ij}$ (see Section 2), there exist $M_1 \ge 10$ and $c$ such that for any $M \ge M_1$ and $x \in U_1(M)$ we have $|u_{11}(x)| \le c x_3 M^{-5}$, 
$u_{22}(x) \approx - x_3 M^{-5}$, $|u_{33}(x)| \le c x_3 M^{-5}$, $|u_{13}(x)| \le c M^{-4}$, $|u_{23}(x)| \le c M^{-5}$, $|u_{12}(x)| \le c x_3 M^{-6}$.

Let us fix arbitrary $M \ge M_1$.

Let $x \in U_1(M)$ (recall that we assume that $x_2 = 0$, $x_3 > 0$, $x_1 \le 0$). We have
\begin{eqnarray}
\label{W1}
|W_1(x)| &\le& 
c  x_3 M^{-6} (M^{-4} M^{-5} + x_3 M^{-6} (x_3 M^{-5} + 2 \eps))
\le c x_3 M^{-15} + c \eps M^{-10},\\
\label{W2}
|W_2(x)| &\le& 
c  M^{-5} ((x_3 M^{-5} + \eps) M^{-5} + M^{-4} x_3 M^{-6})
\le c x_3 M^{-15} + c \eps M^{-10}.
\end{eqnarray}
Now we estimate $W_3(x)$. We have
\begin{equation}
\label{v22epsilon}
v_{22}^{(\eps)}(x) = u_{22}(x) - \eps \approx - c x_3 M^{-5} - \eps.
\end{equation}
The most important is the estimate of $f(\eps,x)$. To obtain this estimate we will consider 6 cases.

\vskip 2pt
{\bf{Case 1.1.}} $x_3 = M$, $|x_1| < x_3/3$.

Put $m(x) = C_K (x_1^2 + x_3^2)^{-7/2}$. We have
$$
u_{11}(x) \approx K_{11}(x) = m(x) x_3 (12 x_1^2 - 3 x_3^2) < c M^{-7} x_3 \left(12 \left(\frac{x_3}{3}\right)^2 - 3 x_3^2 \right),
$$
so $u_{11}(x) \le - c M^{-4}$.
We also have
$$
u_{33}(x) \approx K_{33}(x) = m(x) x_3 (6 x_3^2 - 9 x_1^2) \ge c M^{-7} x_3 
\left( 6 x_3^2 - 9 \left(\frac{x_3}{3}\right)^2 \right),
$$
so $u_{33}(x) \ge c M^{-4}$. Therefore for any $\eps \ge 0$ we have 
$v_{11}^{(\eps)}(x) \le - c M^{-4}$, $v_{33}^{(\eps)}(x) \ge c M^{-4}$. Hence $f(\eps,x) \le -c M^{-8}$.

\vskip 2pt
{\bf{Case 1.2.}} $x_3 = M$, $|x_1| \in [x_3/3,x_3/\sqrt{3/2}]$.

By the arguments from Subcase 4b in the proof of Lemma \ref{uplusaw} we have 
$u_{11}(x) u_{33}(x)-(u_{13}(x))^2 < -c M^{-8}$ for sufficiently large $M$. For any $\eps \ge 0$ we have
$$
\left|f(\eps,x) - \left(u_{11}(x) u_{33}(x)-(u_{13}(x))^2\right)\right|
\le 2 \eps^2 + 2\eps |u_{11}(x)| + \eps |u_{33}(x)|.
$$
For any $c_1 \in (0,1]$ and all $\eps \in (0,c_1 M^{-7}]$ this is bounded from above by $c c_1 M^{-11}$. It follows that for sufficiently small $c_1 \in (0,1]$, for sufficiently large $M$ and all $\eps \in (0,c_1 M^{-7}]$ we have $f(\eps,x) < -c M^{-8}$.

\vskip 2pt
{\bf{Case 1.3.}} $x_3 = M$, $|x_1| \in [x_3/\sqrt{3/2},x_3]$.

We have
$$
u_{11}(x) \approx K_{11}(x) = m(x) x_3 (12 x_1^2 - 3 x_3^2)
\approx M^{-7} x_3 \left(12 \frac{x_3^2}{3/2} - 3 x_3^2\right)
\approx M^{-4}.
$$
For $y \in D \subset B(0,1)$ we also have
\begin{eqnarray*}
&& K_{33}(x_1-y_1,-y_2,x_3) \le
C_K x_3 ((x_1-y_1)^2 + y_2^2 + x_3^2)^{-7/2} (6 x_3^2 - 9(x_1-y_1)^2)\\
&=& C_K x_3 ((x_1-y_1)^2 + y_2^2 + x_3^2)^{-7/2} (6 x_3^2 - 9 x_1^2 + 18 x_1 y_1 - 9 y_1^2) \le c M^{-5},
\end{eqnarray*}
so $u_{33}(x) \le c M^{-5}$. For sufficiently small $c_1 \in (0,1]$ and all $\eps \in (0,c_1 M^{-7}]$ we obtain $v_{11}^{(\eps)}(x) \approx M^{-4}$, $v_{33}^{(\eps)}(x) \le c M^{-5}$. We also have 
$u_{13}(x) \approx K_{13}(x) = m(x) x_1 (12 x_3^2 - 3 x_1^2)
\ge c M^{-4}$. It follows that for sufficiently small $c_1$, for sufficiently large $M$ and all $\eps \in (0,c_1 M^{-7}]$ we have $f(\eps,x) < - c M^{-8}$.

\vskip 2pt
{\bf{Case 1.4.}} $x_3 \in [M/4,M]$, $x_1 = -M$.

We have
$$
u_{11}(x) \approx K_{11}(x) = m(x) x_3 (12 x_1^2 - 3 x_3^2),
$$
so $u_{11}(x) \ge  c M^{-4}$.
We also have
$$
u_{33}(x) \approx K_{33}(x) = m(x) x_3 (6 x_3^2 - 9 x_1^2),
$$
so $u_{33}(x) \le -c M^{-4}$. Therefore for sufficiently small $c_1 \in (0,1]$ and all $\eps \in (0,c_1 M^{-7}]$ we have $v_{11}^{(\eps)}(x) \ge c M^{-4}$, $v_{33}^{(\eps)}(x) \le -c M^{-4}$. Hence $f(\eps,x) \le -c M^{-8}$.

\vskip 2pt
{\bf{Case 1.5.}} $x_3 \in [1,M/4]$, $x_1 = -M$.

We have
$$
u_{13}(x) \approx K_{13}(x) = m(x) x_1 (12 x_3^2 - 3 x_1^2),
$$
so $u_{13}(x) \le  - c M^{-4}$.
We also have
$$
u_{11}(x) \approx K_{11}(x) = m(x) x_3 (12 x_1^2 - 3 x_3^2),
$$
$$
u_{33}(x) \approx K_{33}(x) = m(x) x_3 (6 x_3^2 - 9 x_1^2),
$$
so $u_{11}(x) \ge c M^{-5}$, $u_{33}(x) \le -c M^{-5}$. Therefore for sufficiently small $c_1 \in (0,1]$ and all $\eps \in (0,c_1 M^{-7}]$ we have $v_{11}^{(\eps)}(x) \ge c M^{-5}$, $v_{33}^{(\eps)}(x) \le -c M^{-5}$. Hence $f(\eps,x) \le -c M^{-8}$.

\vskip 2pt
{\bf{Case 1.6.}} $x_3 \in (0,1]$, $x_1 = -M$.

By similar arguments as in Case 1.5 we get $u_{13}(x) \le - c M^{-4}$, $|u_{11}(x)| \le c M^{-5}$, $|u_{33}(x)| \le c M^{-5}$. Therefore for sufficiently small $c_1 \in (0,1]$ and all $\eps \in (0,c_1 M^{-7}]$ we have $|v_{11}^{(\eps)}(x)| \le c M^{-5}$, $|v_{33}^{(\eps)}(x)| \le c M^{-5}$. Hence for sufficiently small $c_1 \in (0,1]$, for sufficiently large $M$ and all $\eps \in (0,c_1 M^{-7}]$ we have $f(\eps,x) \le -c M^{-8}$.

\vskip 5pt

Finally in all 6 cases we get that for sufficiently small $c_1 \in (0,1]$, for sufficiently large $M$ and all $\eps \in (0,c_1 M^{-7}]$ we have $f(\eps,x) \le -c M^{-8}$. By (\ref{v22epsilon}) we get 
$W_3(x) = v_{22}^{(\eps)}(x) f(\eps,x) \ge c x_3 M^{-13} + c \eps M^{-8}$. By (\ref{W1}), (\ref{W2}) we have 
$|W_1(x) + W_2(x)| \le c x_3 M^{-15} + c \eps M^{-10}$. Recall that $H(v^{(\eps)})(x) = W_1(x) + W_2(x) + W_3(x)$. It follows that there exists sufficiently small $c'_1 = c'_1(\Lambda) \in (0,1]$ and sufficiently large $M_0 \ge M_1 \ge 10$ such that for any $M \ge M_0$ and $\eps \in (0,c'_1 M^{-7}]$ and all $x \in U_1(M)$ we have $H(v^{(\eps)})(x) \ge c \eps M^{-8}$. 

Let us fix the above $M_0$ and $M \ge M_0$ in the rest of the proof of this proposition.

\vskip 5pt
{\bf{Part 2.}} Estimates on $U_2(h)$.

We will use notation and results from Section 4. In particular we choose a point on $\partial D$ and choose a Cartesian coordinate system with origin at that point in the same way as in Section 4 (see Figures 1, 2, 3). Let $h \in (0,h_0]$, where $h_0$ is from Proposition \ref{Hessianboundary}. By Lemma \ref{lowerhalfspace1} we may assume $x_3 \ge 0$, by continuity we may assume $x_3 > 0$. It follows that it is enough to estimate $H(v^{(\eps)})(x)$ for $x \in S_1(h) \cup S_2(h) \cup S_3(h) \cup S_4(h)$. We will consider 2 cases. Assume that $\eps \in (0,1]$.

\vskip 2pt
{\bf{Case 2.1.}} $x \in S_1(h) \cup S_2(h) \cup S_3(h)$.

If $x \in S_1(h) \cup S_3(h)$ we have $(v_{13}^{(\eps)}(x))^2 = u_{13}^2(x) \ge c h^{-3}$, $v_{11}^{(\eps)}(x) v_{33}^{(\eps)}(x) = u_{11}(x) u_{33}(x) + 2 \eps u_{11}(x) - \eps u_{33}(x) - 2 \eps^2$, 
$|2 \eps u_{11}(x)| \le c \eps h^{-3/2}$, $|- \eps u_{33}(x)| \le c \eps h^{-3/2}$.

If $u_{11}(x) \le 0$ or $u_{33}(x) \le 0$ then $u_{11}(x) u_{33}(x) \le 0$ (recall that $u_{11}(x) + u_{33}(x) = -u_{22}(x) > 0$). If $u_{11}(x) > 0$ and $u_{33}(x) > 0$ then
$$
u_{11}(x) u_{33}(x) \le \left(\frac{u_{11}(x) + u_{33}(x)}{2}\right)^2 =
\left(\frac{u_{22}(x)}{2}\right)^2 \le c h^{-1}.
$$
Hence $f(\eps,x) = -(v_{13}^{(\eps)}(x))^2 + v_{11}^{(\eps)}(x) v_{33}^{(\eps)}(x) \le -c h^{-3}$ for sufficiently small $h$ and all $\eps \in (0,1]$.

If $x \in S_2(h)$ we have $u_{11}(x) \approx h^{-3/2}$, $u_{33}(x) \approx -h^{-3/2}$. Hence for sufficiently small $h$ and all $\eps \in (0,1]$ we have $v_{11}^{(\eps)}(x) \approx h^{-3/2}$, $v_{33}^{(\eps)}(x) \approx -h^{-3/2}$ and 
$f(\eps,x) \le -c h^{-3}$.

Hence for any $x \in S_1(h) \cup S_2(h) \cup S_3(h)$ for sufficiently small $h$ and all $\eps \in (0,1]$ we have $f(\eps,x) \le -c h^{-3}$. We have $v_{22}^{(\eps)}(x) \approx -x_3 h^{-3/2} - \eps$. It follows that $W_3(x) = v_{22}^{(\eps)}(x) f(\eps,x) \ge c x_3 h^{-9/2} + c \eps h^{-3}$. By Proposition \ref{Hessianboundary} we also have 
\begin{eqnarray*}
|W_1(x)| &\le& 
c x_3 h^{-3/2} |\log h| 
\left(h^{-3/2} h^{-1/2} |\log h| + (2 \eps + x_3 h^{-5/2}) x_3 h^{-3/2} |\log h|\right)\\
&\le& c x_3 h^{-7/2} |\log h|^2 + c \eps h^{-1} |\log h|^2,
\end{eqnarray*}
\begin{eqnarray*}
|W_2(x)| &\le& 
c h^{-1/2} |\log h| 
\left((\eps + x_3 h^{-5/2}) h^{-1/2} |\log h| + h^{-3/2} x_3 h^{-3/2} |\log h| \right)\\
&\le& c x_3 h^{-7/2} |\log h|^2 + c \eps h^{-1} |\log h|^2.
\end{eqnarray*} 
Hence there exists sufficiently small $h'_1$ such that for all $h \in (0,h'_1]$ and $\eps \in (0,1]$ we have $H(v^{(\eps)})(x) \ge c x_3 h^{-9/2} + c \eps h^{-3}$.

\vskip 2pt
{\bf{Case 2.2.}} $x \in S_4(h)$.

By Proposition \ref{Hessianboundary} for sufficiently small $h$ and all $\eps \in (0,1]$ we have $W_3(x) \ge c h^{-1/2} h^{-3} = c h^{-14/4}$,
\begin{eqnarray*}
|W_1(x)| &\le& 
c h^{-1/2} |\log h| 
\left(h^{-3/2} h^{-3/4} |\log h| +  h^{-3/2} h^{-1/2} |\log h|\right)\\
&\le& c h^{-11/4} |\log h|^2,
\end{eqnarray*}
\begin{eqnarray*}
|W_2(x)| &\le& 
c h^{-3/4} |\log h| 
\left(h^{-3/2} h^{-3/4} |\log h| + h^{-1/2} |\log h| h^{-3/2} \right)\\
&\le& c h^{-12/4} |\log h|^2.
\end{eqnarray*}
So there exists sufficiently small $h''_1$ such that for all $h \in (0,h''_1]$ and $\eps \in (0,1]$ we have $H(v^{(\eps)})(x) \ge c h^{-14/4}$.

Let us fix $h_1 = h'_1 \wedge h''_1$ in the rest of the proof of this proposition.

\vskip 5pt
{\bf{Part 3.}} Estimates on $U_3(M,h_1,\eta)$.

Let us choose arbitrary point on $\partial D$ and choose a Cartesian coordinate system in the same way as in Part 2. Note that it is enough to estimate $H(v^{(\eps)})(x)$ for $x \in U'_3(M,h_1,\eta) = \{(x_1,x_2,x_3): \, x_2 = 0, x_1 \in [-M,-h_1], x_3 \in (0,\eta]\}$ and sufficiently small $\eta = \eta(\Lambda,M,\eps)$.

Let $x \in U'_3(M,h_1,1/2)$. Note that $\dist(x,\partial D) \ge h_1$. By formulas $u_{ij}(x) = 
\int_D K_{ij}(x_1-y_1,x_2-y_2,x_3) \vp(y_1,y_2) \, dy_1 \, dy_2$ and explicit formulas for $K_{ij}$ (see Section 2) we have $|u_{11}(x)| \le c x_3 h_1^{-5}$, 
$|u_{22}(x)| \le c x_3 h_1^{-5}$, $|u_{33}(x)| \le c x_3 h_1^{-5}$, $|u_{13}(x)| \le c h_1^{-4}$, $|u_{23}(x)| \le c h_1^{-4}$, $|u_{12}(x)| \le c x_3 h_1^{-5}$. Note also that by our choice of coordinate system for any $y = (y_1,y_2) \in D$ we have $y_1 > 0$. From now on let us assume additionally that $x = (x_1,x_2,x_3) \in U'_3(M,h_1,1/2)$ is such that $x_3 \le |x_1|/\sqrt{6}$ (this condition implies $12 x_3^2 \le 2 x_1^2$). For such $x = (x_1,x_2,x_3)$ and any $y = (y_1,y_2) \in D$ we have $12 x_3^2 - 3(x_1 - y_1)^2 - 3(x_2 - y_2)^2 \le - (x_1 - y_1)^2 \le - x_1^2 \le -h_1^2$.

It follows that 
\begin{eqnarray}
\nonumber
|u_{13}(x)| 
&=&
\left|C_K \int_D \frac{(x_1 - y_1) (12 x_3^2 - 3(x_1 - y_1)^2 - 3(x_2 - y_2)^2)}{((x_1 - y_1)^2 + (x_2 - y_2)^2 + x_3^2)^{7/2}} \, \vp(y_1,y_2) \, dy_1 \, dy_2 \right| \\
\label{Ctildeh1}
&\ge& \frac{\tilde{C} h_1^3}{M^7}.
\end{eqnarray}
The constant $\tilde{C}$ will play an important role in the rest of the proof and this is the reason why it is not as usual denoted by $c$. Clearly, $\tilde{C}$ depends only on $\Lambda$.

Let us recall that in Parts 1 and 2 of this proof we have fixed constants $M_0$, $M \ge M_0$, $h_1$. At the end  of Part 1 we have chosen a constant $c'_1 \in (0,1]$. Let us choose a constant $c_1$
 to be 
\begin{equation}
\label{newc1}
c_1 = c'_1 \wedge \frac{1}{4} \tilde{C}h_1^3,
\end{equation}
where $\tilde{C}$ is a constant from (\ref{Ctildeh1}). In the rest of the proof let us fix this constant $c_1$ and $\eps \in (0,c_1 M^{-7}]$. The reason to define $c_1$ by (\ref{newc1}) is so that $2 \eps^2 \le 2 c_1^2 M^{-14} \le \frac{1}{8} \tilde{C}^2 h_1^6 M^{-14}$ which implies
\begin{equation}
\label{2eps3}
2 \eps^3 \le \frac{1}{4} \frac{\eps}{2} \tilde{C}^2 h_1^6 M^{-14},
\end{equation}
which will be crucial in the sequel.

Note that for sufficiently small $\eta = \eta(\Lambda,M,\eps)$ and $x \in U'_3(M,h_1,\eta)$ we have $x_3 \le |x_1|/\sqrt{6}$ and
\begin{eqnarray*}
v_{22}^{(\eps)}(x) &=& -\eps + u_{22}(x) \le - \eps + c x_3 h_1^{-5} \le - \frac{\eps}{2},\\
v_{11}^{(\eps)}(x) &=& -\eps + u_{11}(x) \le - \eps + c x_3 h_1^{-5} \le - \frac{\eps}{2}.
\end{eqnarray*}
We have
\begin{eqnarray*}
&& H(v^{(\eps)})(x) =
v_{11}^{(\eps)}(x) v_{22}^{(\eps)}(x) v_{33}^{(\eps)}(x)
+ 2 v_{12}^{(\eps)}(x) v_{23}^{(\eps)}(x) v_{13}^{(\eps)}(x)\\
&& - v_{22}^{(\eps)}(x) \left(v_{13}^{(\eps)}(x)\right)^2
   - v_{11}^{(\eps)}(x) \left(v_{23}^{(\eps)}(x)\right)^2
	 - v_{33}^{(\eps)}(x) \left(v_{12}^{(\eps)}(x)\right)^2,
\end{eqnarray*}
\begin{equation}
\label{Ctildeh1eps}
- v_{22}^{(\eps)}(x) \left(v_{13}^{(\eps)}(x)\right)^2 
\ge \frac{\eps}{2} \frac{\tilde{C}^2 h_1^6}{M^{14}},
\end{equation}
$$
- v_{11}^{(\eps)}(x) \left(v_{23}^{(\eps)}(x)\right)^2 \ge 0,
$$
\begin{eqnarray}
\label{v1233}
\left| v_{33}^{(\eps)}(x) \left(v_{12}^{(\eps)}(x)\right)^2 \right|
&\le&
(c x_3 h_1^{-5})^2 (2 \eps + c x_3 h_1^{-5}),\\
\label{v122313}
|v_{12}^{(\eps)}(x) v_{23}^{(\eps)}(x) v_{13}^{(\eps)}(x)|
&\le&
c x_3 h_1^{-5} h_1^{-4} h_1^{-4},\\
\label{112233}
|v_{11}^{(\eps)}(x) v_{22}^{(\eps)}(x) v_{33}^{(\eps)}(x)|
&\le&
(\eps + c x_3 h_1^{-5})^2 (2 \eps + c x_3 h_1^{-5}).
\end{eqnarray}
Note that the right hand sides of (\ref{v1233}), (\ref{v122313}), (\ref{112233}) are bounded by $2 \eps^3 + x_3 C(\Lambda,h_1)$ (note that $h_1$ depends only on $\Lambda$ so $C(\Lambda,h_1) = C(\Lambda)$). By (\ref{2eps3}) and (\ref{Ctildeh1eps}) we have $2 \eps^3 \le - \frac{1}{4} v_{22}^{(\eps)}(x) \left(v_{13}^{(\eps)}(x)\right)^2$. We also have $x_3 C(\Lambda,h_1) < - \frac{1}{4} v_{22}^{(\eps)}(x) \left(v_{13}^{(\eps)}(x)\right)^2$ for sufficiently small $\eta = \eta(\Lambda,M,\eps)$ and $x \in U'_3(M,h_1,\eta)$. For such $\eta$ and $x$ we have 
$$
H(v^{(\eps)})(x) \ge 
- \frac{1}{2} v_{22}^{(\eps)}(x) \left(v_{13}^{(\eps)}(x)\right)^2
\ge \frac{\eps}{4} \frac{\tilde{C}^2 h_1^6}{M^{14}}.
$$
\end{proof}

\begin{lemma}
\label{ball1}
Let $\vp$ be the solution of (\ref{maineq1}-\ref{maineq2}) for $B(0,1)$, $u$ the harmonic extension of $\vp$ and $v^{(\eps)}$ given by (\ref{ve}). For $M \ge 10$, $h \in (0,1/2]$, $\eta \in (0,1/2]$ we define
\begin{eqnarray*}
U_1(M) 
&=&
\{x \in \R^3: \, x_1^2 + x_2^2 \le M^2, x_3 = M \,\, \text{or} \,\, x_3 = -M\}\\
&& 
\cup \{x \in \R^3: \, x_1^2 + x_2^2 = M^2, x_3 \in [-M,M] \setminus \{0\}\},\\
U_2(h) 
&=&
\{x \in \R^3:\, x_1^2 + x_2^2 \in [(1-h)^2,1), x_3 \in [-h,h]\}\\
&& 
\cup \{x \in \R^3:\, x_1^2 + x_2^2 \in [1,(1+h)^2], x_3 \in [-h,h] \setminus \{0\}\},\\
U_3(M,h,\eta) 
&=&
\{x \in \R^3:\, x_1^2 + x_2^2 \in [(1+h)^2,M^2], x_1^2 + x_2^2 \le M^2, x_3 \in [-\eta,\eta] \setminus \{0\}\}.
\end{eqnarray*}
Then we have
\begin{eqnarray*}
&&\exists c_1 \in (0,1] \,\, \exists M_0 \ge 10 \,\, \exists h_1 \in (0,1/2] \,\, \forall M \ge M_0 \, \, \exists \eta = \eta(M) \in (0,1/2]\\
&& \forall \eps \in (0,c_1 M^{-7}] \, \, \forall x \in U_{1}(M) \cup U_2(h_1) \cup U_3(M,h_1,\eta)\\
&& \quad \quad \quad \quad \quad \quad \quad \quad H(v^{(\eps)})(x) > 0.
\end{eqnarray*}
\end{lemma}
\begin{remark}
It is important here that $\eta$ does not depend on $\eps$.
\end{remark}
\begin{proof}
Existence of $c_1$, $M_0$, $h_1$ and the estimate  $H(v^{(\eps)})(x) >0$ for $x \in U_{1}(M) \cup U_2(h_1)$ (where $M \ge M_0$, $\eps \in (0,c_1 M^{-7}]$) follow from the arguments from the proof of Proposition \ref{vepsilon}. 

Let $\eps \in (0,1]$. Fix $M \ge M_0$ and let $x \in U_3(M,h_1,1/2)$. We may assume that $x_2 =0$, $x_3 > 0$, $x_1 < 0$. We have $H(v^{(\eps)})(x) = v_{22}^{(\eps)}(x) f(\eps,x)$, where
$f(\eps,x) = v_{11}^{(\eps)}(x)v_{33}^{(\eps)}(x)-(v_{13}^{(\eps)}(x))^2$. We have $u_{22}(x) < 0$ so $v_{22}^{(\eps)}(x) = u_{22}(x) - \eps < 0$. We also have $|u_{11}(x)| \le c x_3 h_1^{-5}$, $|u_{33}(x)| \le c x_3 h_1^{-5}$ which gives
$$
v_{11}^{(\eps)}(x)v_{33}^{(\eps)}(x) = (u_{11}(x) - \eps) (u_{33}(x) + 2\eps)
< c x_3 h_1^{-10} + c x_3 h_1^{-5}.
$$
Let us additionally assume that $x_3$ is sufficiently small so that $x_3 \le \frac{|x_1| - 1}{\sqrt{6}}$. For such $x$ by the arguments from the proof of Proposition \ref{vepsilon} we have $|u_{13}(x)| \ge c h_1^3 M^{-7}$ so $|v_{13}^{(\eps)}(x)|^2 = |u_{13}(x)|^2 \ge c h_1^6 M^{-14}$. 
Hence for sufficiently small $\eta = \eta(M)$ and $x \in U_3(M,h_1,\eta)$ we have $f(\eps,x) < 0$, which implies $H(v^{(\eps)})(x) > 0$.
\end{proof}

\begin{proposition}
\label{ballepsilon}
Let $\vp$ be the solution of (\ref{maineq1}-\ref{maineq2}) for $B(0,1)$, $u$ the harmonic extension of $\vp$ and $v^{(\eps)}$ given by (\ref{ve}). For $M \ge 10$ put
$$
\Omega_M = \{x \in \R^3: \, x_1^2 + x_2^2 \le M^2, x_3 \in [-M,M]\} 
\setminus \{x \in \R^3: \, x_1^2 + x_2^2 \in [1,M^2], x_3 = 0\}. 
$$
Let $c_1$ and $M_0$ be the constants from Lemma \ref{ball1}. Then we have
$$
\forall M \ge M_0 \,\, \forall \eps \in (0,c_1 M^{-7}] \, \, \forall x \in \Omega_M \quad \quad H(v^{(\eps)})(x) > 0.
$$
\end{proposition}
\begin{proof}
On the contrary assume that there exists $M_1 \ge M_0$, $\eps_1 \in (0,c_1 M_1^{-7}]$, $z \in \Omega_{M_1}$ such that 
$
H(v^{(\eps_1)})(z) \le 0
$.
By Lemma \ref{ball1} there exists $h_1 \in (0,1/2]$ and $\eta_1 = \eta_1(M_1) \in (0,1/2]$ such that $\forall \eps \in (0,c_1 M_1^{-7}]$, $\forall x \in U_{1}(M_1) \cup U_2(h_1) \cup U_3(M_1,h_1,\eta_1)$ $H(v^{(\eps)})(x) > 0$. 

Note that by $v^{(0)} = u$ and by Proposition \ref{extensionball} we have $H(v^{(0)})(x) > 0$ for all $x \in \Omega_{M_1}$. It follows that there exists $\eps_2 \in (0,\eps_1]$ and $\tilde{z} \in \Omega_{M_1} \setminus (U_{1}(M_1) \cup U_2(h_1) \cup U_3(M_1,h_1,\eta_1))$ such that $H(v^{(\eps_2)})(\tilde{z}) = 0$ and $H(v^{(\eps_2)})(x) \ge 0$ for all $x \in \Omega_{M_1}$. This gives contradiction with Theorem \ref{HL}.
\end{proof}

As a direct conlusion of Propositions \ref{vepsilon} and \ref{ballepsilon} we obtain
\begin{corollary}
\label{Qepsilon}
Fix $C_1 > 0$, $R_1 > 0$, $\kappa_2 \ge \kappa_1 > 0$ and $D \in F(C_1,R_1,\kappa_1,\kappa_2)$. Denote $\Lambda = \{C_1,R_1,\kappa_1,\kappa_1\}$. Let $\vp^{(D)}$ be the solution of (\ref{maineq1}-\ref{maineq2}) for $D$, $u^{(D)}$ the harmonic extension of $\vp^{(D)}$ given by (\ref{ext1}-\ref{ext2}) and $v^{(\eps,D)}$ given by (\ref{ve}).
Then we have
\begin{eqnarray*}
&&\exists c_1 = c_1(\Lambda) \in (0,1] \,\, \exists M_0 \ge 10 \,\, \exists h_1 = h_1(\Lambda) \in (0,1/2] \,\, \forall M \ge M_0 \, \, \forall \eps \in (0,c_1 M^{-7}]\\
&& \exists \eta = \eta(\Lambda,M,\eps) \in (0,(1/2) \wedge \eps] \,\, \exists c_2 = c_2(\Lambda,M,\eps) > 0 \\
&& \forall x \in Q(M,D,\eps) \quad \quad  H(v^{(\eps,D)})(x) \ge c_2,\\
&& \forall x \in \Omega(M,B(0,1),\eps) \quad \quad  H(v^{(\eps,B(0,1))})(x) \ge c_2,
\end{eqnarray*}
where $Q(M,D,\eps) = Q_1(M) \cup Q_2(M,D,\eps) \cup Q_3(M,D,\eps)$,
\begin{eqnarray*}
Q_1(M) 
&=&
\{x \in \R^3: \, x_1^2 + x_2^2 \le M^2, x_3 = M \,\, \text{or} \,\, x_3 = -M\}\\
&& 
\cup \{x \in \R^3: \, x_1^2 + x_2^2 = M^2, x_3 \in [-M,M] \setminus \{0\}\},\\
Q_2(M,D,\eps) 
&=&
\{x \in \R^3:\, (x_1,x_2) \in D, \delta_D((x_1,x_2)) \le h_1, x_3 \in [-\eta,\eta]\},\\
Q_3(M,D,\eps) 
&=&
\{x \in \R^3:\, (x_1,x_2) \in D^c, x_1^2 + x_2^2 \le M^2, x_3 \in [-\eta,\eta] \setminus \{0\}\},\\
\Omega'(M,D,\eps)
&=&
\{x \in \R^3: \, x_1^2 + x_2^2 \le M^2, x_3 \in [-M,M]\}\\
&& \quad \quad \quad \quad \quad \quad \quad \quad
\setminus \overline{(Q_2(M,D,\eps) \cup Q_3(M,D,\eps))},\\
\Omega(M,D,\eps)
&=&\overline{\Omega'(M,D,\eps)}.
\end{eqnarray*}
\end{corollary}

\begin{proof}[proof of Theorem \ref{mainthm}] 

$\\${\bf{Step 1.}} 

In this step we will use the notation from Corollary \ref{Qepsilon}. We will show that for any $\Lambda = \{C_1,R_1,\kappa_1,\kappa_2\}$, $D \in F(\Lambda)$ and $x \in \R^3 \setminus (D^c \times \{0\})$ we have 
$H(u^{(D)})(x) > 0$.

Fix $\Lambda = \{C_1,R_1,\kappa_1,\kappa_2\}$ where $C_1 > 0$, $R_1 > 0$, $\kappa_2 \ge \kappa_1 > 0$ and fix $D_0 \in F(\Lambda)$. Let $\{D(t)\}_{t \in [0,1]}$, $D(0) = D_0$, $D(1) = B(0,1)$ be the family of domains defined by (\ref{construction}). By Lemma \ref{deformation} there exists $\Lambda' = \{C'_1,R'_1,\kappa'_1,\kappa'_2\}$ where $C'_1 > 0$, $R'_1 > 0$, $\kappa'_2 \ge \kappa'_1 > 0$ such that $\forall t \in [0,1]$ $D(t) \in F(\Lambda')$. Note that $H(u^{(D_0)})(x)$ does not vanish identically in $\R^3 \setminus (D_0^c \times \{0\})$ because it does not vanish near $\partial D_0 \times \{0\}$. 

On the contrary assume that there exists $x_0 \in \R^3 \setminus (D_0^c \times \{0\})$ such that $H(u^{(D_0)})(x_0) \le 0$. If $H(u^{(D_0)})(x_0) = 0$ and $\forall x \in \R^3 \setminus (D_0^c \times \{0\})$ $H(u^{(D_0)})(x) \ge 0$ then we get contradiction with Theorem \ref{HL}. So we may assume that $H(u^{(D_0)})(x_0) < 0$.

By  Corollary \ref{Qepsilon} applied to $\Lambda' = \{C'_1,R'_1,\kappa'_1,\kappa'_1\}$ there exist $M \ge M_0 \ge 10$, $\eps \in (0,c_1 M^{-7}]$ such that $x_0 \in \Omega(M,D_0,\eps)$ and $H(v^{(\eps,D_0)})(x_0) < 0$. Let us fix such $M$ and $\eps$. By Corollary \ref{Qepsilon} $\forall t \in [0,1]$ $\forall x \in Q(M,D(t),\eps)$   $H(v^{(\eps,t)})(x) \ge c = c(\Lambda',M,\eps) > 0$, where $v^{(\eps,t)} = v^{(\eps,D(t))}$.

By the construction in Lemma \ref{deformation} there exist $n \in \N$ and $0 = t(0) < t(1) < \ldots < t(n) = 1$ such that $\forall i \in \{0,\ldots,n-1\}$ $\forall t \in [t(i),t(i+1)]$
\begin{equation}
\label{smallh13}
d(D(t(i)),D(t)) < h_1/3,
\end{equation}
where $d(D_1,D_2) = [\sup_{x \in \partial D_1} \dist (x,\partial D_2)] \wedge [\sup_{x \in \partial D_2} \dist (x,\partial D_1)]$ and $h_1 = h_1(\Lambda')$ is the constant from Corollary \ref{Qepsilon}. Let us note that by our assumption 
$\inf\{H(v^{(\eps,t(0))})(x): \, x \in \Omega(M,D(t(0)),\eps)\} < 0$. 
By Corollary \ref{Qepsilon}
$\inf\{H(v^{(\eps,t(n))})(x): \, x \in \Omega(M,D(t(n)),\eps)\} > 0$.
Hence there exists $j \in \{0,\ldots,n-1\}$ such that 
$\inf\{H(v^{(\eps,t(j))})(x): \, x \in \Omega(M,D(t(j)),\eps)\} < 0$
and
$\inf\{H(v^{(\eps,t(j + 1))})(x): \, x \in \Omega(M,D(t(j + 1)),\eps)\} \ge 0$. 
Let us fix such $j$.

Let us define
\begin{eqnarray*}
A &=&
\{x \in \R^3: \, x_1^2 + x_2^2 \le M^2, x_3 \in [-M,M]\}\\
&& \quad \setminus \{x \in \R^3: \, \dist((x_1,x_2),D^c(t(j))) < 2h_1/3, x_3 \in (-\eta,\eta)\},\\
P_1 &=& \{ x \in Q_1(M): \, |x_3| \ge \eta\},\\ 
P_2 &=& \{x \in \R^3: \, \dist((x_1,x_2),D^c(t(j))) < 2h_1/3, x_1^2 + x_2^2 \le M^2, x_3 = -\eta \,\, \text{or} \,\, \eta\},\\
P_3 &=& \{x \in \R^3: \, (x_1,x_2) \in D(t(j)), \dist((x_1,x_2),\partial D(t(j))) = 2h_1/3, x_3 \in [-\eta,\eta]\}.
\end{eqnarray*}
Note that $\partial A = P_1 \cup P_2 \cup P_3$. By (\ref{smallh13}) (applied for $i = j$) for any $t \in [t(j),t(j+1)]$ we have $\Omega(M,D(t),\eps) \subset A$, $P_2 \cup P_3 \subset Q_2(M,D(t),\eps) \cup Q_3(M,D(t),\eps)$  so $\partial A \subset Q(M,D(t),\eps)$. By Corollary \ref{Qepsilon} for any $t \in [t(j),t(j+1)]$ and $x \in \partial A$ we have $H(v^{(\eps,t)})(x) \ge c$, where $c = c(\Lambda',M,\eps)$.

Now we will justify uniform Lipschitz property of $v_{ik}^{(\eps,t)}$. Note that $v^{(\eps,t)}$ are harmonic on $\R^3 \setminus (D^c(t) \times \{0\})$. Note also that for any $t \in [t(j),t(j+1)]$ $\dist(D^c(t) \times \{0\},A) \ge h_1/3 \wedge \eta$. This implies that for $t \in [t(j),t(j+1)]$ all second derivatives $v_{ik}^{(\eps,t)}$ are uniformly Lipschitz continuous on $A$. That is there exists $c > 0$ such that for any $t \in [t(j),t(j+1)]$, $x,y \in A$, $i,k \in \{1,2,3\}$ we have
\begin{equation}
\label{Lipcont}
\left|v_{ik}^{(\eps,t)}(x) - v_{ik}^{(\eps,t)}(y)\right| \le c |x - y|.
\end{equation}

Now we will show that if $[t(j),t(j+1)] \ni t \to s$ then for any $x \in A$
\begin{equation}
\label{Hessconv}
H(v^{(\eps,t)})(x) \to H(v^{(\eps,s)})(x).
\end{equation}
Denote $\vp^{(t)} = \vp^{(D(t))}$, $u^{(t)} = u^{(D(t))}$. By Lemma \ref{phiconv} for any $y \in D(s)$ if $t \to s$ then $\vp^{(t)}(y) \to \vp^{(s)}(y)$. If $x \in A$ and $x_3 > 0$ then we have 
$u_{ik}^{(t)}(x) = 
\int_{\R^2} K_{ik}(x_1-y_1,x_2-y_2,x_3) \vp^{(t)}(y_1,y_2) \, dy_1 \, dy_2$
which implies (\ref{Hessconv}). Using this and Lemma \ref{lowerhalfspace1} we get (\ref{Hessconv}) for $x \in A$ with $x_3 < 0$. (\ref{Hessconv}) for $x \in A$ with $x_3 = 0$ follows from (\ref{Lipcont}).

By the fact that $\Omega(M,D(t),\eps) \subset A$ for $t \in [t(j),t(j+1)]$ and our assumptions on $j$ we have
$\inf\{H(v^{(\eps,t(j))})(x): \, x \in A\} < 0$ and 
$\inf\{H(v^{(\eps,t(j+1))})(x): \, x \in A\} \ge 0$. Put 
$$
s = \inf\{t \in [t(j),t(j+1)]: \, \exists x \in A \,\,\,\, H(v^{(\eps,t)})(x)< 0\}.
$$
There exist a sequence $\{s(n)\}_{n = 1}^{\infty} \subset [t(j),t(j+1)]$ and $\{x(n)\}_{n = 1}^{\infty} \subset A$ such that 
$$
H(v^{(\eps,s(n))})(x(n))< 0
$$
and $s(n) \to s$. Since $A$ is compact we may assume that $x(n) \to z \in A$. By pointwise convergence and uniform Lipschitz continuity 
$H(v^{(\eps,s)})(z) = 0$. Since for any $x \in \partial A$ $H(v^{(\eps,s)})(x) > 0$ we have $z \in \text{int}A$. On the other hand, by pointwise convergence, we have $H(v^{(\eps,s)})(x) \ge 0$ for any $x \in A$. This gives contradiction with Theorem \ref{HL}.

\vskip 5pt
{\bf{Step 2.}}

By $\text{sign}(\text{Hess}(u(y)))$ we denote a signature of the Hessian matrix of $u(y)$. In this step we will show that for arbitrary $\Lambda = \{C_1,R_1,\kappa_1,\kappa_1\}$, $D \in F(\Lambda)$ and $y \in \R^3 \setminus (D^c \times \{0\})$ we have $\text{sign}(\text{Hess}(u(y))) = (1,2)$ and $\vp$ is strictly concave on $D$.

Fix $\Lambda = \{C_1,R_1,\kappa_1,\kappa_1\}$ where $C_1 > 0$, $R_1 > 0$, $\kappa_2 \ge \kappa_1 > 0$ and fix $D \in F(\Lambda)$. Let $\vp$ be the solution of (\ref{maineq1}-\ref{maineq2}) for $D$, $u$ the harmonic extension of $\vp$. Let $(x_1,x_2) \in D$, put $x = (x_1,x_2,0)$. Denote $f(x) = u_{11}(x) u_{22}(x) - u_{12}^2(x)$. By Lemma \ref{onD} $u_{13}(x) = u_{23}(x) = 0$, $u_{33}(x) > 0$. By Step 1 $H(u)(x) > 0$. Hence $f(x) > 0$. We have $u_{11}(x) + u_{22}(x) + u_{33}(x) = 0$ so $u_{11}(x) + u_{22}(x) < 0$. This and $f(x) > 0$ implies that $u_{11}(x) < 0$, $u_{22}(x) < 0$. Hence $\text{sign}(\text{Hess}(u(x))) = (1,2)$. Since $H(u)(y) > 0$ for any $y \in \R^3 \setminus (D^c \times \{0\})$ we get $\text{sign}(\text{Hess}(u(y))) = (1,2)$. 

Inequalities $f(x) > 0$, $u_{11}(x) < 0$, $u_{22}(x) < 0$ give that $\vp(x_1,x_2) = u(x_1,x_2,0)$ is strictly concave on $D$.

\vskip 5pt
{\bf{Step 3.}}

In this step we will show that for any open bounded convex set $D \subset \R^2$ $\vp$ is concave on $D$.

Fix an open bounded convex set $D \subset B(0,1) \subset \R^2$. It is well known (see e.g. \cite[page 451]{CF1985}) that there exists a sequence of sets $D_n$ such that $D_n \in F(\Lambda_n)$ for some $\Lambda_n = \{C_{1,n},R_{1,n},\kappa_{1,n},\kappa_{2,n}\}$ and $\bigcup_{n = 1}^{\infty} D_n = D$, $D_n \subset D_{n+1}$, $n \in \N$, $d(D_n,D) \to 0$ as $n \to \infty$ (where $C_{1,n} > 0$, $R_{1,n} > 0$, $\kappa_{2,n} \ge \kappa_{1,n} > 0$). Let $\vp^{(n)}$, $\vp$ denote solutions of (\ref{maineq1}-\ref{maineq2}) for $D_n$ and $D$. By Step 2 $\vp^{(n)}$ are concave on $D_n$. By Lemma \ref{phiconv} we have $\lim_{n \to \infty} \vp^{(n)}(x) = \vp(x)$ for $x \in D$. So $\vp$ is concave on $D$.

By scaling we may relax the assumption $D \subset B(0,1)$.
\end{proof}

\section{Extensions and conjectures}

\begin{proof}[proof of Theorem \ref{generalthm}]

a) It is well known that if $\psi_r(x) = \psi(rx)$, for some $r > 0$ and all $x \in \R^d$ then $(-\Delta)^{\alpha/2} \psi_r(x) = r^{\alpha} (-\Delta)^{\alpha/2} \psi(rx)$ (see e.g. \cite[page 9]{book2009}).
Fix $x_0 \in \partial D$ and $\lambda \in (0,1)$. Put $f(x) = \vp(\lambda x + (1 - \lambda) x_0) - \lambda^{\alpha} \vp(x)$. We have $(-\Delta)^{\alpha/2} f(x) = 0$ for $x \in D$ and $f(x) \ge 0$ for $x \in D^c$. Hence $f(x) \ge 0$ for $x \in D$.

b) Fix $x, y \in D$ and $\lambda \in (0,1)$. Put $z = \lambda x + (1-\lambda) y$. Let $l$ be the line which contains $x$ and $y$. Let $x_0 \in \partial D$ be the point on $l$ which is closer to $x$ than to $y$ and $y_0 \in \partial D$ be the point on $l$ which is closer to $y$ than to $x$. We have
$$
z = y \frac{|z - x_0|}{|y - x_0|} + x_0 \left(1 - \frac{|z - x_0|}{|y - x_0|}\right).
$$
By a) we get
$$
\vp(z) \ge \left(\frac{|z - x_0|}{|y - x_0|}\right)^{\alpha} \vp(y) 
\ge \left(\frac{|z - x|}{|y - x|}\right)^{\alpha} \vp(y) 
= (1 - \lambda)^{\alpha} \vp(y).
$$
We also have
$$
z = x \frac{|z - y_0|}{|x - y_0|} + y_0 \left(1 - \frac{|z - y_0|}{|x - y_0|}\right).
$$
Again by a) we get
$$
\vp(z) \ge \left(\frac{|z - y_0|}{|x - y_0|}\right)^{\alpha} \vp(x) 
\ge \left(\frac{|z - y|}{|x - y|}\right)^{\alpha} \vp(x) 
= \lambda^{\alpha} \vp(x).
$$
\end{proof}

Now we present some conjectures concerning solutions of (\ref{aeq1}-\ref{aeq2}).
\begin{conjecture}
\label{alpha1}
Let $\alpha = 1$, $d \ge 3$. If $D \subset \R^d$ is an arbitrary bounded convex set then the solution of (\ref{aeq1}-\ref{aeq2}) is concave on $D$.
\end{conjecture}
It seems that using the generalization of H. Lewy's result obtained by S. Gleason and T. Wolff \cite[Theorem 1]{GW1991} one can show this conjecture. Let $\alpha = 1$, $d \ge 3$ and $D \subset \R^d$ be a sufficiently smooth bounded convex set such that $\partial D$ has a strictly positive curvature, $\vp$ the solution of (\ref{aeq1}-\ref{aeq2}) and $u$ its harmonic extension in $\R^{d+1}$. It seems that using the method of continuity, in the similar way as in this paper, one can show that the Hessian matrix of $u$ has a constant signature $(1,d - 1)$. This implies concavity of $\vp$ on $D$. Anyway, Conjecture \ref{alpha1} remains an open challenging problem.

\begin{conjecture}
\label{alphageneral}
Let $d \ge 2$, $D \subset \R^d$ be an arbitrary bounded convex set and $\vp$ be the solution of (\ref{aeq1}-\ref{aeq2}).

a) If $\alpha \in (1,2)$ then $\vp$ is $1/\alpha$-concave on $D$.

b) If $\alpha \in (0,1)$ then $\vp$ is concave on $D$.
\end{conjecture}
\begin{remark}
\label{remarkgeneral}
For any $\alpha \in (1,2)$, $\eta \in (0,1 - 1/\alpha)$ and $d \ge 2$ there exists a bounded convex set $D \subset \R^d$ (a sufficiently narrow bounded cone) such that the solution of (\ref{aeq1}-\ref{aeq2}) is not $1/\alpha + \eta$ concave on $D$.
\end{remark}
\begin{proof}[Justification of Remarks \ref{Remark3} and \ref{remarkgeneral}]
It is clear that it is sufficient to show Remark \ref{remarkgeneral}. For any $\theta \in (0,\pi/2)$, $d \ge 2$ let
$$
D(\theta) = \{(x_1,\ldots,x_d): \, \sqrt{x_2^2 + \ldots + x_d^2} < x_1 \tan \theta, |x| < 1\}.
$$
Let $\alpha \in (0,2)$ and $\vp$ be the solution of (\ref{aeq1}-\ref{aeq2}) for $D(\theta)$. 

By \cite[Theorem 3.13, Lemma 3.7]{K1999} for any $\eps > 0$ there exists $\theta \in (0,\pi/2)$ and $c > 0$ such that 
\begin{equation}
\label{phieps}
\vp(x) \le c |x|^{\alpha - \eps}, \quad \quad x \in D(\theta).
\end{equation}
Theorem 3.13 and Lemma 3.7 in \cite{K1999} are formulated only for $d \ge 3$ but small modifications of proofs in \cite{K1999} give these results also for $d = 2$. (\ref{phieps}) for any $d \ge 2$ also follows from the recent paper \cite{BSS2014}.

Fix $d \ge 2$, $\alpha \in (1,2)$, $\eta \in (0,1-1/\alpha)$ and $\eps \in \left( 0,\frac{\alpha^2\eta}{1 + \eta \alpha}\right)$. There exists $\theta \in (0,\pi/2)$ and $c > 0$ such that the solution $\vp$ of (\ref{aeq1}-\ref{aeq2}) for $D(\theta)$ satisfies $\vp(x) \le c |x|^{\alpha - \eps}$. Fix $x_0 = (a,0,\ldots,0) \in D(\theta)$. If $\vp$ is $1/\alpha + \eta$ concave on $D(\theta)$ then for any $\lambda \in (0,1)$ we have 
$$
\vp(\lambda x_0) \ge \lambda^{\frac{\alpha}{1+\eta \alpha}} \vp(x_0) =
\lambda^{\alpha - \frac{\alpha^2\eta}{1 + \eta \alpha}} \vp(x_0).
$$
On the other hand $\vp(\lambda x_0) \le c \lambda^{\alpha - \eps} |x_0|^{\alpha - \eps}$, so 
$$
c \lambda^{\alpha - \eps} |x_0|^{\alpha - \eps} 
\ge \lambda^{\alpha - \frac{\alpha^2\eta}{1 + \eta \alpha}} \vp(x_0),
$$
which gives
$$
\lambda^{\frac{\alpha^2\eta}{1 + \eta \alpha} - \eps} 
\ge \vp(x_0) c^{-1} |x_0|^{\eps - \alpha}
$$
for any $\lambda \in (0,1)$, contradiction.
\end{proof}

We finish this section with an open problem concerning $p$-concavity of the first eigenfunction for the fractional Laplacian with Dirichlet boundary condition. 

Let $\alpha \in (0,2)$, $d \ge 1$, $D \subset \R^d$ be a bounded open set and let us consider the following Dirichlet eigenvalue problem for $(-\Delta)^{\alpha/2}$
\begin{eqnarray}
\label{eigen1}
(-\Delta)^{\alpha/2} \vp_n(x) &=& \lambda_n \vp_n(x), \quad \quad x \in D,\\ 
\label{eigen2}
\vp_n(x) &=& 0, \quad \quad \quad \quad \quad \, x \in D^c.
\end{eqnarray}
It is well known (see e.g. \cite{CS1997}, \cite{K1998}) that there exists a sequence of eigenvalues $0 < \lambda_1 < \lambda_2 \le \lambda_3 \le \ldots$, $\lambda_n \to \infty$ and corresponding eigenfunctions $\vp_n \in L^2(D)$. $\{\vp_n\}_{n=1}^{\infty}$ form an orthonormal basis in $L^2(D)$, all $\vp_n$ are continuous and bounded on $D$, one may assume that $\vp_1 > 0$ on $D$.

\vskip 2pt
{\bf{Open problem.}}  For any $\alpha \in (0,2)$, $d \ge 2$ find $p = p(d,\alpha) \in [-\infty,1]$ such that for arbitrary open bounded convex set $D \subset \R^d$ the first eigenfunction of (\ref{eigen1}-\ref{eigen2}) is $p$-concave on $D$. It is not clear whether such $p = p(d,\alpha) \in [-\infty,1]$ exists.

\vskip 2pt
Any results, even numerical, concerning this problem would be very interesting.

\vskip 5pt

{\bf{ Acknowledgements.}} 
I thank R. Ba{\~n}uelos for posing 
the problem of $p$-concavity of $E^x(\tau_D)$ for symmetric 
$\alpha$-stable processes. I also thank K.-A. Lee for interesting discussions. I am grateful for the hospitality of the Institute of Mathematics, Polish Academy of Sciences, the branch in Wroc{\l}aw, where a part of this paper was written.


\begin{thebibliography}{99}
\bibliographystyle{plain}

\bibitem{BB2013} R. Ba{\~n}uelos, R. D. DeBlassie, \emph{On the First Eigenfunction of the Symmetric Stable Process in a Bounded Lipschitz Domain}, arXiv:1310.7869 (2013).

\bibitem{BK2004} R. Ba{\~n}uelos, T. Kulczycki, \emph{The Cauchy process and the Steklov problem}, J. Funct. Anal. 211 (2004), 355-423.

\bibitem{BKM2006} R. Ba{\~n}uelos, T. Kulczycki, P. J. M{\'e}ndez-Hern{\'a}ndez, \emph{On the shape of the ground state eigenfunction for stable processes}, Potential Anal. 24 (2006), 205-221.

\bibitem{book2009} K. Bogdan, T. Byczkowski, T. Kulczycki, M. Ryznar, R. Song, Z. Vondra{\v{c}}ek, \emph{Potential analysis of stable processes and its extensions}, Lecture Notes in Mathematics 1980, Springer-Verlag, Berlin, (2009).

\bibitem{BKK2008} K. Bogdan, T. Kulczycki, M. Kwa{\'s}nicki, 
\emph{Estimates and structure of $\alpha$-harmonic functions}, Probab. Theory Related Fields 140 (2008), 345-381.

\bibitem{BKN2002} K. Bogdan, T. Kulczycki, A. Nowak, \emph{Gradient estimates for harmonic and $q$-harmonic functions of symmetric stable processes}, Illinois J. Math. 46 (2002), 541-556.

\bibitem{BSS2014} K. Bogdan, B. Siudeja, A. Stos, \emph{Martin kernel for fractional Laplacian in narrow cones}, arXiv:1403.6581 (2014).

\bibitem{B1985} Ch. Borell, \emph{Greenian potentials and concavity}, Math. Anal. 272, (1985), 155-160.

\bibitem{CF1985} L. A. Caffarelli, A. Friedman, \emph{Convexity of solutions of semilinear elliptic equations}, Duke Math. J. 52 (1985), 
431-456.

\bibitem{CS2007} L. A. Caffarelli, L. Silvestre, \emph{An extension problem related to the fractional Laplacian}, Comm. Partial Differential Equations 32 (2007), 1245-1260.

\bibitem{CS2010} Z.-Q. Chen, P. Kim, R. Song, \emph{Heat kernel estimates for the Dirichlet fractional Laplacian}, J. Eur. Math. Soc. (JEMS) 12 (2010), 1307-1329.

\bibitem{CS1998} Z.-Q. Chen, R. Song, \emph{Estimates on Green functions and Poisson kernels for symmetric stable processes}, Math. Ann. 312 (1998), 465-501.

\bibitem{CS1997} Z.-Q. Chen, R. Song, \emph{Intrinsic ultracontractivity and conditional gauge for symmetric stable processes}, J. Funct. Anal. 150 (1997), 204-239.

\bibitem{B1990} R. D. DeBlassie, \emph{The first exit time of a two-dimensional symmetric stable process from a wedge}, Ann. Probab. 18 (1990), 1034-1070.

\bibitem{ES1992} Yu. V. Egorov, M. A. Shubin, \emph{Linear partial differential equations. Foundations of the classical theory. Partial differential equations I}, Encyclopaedia Math. Sci. 30, Springer, Berlin (1992).

\bibitem{EIM2009} A. El Hajj, H. Ibrahim, R. Monneau, \emph{Dislocation dynamics: from microscopic models to macroscopic crystal plasticity}, Contin. Mech. Thermodyn. 21 (2009), 109-123.

\bibitem{E1959} J. Elliot, \emph{Absorbing barrier processes connected with the symmetric stable densities}, Illinois J. Math. vol. 3 (1959), 200-216.

\bibitem{G1961} R. K. Getoor, \emph{First passage times for symmetric stable processes in space}, Trans. Amer. Math. Soc. 101 (1961), 75-90.

\bibitem{GT1977} D. Gilbarg, N. S. Trudinger, \emph{Elliptic partial differential equations of second order}, Grundlehren der Mathematischen Wissenschaften, Vol. 224. Springer-Verlag, Berlin-New York, (1977).

\bibitem{GW1991} S. Gleason, T. H. Wolff, \emph{Lewy's harmonic gradient maps in higher dimensions}, Comm. Partial Differential Equations 16 (1991), 1925-1968.

\bibitem{KP1950} M. Kac, H. Pollard, \emph{Partial sums of independent random variables}, Canad. J. Mat. vol. 11 (1950), 375-384.

\bibitem{KS} M. Ka{\ss}mann, L. Silvestre, \emph{private communication}.

\bibitem{Ke1985} A. U. Kennington, \emph{Power concavity and boundary value problems}, Indiana Univ. Math. J. 34 (1985), 687-704.

\bibitem{Ke1984} A. U. Kennington, \emph{Power concavity of solutions of Dirichlet problems}, Miniconference on nonlinear analysis (Canberra, 1983), 133-136, Proc. Centre Math. Anal. Austral. Nat. Univ., 8, Austral. Nat. Univ., Canberra, (1984). 

\bibitem{K1983} N. Korevaar, \emph{Capillary surface convexity above convex domains}, Indiana Univ. Math. J. 32 (1983), 73-81.

\bibitem{KL1987} N. Korevaar, J. L. Lewis, \emph{Convex solutions of certain elliptic equations have constant rank Hessians}, Arch. Rational Mech. Anal. 97 (1987), 19-32.

\bibitem{K1998} T. Kulczycki, \emph{Intrinsic ultracontractivity for symmetric stable processes}, Bull. Polish Acad. Sci. Math. 46 (1998), 325-334.

\bibitem{K1997} T. Kulczycki \emph{Properties of Green function of symmetric stable processes}, Probab. Math. Statist. (1997), 339-364.

\bibitem{K1999} T. Kulczycki, \emph{Exit time and Green function of cone for symmetric stable processes}, Probab. Math. Statist. 19 (1999), 337-374.

\bibitem{KR2013} T. Kulczycki, M. Ryznar, \emph{Gradient estimates of harmonic functions and transition densities for Levy processes}, arXiv:1307.7158, to appear in Trans. Amer. Math. Soc.

\bibitem{L1968} H. Lewy, \emph{On the non-vanishing of the jacobian of a homeomorphism by harmonic gradients}, Ann. of Math. (2) 88, (1968), 518-529.

\bibitem{ML1971} L. G. Makar-Limanov, \emph{Solution of Dirichlet’s  problem for the equation $\Delta u = -1$ in a convex region}, Math. Notes Acad. Sci. USSR 9 (1971), 52-53.

\bibitem{M2002} P. J. M{\'e}ndez-Hern{\'a}ndez, \emph{Exit times from cones in $\R^n$ of symmetric stable processes}, Illinois J. Math. 46 (2002), 155-163.

\bibitem{MO1969} S. A. Molchanov, E. Ostrovskii,
\emph{Symmetric stable processes as traces of degenerate diffusion processes}, Theory Probab. Appl. 14 (1969), 128-131.

\bibitem{S1958} F. Spitzer, \emph{Some theorems concerning 2-dimensional Brownian motion}, Trans. Amer. Math. Soc. 87 (1958), 187-197.

\bibitem{Z2014} G. {\.Z}urek, \emph{Concavity of $\alpha$-harmonic functions} (in Polish), Master Thesis, Institute of Mathematics and Computer Science, Wroc{\l}aw University of Technology (2014).

\end{thebibliography}
\end{document}